\documentclass[11pt,reqno]{amsart}

\usepackage[dvipsnames]{xcolor}
\usepackage{amsmath}
\usepackage{amsthm}
\usepackage{graphicx}
\usepackage{upgreek}
\usepackage{hyperref}
\usepackage{mathrsfs}
\usepackage{bm,bbm}
\usepackage{amsfonts}
\usepackage{amssymb}
\usepackage{csquotes}
\usepackage{stmaryrd}
\usepackage{ulem}
\usepackage{tikz-cd}
\usepackage[toc]{appendix}
\usepackage{cleveref}
\usepackage{algpseudocode}
\usepackage{algorithm}

\newcommand{\sredm}[1]{\ifmmode\text{\xout{\ensuremath{\displaystyle \textcolor{red}{#1}}}}\else\sout{\textcolor{red}{#1}}\fi}

\newcount\Comments  
\definecolor{darkgreen}{rgb}{0,0.5,0}
\newcommand{\kibitz}[2]{\ifnum\Comments=1\textcolor{#1}{#2}\fi}

\usepackage{enumerate}

\usepackage{soul}

\usepackage[margin=0.8in]{geometry}
\numberwithin{equation}{section}
\theoremstyle{definition}

\newtheorem{theorem}{Theorem}[section]
\newtheorem{lemma}{Lemma}[section]

\newtheorem{proposition}{Proposition}[section]
\newtheorem{corollary}{Corollary}[section]

\newtheorem{definition}{Definition}[section]
\newtheorem{remark}{Remark}[section]

\newcommand{\noi}{\noindent}

\newcommand{\fra}{\mathfrak{a}}

\newcommand{\E}{\mathbb{E}}
\newcommand{\R}{\mathbb{R}}

\newcommand{\N}{\mathbb{N}}

\newcommand{\eps}{\varepsilon}
\newcommand{\lam}{\lambda}

\newcommand{\vr}{\varrho}
\newcommand{\al}{\alpha}

\newcommand{\Ups}{\mathnormal{\Upsilon}}

\newcommand{\EE}{{\mathbb E}}

\newcommand{\PP}{{\mathbb P}}

\newcommand{\calA}{{\mathcal A}}

\newcommand{\calC}{{\mathcal C}}

\newcommand{\calF}{{\mathcal F}}

\newcommand{\calK}{{\mathcal K}}
\newcommand{\calL}{{\mathcal L}}

\newcommand{\calN}{{\mathcal N}}
\newcommand{\calO}{{\mathcal O}}
\newcommand{\calP}{{\mathcal P}}
\newcommand{\calQ}{{\mathcal Q}}

\newcommand{\calS}{{\mathcal S}}

\newcommand{\calU}{{\mathcal U}}
\newcommand{\calV}{{\mathcal V}}
\newcommand{\calX}{{\mathcal X}}

\newcommand{\scrL}{\mathscr{L}}

\newcommand{\skp}{\vspace{\baselineskip}}

\newcommand\iy{\infty}

\newcommand{\A}{\mathbb{A}}

\newcommand{\1}{\boldsymbol{\mathbbm{1}}}

\newcommand{\bx}{\boldsymbol{x}}
\newcommand{\bX}{\boldsymbol{X}}
\newcommand{\by}{\boldsymbol{y}}

\newcommand{\bal}{\boldsymbol{\alpha}}

\newcommand{\bbeta}{\boldsymbol{\beta}}

\newcommand{\blue}{\textcolor{black}}

\DeclareMathOperator*{\argmin}{arg\,min}

\title{Asymptotic Nash Equilibria of Finite-State Ergodic Markovian Mean Field Games}
\author{ 
Asaf Cohen}\address{Department of Mathematics\\
University of Michigan\\
Ann Arbor, MI 48109\\
United States
}
\email{shloshim@gmail.com}
\author{Ethan Zell}
\address{Department of Mathematics\\
University of Michigan\\
Ann Arbor, MI 48109\\
United States
}
\email{ezell@umich.edu}
\thanks{* This is the final version of the paper. To appear in {\it Mathematics of Operations Research
}.}

\def\namedlabel#1#2{\begingroup
    #2%
    \def\@currentlabel{#2}%
    \phantomsection\label{#1}\endgroup
}

\allowdisplaybreaks 

\begin{document}
\maketitle

\begin{abstract}
Mean field games (MFGs) model equilibria in games with a continuum of weakly interacting players as limiting systems of symmetric $n$-player games. We consider the finite-state, infinite-horizon problem with ergodic cost. Assuming Markovian strategies, we first prove that any solution to the MFG system gives rise to a $(C/\sqrt{n})$-Nash equilibrium in the $n$-player game. We follow this result by proving the same is true for the strategy profile derived from the master equation. We conclude the main theoretical portion of the paper by establishing a large deviation principle for empirical measures associated with the asymptotic Nash equilibria. Then, we contrast the asymptotic Nash equilibria using an example. We solve the MFG system directly and numerically solve the ergodic master equation by adapting the deep Galerkin method of Sirignano and Spiliopoulos \cite{MR3874585}. We use these results to derive the strategies of the asymptotic Nash equilibria and compare them. Finally, we derive an explicit form for the rate functions in dimension two.
\end{abstract}
\skp
\noi{\bf Keywords:} Mean field games, master equation, ergodic games, Markov chains, large deviations.

\noi{\bf AMS Classification:} 
49N80, 
91A06, 
60J27, 
35Q89,  	
68T07,  	
49L12,  	
49N10. 	

 \tableofcontents

\section{Introduction}\label{sec:1}

\section{Introduction}
\vspace{-15mm}

\subsection{Focus of the paper}
We consider a finite-state, ergodic mean field game (MFG) to model a large interacting system of symmetric agents using Markovian strategies. In such games, a \textit{representative player} moves according to her jump process on a finite state-space $[d] := \{1,\dots, d\}$ while subject to an ergodic cost that depends on a flow of measures. The representative player chooses a control to minimize her cost, \blue{and the control she chooses is the transition rate between states}, and hence the distributional law of her state dynamics as well. Her distributional law, in turn, is \textit{another} flow of measures over the states $[d]$. When the representative player's law matches the original measure flow she responded to, we refer to the control and the flow as a \textit{mean field equilibrium} (MFE). The  representative player's ergodic cost is given by:
\begin{equation}\label{eqn:MFG_cost}
J_0(\al,\mu):=\limsup_{T\to\iy}\frac{1}{T}\E\Big[\int_0^T \left\{f(\calX_t^\al, \al(t, \calX_t^\al))+F(\calX_t^\al,\mu(t))\right\}dt\Big],
\end{equation} where $(\calX_t^\al)_{t\ge0}$ is the position of the representative player using her chosen control $\al$, \blue{$\al(t,x)=(\al_y(t,x))_{y\in[d]}$ stands for the transition rate vector from state $x$ to $y$}, $f+F$ is the running cost, 
and $(\mu(t))_{t\geq 0}$ is a flow of distributions over $[d]$. In these terms, we have an MFE when: (1) the distributional law of the player under control $\al$, $\calL(\calX_t^\al)$, equals $\mu(t)$ and (2) $\al$ minimizes $J_0(\cdot, \mu)$.

In a typical $n$-player game, by contrast, players react to the dynamics of others, and when no player has a cost incentive to change their strategy, we say the profile of all strategies constitutes a \textit{Nash equilibrium}. In a \textit{symmetric} Nash equilibrium, a given player's strategy will mirror that of any other player, rendering players indistinguishable from a population perspective. While explicitly solving for Nash equilibria in finite-player games can be computationally intractable, the relationship between MFGs and $n$-player games provides one avenue to approximate Nash equilibria for large, finite-player games. Specifically, when the number of players $n$ is large, we expect that the dynamics of the representative player in the MFG are close to the dynamics of any indistinguishable player from the $n$-player game. Denote the set of all players $[n] := \{1,\dots, n\}$. The cost to player $i \in [n]$ is:
\begin{align}\label{eqn:n_cost}
J^{n,i}_0(\bx,\bal):=\limsup_{T\to\infty} \frac{1}{T}\E \Big[\int_0^T \left\{f(X^{\bal, i}_t, \al^i(t, \bX_t^{\bal}))+F(X^{\bal, i}_t, \mu^{\bal, i}_t)\right\}dt\Big],
\end{align} 
where now $\bal := (\al^1, \dots, \al^n)$ is the collection of strategies for all players, $(\bX_t^{\bal} := (X^{\bal, 1}_t, \dots, X^{\bal, n}_t))_{t\ge0}$ are the players' positions, $(\mu^{\bal, i}_t))_{t\ge0}$ is the empirical distribution of all players $\bX^{\bal}$ without the $i$-th player, and $\bx$ is the initial state of the players, $\bX^{\bal}_0 = \bx \in [d]^n$. With respect to a Nash equilibrium, note that the $i$-th player minimizing $J_0^{n,i}$ can only select her own strategy $\al^i$.

One way to characterize the MFG and $n$-player relationship then, and a focus of this paper, is whether the controls from the MFE result in an approximate Nash equilibrium when applied in the finite-player game. That is, when all $n$ players are using the same strategies derived by the MFE, a single player who switches strategies receives, at most, a marginal benefit that vanishes as the number of players increases. Since the potential benefit of switching strategies decreases as the number of players increases, we refer to such strategy profiles as \textit{asymptotic Nash equilibria}. 

Typically, an ergodic MFE is characterized by a \textit{MFG system} that is, in general, a coupled system of differential equations. One equation is Kolmogorov's equation that, evolving forward in time, describes the distribution of the representative player under the MFE's control. The second differential equation is a Bellman equation that includes the {\it value of the MFG} (optimal cost) and a \textit{potential function}, that characterizes the state-wise deviations from the value. The {\it time-dependent ergodic MFG system} is written:
\begin{align}\label{erg_MFG_t}
    \begin{split}
        &-\frac{d}{dt} \check u(t,x) + \check \vr =  H(x,\Delta_x \check u(t,\cdot)) + F(x,\check \mu(t)), \\
        &\frac{d}{dt} \check \mu(t,x) = \sum_{y\in [d]}  \check \mu(t,y) \gamma^*_x(y,\Delta_y \check u (t,\cdot)), \\
        &\check \mu(0) = \eta, \quad t\in \R_+ := (0,\iy), \quad x\in [d],
    \end{split}
\end{align} 
where $\calP([d])$ is the set of probability distributions over $[d]$, 
$\Delta_x$ is the finite difference operator $\Delta_x b := (b_y -b_x)_{y\in [d]}$, \blue{$H:[d]\times\R^d\to\R$ is the Hamiltonian and $\gamma^*: [d]\times\R^d \to \A_{[d]}^d$, is its optimal selector, representing the rates to move from state $x$ to the other states, given by:
\begin{align}\notag
H(x,p):= \min_{a\in\A^d_{-x}}\Big\{f(x,a)+a\cdot p\Big\},\qquad \gamma^*(x,p):=\argmin_{a\in\A^d_{-x}}\Big\{f(x,a)+a\cdot p\Big\},
\end{align}
where $\A^d_{-x}$ and $\A_{[d]}^d$ represent all allowable rates leaving state $x$ and all allowable rates in general, respectively. Their precise definitions are provided in Section \ref{sec15}, where the notation is introduced.} 
The running cost component $f$ is part of the Hamiltonian $H$. A solution to \eqref{erg_MFG_t} is a triple $(\check \vr, \check u,\check \mu)$ where $\check \vr \in \R$ is the value, the potential function is $\check u : \R_+ \times [d] \to \R$, and the equilibrium measure flow is $\check \mu : \R_+ \to \calP([d])$. \textit{Any solution to the MFG system is linked to an MFE}: when the representative player uses the rates $[\gamma^*_y(x,\Delta_x \check u (t,\cdot))]_{x,y\in [d]}$ from the MFG system, then: (1) the controls and $\check \mu$ constitute an MFE, (2) the player pays the cost $\check \vr$, and (3) the representative player's law equals $\check \mu$.

Our ergodic systems consider a long-time average cost and so admit a stationary system whose optimal control, measure flow, and potential function are all independent of time. The \textit{stationary ergodic MFG system}, is:
\begin{align}\label{erg_MFG}
    \begin{split}
        &\bar \vr = H(x,\Delta_x \bar u) + F(x,\bar\mu), \\
        &0 = \sum_{y\in [d]} \bar \mu_y \gamma^*_y(x,\Delta_x \bar u),\quad x\in [d],
    \end{split}
\end{align} where any solution is a triple $(\bar\vr, \bar u, \bar\mu) \in \R\times\R^d\times \calP([d])$. When the initial data from \eqref{erg_MFG_t}, $\eta = \bar\mu$, we note that \eqref{erg_MFG} is a special case of \eqref{erg_MFG_t}. \blue{This is an algebraic equation, hence its solvability is generally more straightforward than that of \eqref{erg_MFG_t} or the master equation presented in \eqref{ME}.}

The MFG system is usefully characterized by a single PDE, the master equation. On one hand, the MFG system fixes the initial distribution of the players beforehand and then solves the system to find the MFE evolving from that distribution. More advantageously, the master equation solves for every initial distribution simultaneously. The master equation is particularly useful to study the convergence problem of (exact) Nash equilibria to the MFE. In this paper, we take the reversed direction, aiming to construct asymptotic Nash equilibria in the large population game from the MFE. For $(x,\eta) \in [d]\times \calP([d])$, the finite-state {\it ergodic master equation} was studied in \cite{CZ2022} and is given by:
\begin{equation}\label{ME}
    \vr = H(x,\Delta_x U_0 (\cdot, \eta)) + F(x,\eta) + \sum_{y,z \in[d]} \eta_y  D^\eta_{yz} U_0 (x,\eta) \gamma_z^*(y,\Delta_y U_0(\cdot, \eta)).
\end{equation} 
A solution to \eqref{ME} is a pair $(\vr, U_0)$ where $\vr\in\R$ is the {\it value of the MFG} and $U_0 : [d]\times \calP([d])\to\R$ is again referred to as the master equation's {\it potential function}. In Cohen, Zell, and Lauri\`ere \cite{CLZ2024}, the authors present two numerical schemes for solving this master equation.

Analogously to finite-horizon models, there are two natural ways to construct asymptotic Nash equilibria for the ergodic game: from the MFG system and from the master equation. By using controls from the MFG system, we mean players compute their optimal strategy through the potential function $\check u$ and the minimizer of the Hamiltonian, $\gamma^*$, and we expect the 
\begin{align}\label{eq:MFG_derived}
\text{{\it MFG system-derived strategy} \qquad $(\gamma^*_y(X^{n,i}_t,\Delta_{X^{n,i}_t} \check u(t))_{t\ge0}$\quad$i\in[n]$},
\end{align}
would be approximately optimal in the $n$-player game. Note that the potential function $\check u$ is coupled with the MFE flow $\check \mu$ through the MFG system \eqref{erg_MFG_t} and using the MFG system-derived strategy results in a player's law equal to the MFE flow. When instead considering the stationary ergodic MFG system \eqref{erg_MFG}, the corresponding strategy is $\gamma^*_y(x,\Delta_x \bar u).$

We also consider the 
\begin{align}\label{eq:ME_derived}
\begin{split}
\text{{\it master equation-derived strategy} \qquad
$(\gamma^*_y(X^{n,i}_t,\Delta_{X^{n,i}_t} U_0(\cdot, \mu^{n,i}_t)))_{t\ge0}$\quad$i\in[n]$}.
\end{split}
\end{align}
Here, $\mu^{n,i}_t$ denotes the current empirical distribution of the other players when all players are employing this profile. A property of the master equation assures us that replacing the empirical distribution in the above by the MFE flow just recovers the MFG system-derived strategy (namely, $\gamma^*_y(x,\Delta_x U_0(\cdot, \check\mu(t)))=\gamma^*_y(x,\Delta_x \check u(t))$) and so does not offer anything new. These two methods differ since, in the former, players respond to the deterministic flow of measures and their own states, hence can be thought of as {\it controls} while the latter lets players respond to the empirical distribution of the other players. In either case, players plug in a potential function different from the $n$-player game's potential function, which is harder to calculate. 
 
\blue{At first glance}, using the MFG system offers an advantage because \eqref{erg_MFG_t} and \eqref{erg_MFG} might have solutions even when \eqref{ME} might not. Moreover, the MFG system-derived strategy is simpler because players only observe a theoretical MFE flow and their own state rather than the realized empirical distribution. However, through a numerical example provided in \Cref{sec:numerics}, we demonstrate the advantages of using the master equation-derived strategies; specifically, they result in a stationary distribution that is more concentrated around the stationary MFE, the law of the empirical distribution has fewer deviations to extreme values, and the realized cost in the $n$-player game is lower, as opposed to the MFG system-derived strategies.

\subsection{Relevant work} 
The study of MFGs dates back to the pioneering work by Lasry and Lions \cite{LasryLions}, and Huang, Caines and Malham{\'e}, e.g., \cite{Huang2006,Huang2007}. Ergodic MFGs on continuous state-spaces consider diffusions instead of jump processes and trace back to the origins of MFGs by Lasry and Lions \cite{Lasry2006}. Later, Cardaliaguet and Porretta \cite{car-por} developed the theory of ergodic MFGs as limits of both finite-horizon MFGs and of infinite-horizon, discounted MFGs, on the torus, and treated the ergodic master equation. 
The first reference to the \textit{finite-state}, ergodic MFG comes from Gomes et. al. \cite{Gomes2013}. Later, Cohen and Zell \cite{CZ2022} studied the finite-state ergodic MFG without the contractive assumptions of \cite{Gomes2013}, and further proved the well-posedness of the finite-state, ergodic master equation. Other finite-state, ergodic MFG frameworks consider the discrete-time case \cite{GomesMohrSouza_discrete} and with finite action space by Neumann \cite{neumann2020stationary}.

In addition to investigating asymptotic Nash equilibria, we can directly compare the optimal cost functions of the MFG and the $n$-player game (under exact Nash equilibria). \blue{We categorize games into three types based on control methods: {\it open-loop controls}, where players only observe the noise components of others; {\it closed-loop controls}, where players can observe the entire dynamics of others (potentially including their full history); and {\it Markovian (closed-loop) controls}, where players observe only the current state of the dynamics of other players.}

For open-loop controls, Lacker \cite{Lacker2015general} and Fischer \cite{Fischer2017} established convergence of both asymptotic and exact Nash equilibria in the $n$-player finite-horizon game to the MFE in the corresponding MFG. The convergence in the continuous state-space, finite-horizon, with Markovian controls was established in the seminal work of Cardaliaguet, Delarue, Lasry, and Lions \cite{CardaliaguetDelarueLasryLions} using the master equation, and subsequently in the finite state-space by Cecchin and Pelino \cite{cec-pel2019} and Bayraktar and Cohen \cite{bay-coh2019}. Recently, Lacker \cite{Lacker2020, Lacker2023} showed that the convergence holds in the continuous state-space, finite horizon with general closed-loop controls (with and without common noise). The ergodic model in continuous state-space was studied \blue{by Feleqi \cite{fel2013}, Bardi and Priuli \cite{bar-pri2014},} and Arapostathis et al. \cite{Arapostathis_2017}, establishing the convergence of exact Nash equilibria to the MFE for open-loop controls. In sharp contrast to the finite-horizon work with closed-loop controls \cite{Lacker2020} and to the infinite-horizon with open-loop controls \cite{Arapostathis_2017}, Cardaliaguet and Rainer \cite{MR4083905} provided an example of multiple mean field limits in ergodic differential games with closed-loop controls, while the corresponding MFG has a unique MFE. The convergence problem in the ergodic Markovian setting is still an open question \blue{as we point out in Remark \ref{rem10}}.

We emphasize that we study the question of asymptotic Nash equilibria and not convergence of the optimal value functions, as discussed in the preceding paragraph. In the following section, we precisely outline this work's novel contributions.

\subsection{Contributions} 
\Cref{sec:25} provides the complete presentation of our results, and we describe them heuristically here. We introduce the MFG system and master equation-derived strategy profiles \eqref{eq:MFG_derived} and \eqref{eq:ME_derived}, 
and prove that they constitute asymptotic Markovian Nash equilibria in
\Cref{thm:epsilon_equilibrium} and \Cref{thm:epsilon_equilibrium_2}, respectively. Commensurate results for the stationary ergodic MFG system \eqref{erg_MFG} follow from \Cref{thm:epsilon_equilibrium} and are provided in \Cref{cor:stationary_eq}. In each asymptotic Nash equilibria result, we also assert that the associated empirical measures converge to the MFE flow as the number of players grows and refer to this as \textit{propagation of chaos}. \Cref{thm:large_deviations} presents corresponding large deviation results for the empirical measures. We begin by describing two lemmas that assist the proofs of both \Cref{thm:epsilon_equilibrium} and \Cref{thm:epsilon_equilibrium_2}, describe the differences between the proofs of each, and afterward speak to \Cref{thm:large_deviations}.

Since we assume all transition rates are bounded away from zero, and since the state-space $[d]$ is finite and fully-connected, any jump process on $[d]$ or $[d]^n$ must be exponentially ergodic. \Cref{lem:freedom} uses this fact to prove that, when such a jump process is plugged into a function under a long-time average, the process' initial distribution is irrelevant. In particular, \Cref{lem:freedom} removes the ergodic cost's dependence on its initial condition. The other intermediary, \Cref{lem:pre_duality}, estimates the difference between an empirical distribution of players under an arbitrary Markovian strategy profile and the measure flow $\check \mu$ from \eqref{erg_MFG_t}. The bound we prove depends on three things: the difference in the initial conditions, along with an exponential decay in time; the number of players, $n$; and, the difference between the arbitrary profile strategies and the MFG system-derived strategies \eqref{eq:MFG_derived}.

\Cref{sec:asymp_proofs} applies these lemmas to prove \Cref{thm:epsilon_equilibrium}. To prove \Cref{thm:epsilon_equilibrium} in the most generality we can, we do not require that the mean-field cost $F$ is Lasry--Lions monotone nor that the Hamiltonian $H$ is concave. We use that  $F$ is Lipschitz and \Cref{lem:freedom} to compare the costs \eqref{eqn:MFG_cost} and \eqref{eqn:n_cost}, and use that the strategy is optimal in the MFG to obtain approximate optimality in the $n$-player game. At one point, \Cref{lem:concentration_inequality_1}, a concentration inequality for Bernoulli variables helps to estimate the error. The propagation of chaos, in turn, follows from \Cref{lem:pre_duality}, simplified by the fact that the strategies selected for the $n$-player game are exactly the controls from the MFG.

\Cref{thm:epsilon_equilibrium_2} establishes that the master equation-derived strategies have propagation of chaos and constitute an asymptotic Nash equilibrium too. While in \Cref{thm:epsilon_equilibrium} players respond to a theoretical MFE flow and so do not respond to their own system, in \Cref{thm:epsilon_equilibrium_2}, players respond to their own empirical distribution. The main difficulty is the \textit{feedback problem} arising from the Markovian strategies \blue{as we now explain. Recall that with open-loop controls, players observe noises rather than dynamics. Therefore, if a player deviates from a predetermined Nash equilibrium, that deviation is not revealed to others. Its impact on the other players is only through the empirical distribution, which has an 
$O(1/n)$ effect. The situation becomes more complex with closed-loop controls, specifically, in the Markovian setting. When a player deviates from the Nash equilibrium, its dynamics change. Since the other players observe the dynamics of all players, including the one that deviated, their controls will also adjust. This may create a snowball effect, potentially leading to significant changes in the players' empirical distribution.}


One might already see the issue. If we compare empirical measures with and without a deviating player, the feedback problem ensures that any naive argument relying only on Lipschitz estimation becomes stuck in a loop. For the finite-horizon case, Gronwall's inequality can escape such loops \cite{Lacker2017}, but it is not applicable for the infinite-horizon. In fact, as mentioned in the introduction, \cite{MR4083905} showed that, for non-Markovian strategies, there may even be multiple limits to the exact Nash equilibria, while the corresponding MFG has only one MFE.

We address the feedback problem by combining \Cref{lem:pre_duality} with \Cref{lem:miracle}. Now, recall that in \Cref{lem:pre_duality}, one of the terms in the  upper bound for the difference between an empirical distribution of players under an
arbitrary Markovian strategy profile and the measure flow $\check\mu$, is an integral over the difference between the arbitrary profile strategies and the MFG system-derived strategies \eqref{eq:MFG_derived}.
In the case of the MFG system-derived strategies, this integral term in the upper bound vanishes. 
This term does not vanish for the master equation-derived strategy profile. 
In the MFG setting, a typical way to deal with an integral term like that one 
is to use the \textit{duality} of the MFG system. That is, integrating a potential function against its measure flow as a test function results in estimates that overcome the feedback problem. However, the MFG system is an ODE system, while what we need to estimates concerning a random, empirical measure. So for \Cref{lem:miracle}, we try to proceed like a standard duality argument, but encounter random terms that we will need to estimate or that vanish under expectation as martingales. Consequently, we accumulate $n$-dependent error that we drag through the proof. The part of the proof of \Cref{thm:epsilon_equilibrium_2} that concerns \Cref{lem:miracle} uses expectations taken under a stationary distribution. Hence, we deal with {\it stationary} Markovian strategies to start. Toward the end of the proof, we upgrade our analysis so it handles any Markovian strategy by \Cref{lem:stat_is_enough}. \Cref{lem:stat_is_enough} relies on an auxiliary ergodic system that we construct through (another) auxiliary discounted system and by sending the discount factor to zero. We then prove the optimal strategy for a deviating player is \textit{stationary} Markovian and use this fact to finish the proof of \Cref{thm:epsilon_equilibrium_2}.

Lastly, \Cref{thm:large_deviations} obtains large deviations results: (1) for the different empirical measures' laws at a fixed time, (2) for the laws of the measure's paths, and (3) for the invariant distributions. Like \Cref{thm:epsilon_equilibrium_2}, the proof of the large deviation principle requires that $F$ is monotone and $H$ concave, but this time much of the heavy lifting is done by \cite{MR3354770}. We only need to verify assumptions concerning ODE stability for the Kolmogorov equation from \eqref{erg_MFG_t} and the details are presented in \Cref{sec:ld_proofs}.

In the final part of the paper, Sections \ref{sec:numerics} and \ref{sec:ld_numerics}, we study a numerical example to elucidate differences between MFG system-derived and master equation-derived asymptotic Nash equilibria. 
We contrast the strategies, the stationary distributions of the empirical distributions, the large deviation rate functions, and the players' realized costs. As we discuss, the master equation-derived strategies result in a lower realized cost to the players than those of the MFG.

Computing the master equation-derived strategies requires computing $U_0$ and, since the master equation is a fully nonlinear PDE, it is not clear how to analytically do this, even for a concrete example. To approximate the solution to the master equation, we adapt the deep Galerkin method (DGM) of \cite{MR3874585} to the ergodic master equation in \Cref{alg:ergodic_DGM}. The DGM samples the state-space of the PDE, and using the PDE as part of the loss function of a neural network, trains the network to approximate the solution to the PDE. For the last leg of the numerics, we return to the large deviation principles established in \Cref{thm:large_deviations}. In the case the dimension $d=2$, we recognize a way to interpret the empirical measures as a birth-death process and introduce a function that we can use to prove an explicit formula for the good rate functions holds; we present this special case in an ancillary result, \Cref{thm:special_case}. 


\subsection{Outline of the paper} 

In \Cref{sec:2} we introduce the $n$-player ergodic game and the assumptions we consider. After, \Cref{sec:25} presents the main theorems of the paper. \Cref{sec:asymp_proofs} gives the proof for \Cref{thm:epsilon_equilibrium} through its several lemmas and \Cref{sec:thm2} does the same for \Cref{thm:epsilon_equilibrium_2}. Briefly, \Cref{sec:ld_proofs} gives the proofs for the large deviation result \Cref{thm:large_deviations}. In \Cref{sec:numerics}, we begin to numerically investigate the different aspects of the asymptotic Nash equilibria, where one profile is derived from the MFG system and one from the master equation. Finally, \Cref{sec:ld_numerics} concludes the paper with the numerics pertaining to large deviations. We conclude the current section by introducing some commonly used notation.

\subsection{Notation and basic definitions}\label{sec15}
We denote the ray $[0,\infty)$ by $\R_+$, the set of positive integers by $\N$, and the set of positive integers larger than one by $\N_{>1}$. \blue{For any Euclidean vector $v$, we denote its Euclidean norm as $|v|$.} Let $\vec{1}\in\R^d$ be the vector of all ones, so $k\vec{1}$ will be an additive constant vector appearing in the context of uniqueness of potential functions. We also denote the sets $[n]:=\{1,2,\dots,n\}$ and $[d]:=\{1,2,\dots,d\}$, where $n$ denotes a number of players and $d$ the number of states. Let $\1_A$ be the indicator function corresponding to the event $A$; that is, the function that is $1$ on $A$ and $0$ on its complement. The parameters $i,j$, and $k$ (resp., $x,y$, and $z$) represent a player (reps., state) and belong to $[n]$ (resp., $[d]$). For any random variable $Y$ with distributional law $\calL(Y)$, we may write $Y\sim \calL(Y)$. Fix two positive real numbers $0<\fra_l<\fra_u<+\iy$. For any $d\in\N_{>1}$ and $m\in\N_{>1}$, we set:
\begin{align}\notag
&\calP([d]):=\Big\{\eta\in\R^d_+:\sum_{x\in[d]}\eta_x = 1\Big\}, \quad \text{ a } (d-1)\text{-dimensional simplex,}
\\\notag
&\calP^m([d]):=\left\{\eta\in\calP([d]):m\eta\in(\N\cup\{0\})^d\right\}, \quad \text{valid configurations of } m \text{ players,} \\\notag
&\A_{-x}^d := \Big\{ a\in\R^d : \forall y\neq x,\quad a_y\in [\fra_l, \fra_u], \quad a_x = -\sum_{y,y\neq x} a_y\Big\}, \text{ allowable rates leaving state }x,\\\notag
&\A^d_{[d]} := \bigcup_{x\in [d]} \A_{-x}^d, \quad \text{ all allowable rates.}
\end{align} Moreover, let $\calQ$ be the set of $d\times d$ rate matrices; that is, $\alpha \in \calQ$ if for every distinct $x,y\in [d]$, $\alpha_{xy}\geq 0$ and $\sum_{z\in [d]} \alpha_{xz}=0$. When necessary, we denote the Lipschitz constant of a given function $h$ by $C_{L,h}$ and a bound by $C_h$. For any $x\in[d]$ and  $p\in\R^d$, set $\Delta_x p:=(p_y-p_x)_{y\in[d]}.$ For a function $K$ mapping $\R^d\ni p \mapsto K(p)\in\R$, we denote its gradient by $K_p$ or $D_p K$ and Hessian matrix by $K_{pp}$. We also use derivatives on the simplex as follows. When $K:\calP([d]) \to \R$ is a measurable function, we say that $K$ is differentiable if there exists a function $D^\eta_{yz} K : \calP([d])\to\R^{d\times d}$ such that 
\[
D^\eta_{yz} K (\eta) = \lim_{h\to 0^+} h^{-1}(K(\eta + h(e_{z} - e_y))-K(\eta))
\] for all $z\in [d]$, where $e_y$ is the $y$-th standard basis vector in $\R^d$. Denote the vector $D^\eta_y K := (D^\eta_{yz} K)_{z\in [d]}$ and the matrix $D^\eta K := (D^\eta_{yz} K)_{y,z \in [d]}$. The set of continuous functions with at least one continuous derivative, which are defined on $\R_+$ and with values in some Banach space $X$, is denoted $\calC^1(\R_+,X)$. Lastly, for $K: [d]\times \calP([d])\to\R$ we say $K$ is \textit{Lasry--Lions monotone} if for all $\eta,\hat \eta \in\calP([d])$,
\[
\sum_{x\in [d]} (K(x,\eta) - K(x,\hat \eta)) (\eta_x-\hat \eta_x) \geq 0.
\]

\section{The model}
\label{sec:2}

In Section \ref{sec:21}, we describe the $n$-player, finite-state ergodic game model. The assumptions about the models are listed in Section \ref{sec:22}. Next in Section \ref{sec:24}, we introduce the corresponding ergodic MFG via its dynamics. Then, we recall some results about the ergodic MFG systems \eqref{erg_MFG_t} and \eqref{erg_MFG} as well as their connection to the master equation \eqref{ME}.

\subsection{The $n$-player model}
\label{sec:21}
Throughout the paper, we consider a filtered probability space $(\Omega, \calF,\{\calF_t\}_{t\in\R_+}, \PP)$, satisfying the usual conditions, supporting Poisson random measures and jump processes as we detail below.

We consider a finite-state, continuous-time, symmetric $n$-player game with weak interactions. We begin by describing the $n$-player game's dynamics.  The state-space is the finite set $[d]$, where $d$ is an integer larger than one. Players choose their rates of transition from their own state to the other states and each player aims to minimize a personal cost. Both the transition rates and the cost depend on the players' current states.

Denote the players' positions via the vector $\bx := (x_1,\dots,x_n) \in [d]^n$. Players move at rates from the compact set $\A=[\fra_l,\fra_u]\subset\R_+$. A measurable function  
$\al:\R_+\times[d]^n\to \R^d$ satisfying $\al_{y}(t,\bx)\in\A$ for any $y\ne x_i$, and $\al_{x_i}(t,\bx)=-\sum_{y\in[d], y\ne x_i}\al_{y}(t,\bx)$ is called a {\it Markovian strategy} for player $i\in[n]$. Whenever $\al$ is independent of the time component we refer to it as a {\it stationary Markovian strategy}. For $y\ne x_i$, the magnitude $\al^i_y(t,\bx)$ represents the transition rate of player $i$ from its own state $x_i$ to state $y$, when the rest of the players are in states $\bx^{-i}:=(x_j : j\in [n]\setminus\{i\})$. Denote by $\calA^{n,i}$ (resp., $\calA^{n,i}_S$) the set of all Markovian strategies (resp. stationary Markovian strategies) for each one of the players. Then $\bal :=(\al^i:i\in[n])\in\calA^{(n)}:=\prod_{i=1}^n\calA^{n,i}$ (resp., $\bal=(\al^i:i\in[n])\in\calA^{(n)}_S:=\prod_{i=1}^n\calA^{n,i}_S$) is called a {\it Markovian strategy profile} (or respectively, {\it stationary Markovian strategy profile}).

For the following description of player dynamics, fix a number $n\in\N_{>1}$, an initial state for the players $\bx\in[d]^n$, and a Markovian strategy profile $\bal\in\calA^{(n)}$. Often, we will refer to the $n$-players' initial state as a {\it configuration} of the players. Given the fixed Markovian strategy profile $\bal$, the process standing for the dynamics of the $n$-players on $[d]^n$ is denoted by the bold vector $(\bX_t)_{t\geq 0}$. On occasions when we want to emphasize the dependence of $\bX$ on the Markovian strategy profile $\bal$, we denote it by $(\bX^{\bal}_t)_{t\ge 0}$. 
\blue{We rigorously utilize the representation of dynamics introduced by Cecchin and Fischer \cite{Cecchin2017}, where the dynamics of player 
$i$ are described by the stochastic differential equation:}
\begin{align*}\label{eqn:Xi_dyn}
X^{\bal,i}_t = X^{\bal,i}_0 + \int_0^t \int_{\A^d} \sum_{y\in [d]} (y-X^{\bal,i}_{s^-}) \1_{\{\xi_y\in(0,\al^i_y(s^-,\bX_{s^-}^{\bal}))\}} \calN_i(ds,d\xi),
\end{align*} where $\{\calN_i\}_{i\in [n]}$ are i.i.d.~Poisson random measures with common intensity measure $\nu$ given by:
\begin{equation}\label{n_intensity_measure}
\nu(E) := \sum_{y\in [d]}\text{Leb}(E\cap \A^d_y),
\end{equation} 
where $\A^d_y := \{u\in\A^d \mid u_x = 0 \text{ for all }x\neq y\}$ and $\text{Leb}$ is the Lebesgue measure on $\R$.

Recall that each player attempts to minimize the cost \eqref{eqn:n_cost} that is a function of the initial configuration $\bx \in[d]^n$ and the Markovian strategy profile $\bal$. 
We have that $f:[d]\times \A^d_{[d]}\to\R$ is measurable, $F:[d]\times \calP([d])\to\R$ is measurable, and 
$$\mu^{\bal,i}_{t}:=\frac{1}{n-1}\sum_{j\in[n], j\ne i}\delta_{X^{\bal, j}_t}.$$
 As a technicality, we also assume that $f(x,\al(t,\bx))$ should be independent of the $x$-th coordinate of $\al(t,\bx)$. This technical assumption is justified since the $x$-th coordinate of the function $\al(t,\bx)$ for a player in a state $x$ does not represent a rate, but rather the negative sum of the player's other rates. In that sense, $f$ already accounts for the ``payment" for the player's rates. With this game setup, we have a notion of $\eps$-Nash equilibrium: 

\begin{definition}
The Markovian strategy profile $\bal\in\calA^{(n)}$ is called an $\eps$\textit{-Nash equilibrium profile} for the $n$-player ergodic game if for any $i\in[n]$ and any $\beta^i\in\calA^{n,i}$, one has
\begin{align*}\notag
&J^{n,i}_0(\bx,\bal)\le J^{n,i}_0(\bx, [\bal^{-i};\beta^i])+\eps, 
\end{align*} 
where \[
[\bal^{-i};\beta^i] := (\al^1,\dots,\al^{i-1},\beta^i,\al^{i+1},\dots,\al^n).
\] For all $n\in\N_{>1}$, let $\bal^{(n)} \in \calA^{(n)}$  be a $\eps_n$-Nash Markovian equilibrium. We say that $(\bal^{(n)})_{n\in\N_{>1}}$ is an \textit{asymptotic Nash equilibrium} whenever $\eps_n\to 0$ as $n\to\iy$. 
\end{definition}

\begin{remark}
While we define an $\eps$-Nash equilibrium in terms of Markovian strategies and not stationary Markovian strategies, an $\eps$-Nash equilibrium can be a stationary Markovian strategy profile. This means that the deviation of one player, even to a time-dependent strategy, will not benefit that player by more than $\eps$. 
\end{remark}


\begin{definition}
    A Markovian strategy profile $\bal := (\al^i)_{i\in [n]}$ is called \textit{symmetric} if for all players $i,j\in [n]$, whenever $x^i=x^j$, for all $t\in\R_+$, and $\bx\in[d]^n$:
    \[
    \al^i(t,\bx) = \al^j(t,\bx).
    \] For a stationary Markovian strategy profile to be called symmetric, the definition is the same, except there is no time component of $\al^i, \al^j$.
\end{definition}

\subsection{Assumptions}\label{sec:22} We require a few assumptions that we split into two sets. The first set of assumptions is labeled with an $A$. We refer to this set of assumptions as the \textit{reduced assumptions}. The reduced assumptions are sufficient to establish the main result concerning asymptotic Nash equilibria, \Cref{thm:epsilon_equilibrium}. For \Cref{thm:epsilon_equilibrium_2} and \Cref{thm:large_deviations}, we require additional regularity of the functions comprising the problem data. These assumptions are denoted with a $B$. The set of all assumptions (both $A$ and $B$ indexed) are referred to as \textit{full assumptions}.

\begin{itemize}
    \item[($A_1$)] For $0<\mathfrak{a}_l < \mathfrak{a}_u<\iy$, the action space is  $\A:=[\mathfrak{a}_l,\mathfrak{a}_u]$.
\end{itemize} 

\begin{itemize}
    \item[($A_2$)] We assume that the Hamiltonian $H$ admits a unique optimal rate selector given by the measurable function $\gamma^*$. Moreover, $\gamma^*$ is Lipschitz in its $p$ argument. 
\end{itemize} The latter assumptions hold, for instance, when the running cost component $f$ is strictly convex in the control/strategy $\al$. In fact, the optimal rate selector $\gamma^*$ is globally Lipschitz when $f$ is uniformly convex. 

\begin{itemize}
    \item[($A_3$)] Given any state $x\in[d]$, the mean field cost $F(x,\cdot)\in\calC^1(\calP([d]))$ is Lipschitz continuous. 
\end{itemize} Consequently, $F(x,\cdot)$ is bounded since it is defined on a compact set.
 Lastly, we assume: 

\begin{itemize}
    \item[($B_1$)] 
    Additionally, the mean field cost $F$ is Lasry--Lions monotone and $D^\eta F(x,\cdot)$ is Lipschitz.
\end{itemize}


\begin{itemize}
    \item[($B_2$)] The Hamiltonian $H$ is twice continuously differentiable with respect to its $p$-argument. Moreover, on any compact set $[-K,K]$ we will assume that $D_pH$ and $D^2_{pp}H$ are Lipschitz in $p$ and that there exists $C_{2,H}>0$ such that:
\[
D^2_{pp} H(x,p) \leq -C_{2,H}.
\]
\end{itemize} Because any argument of $H$ in the paper will be uniformly bounded, it is enough to assume these regularity properties of $H$ only on compact sets. At any point when $H$ is differentiable, we have by \cite[Proposition 1]{Gomes2013} that: \begin{align}
    \label{gamma_H}
\gamma^*(x,p) = D_p H(x,p).
\end{align}

\begin{remark}
    The full assumptions are considered standard in MFG literature---for instance, \cite{CZ2022, bay-coh2019, cec-pel2019} all consider finite-state MFGs with a concave (or if they maximize a profit instead of minimizing a cost, convex) Hamiltonian and a separable cost $f+F$, with Lasry--Lions monotone assumptions on $F$. These assumptions are typical for studying the well-posedness of the MFG and master equation over continuous state-spaces as well \cite{CardaliaguetDelarueLasryLions, car-por, ben-fre-yam2015}. However, some papers consider alternate assumptions to deal with Hamiltonians of a different form. A few examples of alternate schemes include: \cite{gan-mes-mou-zha2022} considers a displacement monotonicity condition allowing for analysis of non-separable Hamiltonians, \cite{mou-zha2019} considers non-smooth Hamiltonians under technical assumptions, and \cite{2022arXiv220110762M} considers an anti-monotonicity condition. Each approach is able to handle data that the other approaches cannot. By considering the reduced assumptions when appropriate, we prove \Cref{thm:epsilon_equilibrium} in the greatest generality so that it might be widely applicable.
\end{remark}

\subsection{The ergodic MFG, the MFG systems, and the master equation}\label{sec:24}



The finite-state stationary ergodic MFG system \eqref{erg_MFG} was first studied in \cite{Gomes2013} under additional contractivity assumptions. Later, \cite{CZ2022} studied the same MFG system without any contractivity assumptions and we recall the major points here.

In an ergodic MFG, the representative player chooses a measurable control function $\al : \R_+\times [d] \to \A^d_{[d]}$ while playing against a measure flow $\mu : \R_+ \to\calP([d])$ in order to minimize the long-run average cost \eqref{eqn:MFG_cost}.

The representative player's position evolves according to a controlled jump process, given by:
\[
\calX_t^\al = \calX_0^\al + \int_0^t \int_{\A^d} \sum_{y\in [d]} (y-\calX_{s^-}^\al)\boldsymbol{1}_{\{\xi_y\in(0,\al_y(s^-,\calX_{s^-}^\al))\}} \calN(ds,d\xi),
\] where $\calX_0^\al$ has the distributional law $\mu(0)$ and $\calN$ is a Poisson random measure with intensity measure $\nu$ as in \eqref{n_intensity_measure}. For the next definition, recall the ergodic cost function in the MFG, $J_0$, provided in \eqref{eqn:MFG_cost}.

\begin{definition}[Ergodic MFE]\label{def:erg_MFE}
    Let $\al \in\calA$ and $\mu :\R_+ \to\calP([d])$. Let $\calX^\al$ be the dynamics of the representative player using $\al$. We say that $(\al, \mu)$ is an ergodic MFE if the following hold:
    \begin{itemize}
        \item For all $t\in\R_+$, $\calL(\calX^{\al}_t) = \mu(t)$,
        \item and the control $\al$ satisfies:
        \[
        J_0(\al, \mu) = \inf_{\beta\in\calA} J_0(\beta, \mu).
        \]
    \end{itemize}
\end{definition}

Typically in ergodic problems, there is an optimal cost or \textit{value} that is independent of the initial state. Moreover, the long-time averaging in the cost often yields a stationary Markovian optimal control and hence a stationary measure that describes a limiting law for the representative player. We saw this stationarity in the system \eqref{erg_MFG} from before while the general system was given in \eqref{erg_MFG_t}. For any solution $(\bar \vr, \bar u, \bar \mu)$ to \eqref{erg_MFG}, and there could be more than one solution, note that the potential function $\bar u$ is modifiable up to addition of a constant vector---this comes from the fact that $\bar u$ only appears in \eqref{erg_MFG} as the finite difference vector $\Delta_x \bar u$.

The result of \cite[Proposition~2.1]{CZ2022}  proved that the stationary ergodic MFG system \eqref{erg_MFG} is well-posed and that any solution to it naturally characterizes a mean field equilibrium, while any mean field equilbrium naturally characterizes a solution. The first main point is recalled here:

\begin{proposition}\label{prop:erg_MFG_red_full}
    With reduced assumptions, \eqref{erg_MFG} admits a solution, and, with full assumptions, the solution to \eqref{erg_MFG} is unique. We emphasize that for a unique solution $(\bar\vr, \bar u, \bar\mu) $ to \eqref{erg_MFG}, the potential function $\bar u$ is only unique up to an additive constant vector $k\vec{1}$, $k\in\R$. Moreover, under the reduced assumptions, $((\gamma_y(x, \Delta_x \bar u))_{x,y\in [d]}, \bar\mu)$ is an ergodic MFE and so for all $t>0$:
    \[
    \calX^{(\gamma_y(x, \Delta_x \bar u))_{x,y\in [d]}}_t \sim \bar\mu, \quad \text{ with } \quad \bar\vr = J_0((\gamma_y(x, \Delta_x \bar u))_{x,y\in [d]}, \bar\mu)
    \] when $\calX^{(\gamma_y(x, \Delta_x \bar u))_{x,y\in [d]}}_0 \sim \bar\mu$.
\end{proposition}

The proof of existence under reduced assumptions is not explicitly provided in \cite{CZ2022}, but a look at the proofs for existence allows us to recognize that only the reduced assumptions are used there. For the time-dependent system, \cite{soner} studied an example that fits under the reduced assumptions, where a solution to \eqref{erg_MFG_t} exists only when $\eta$ is in a neighborhood of $\bar\mu$. Moreover, \eqref{erg_MFG_t} and the master equation \eqref{ME} are only known to have a solution (for all initial data) under the full assumptions; in that case, the solutions to \eqref{erg_MFG_t}, \eqref{erg_MFG}, and \eqref{ME} are also unique. The main points of \cite[Theorem~2.1]{CZ2022} and the relationship between \eqref{erg_MFG_t} and \eqref{ME} are recalled below.  

\begin{proposition}\label{prop:me_properties}
    Under the full assumptions, \eqref{ME} has a unique solution $(\vr, U_0)$, where $U_0$ is unique up to an additive constant vector $k\vec{1}$, $k\in\R$. Similarly, the solution $(\check \vr, \check u, \check \mu)$ to \eqref{erg_MFG_t} is unique. The potential function $U_0$ is Lipschitz and Lasry--Lions monotone with a Lipschitz derivative $D^\eta U_0$. The relationship between the MFG system and master equation is given by:
    \[
    U_0(x, \check \mu(t)) = \check u_x(t), \text{ and } \vr = \check \vr.
    \] Moreover, for any initial condition $\check\mu(0)=\eta\in\calP([d])$, $((\gamma_y(x, \Delta_x \check u(\cdot)))_{x,y\in [d]}, \check \mu)$ is an ergodic MFE and there exist $C,c>0$ independent of $\eta$ such that for all $t>0$,
    \[
    \calX^{(\gamma^*_y(x,\Delta_x \check u(t)))_{x,y\in [d]}}_t \sim \check \mu(t), \; \text{ with } \; \check \vr = J_0((\gamma_y(x, \Delta_x \check u(t)))_{x,y\in [d]}, \check \mu), \; \text{ and }\; |\check \mu(t) - \bar\mu| \leq Ce^{-ct},
    \] when $\calX^{(\gamma^*_y(x,\Delta_x \check u(t)))_{x,y\in [d]}}_0 \sim \check \mu(0).$
\end{proposition}

We now make some final remarks on the time-dependent ergodic MFG system \eqref{erg_MFG_t}.

\begin{remark}\label{rmk:U_0_soln}
Under the full assumptions, there is a way to construct the solution for \eqref{erg_MFG_t} for any initial condition $\eta\in \calP([d])$. Let $(\vr, U_0)$ satisfy \eqref{ME}, and define $\mu_0:\R_+\to\calP([d])$ as the solution to:
\begin{align}\label{erg_FP_t}
    \frac{d}{dt} \mu_0(t,x) = \sum_{y\in [d]} \mu_0(t,y) \gamma^*_x(y,\Delta_y U_0 (\cdot, \mu_0(t))), \quad \mu_0(0) = \eta \in\calP([d]),
\end{align} 
where we abuse notation and denote the $x$-th coordinate of $\mu_0(t)$ by $\mu_0(t,x)$, hence referring to $\mu$ as a mapping from $\R_+\times[d]\to[0,1]$. Note that under the full assumptions, $\check \vr = \bar\vr = \vr$. Note that \eqref{erg_FP_t} is a forward Kolmogorov equation associated with a Markov jump process $\calX^0$, where $\calX^0$ jumps with rates $[\gamma^*_y(x,\Delta_x U_0(\cdot,\mu_0(t)))]_{x,y\in [d]}$. Next, set $u_0(t,x):= U_0(x,\mu_0(t))$. By a short computation, we can verify that $(\vr, u_0,\mu_0)$ satisfies \eqref{erg_MFG_t}. Of course, for $\eta = \bar\mu$, the solution would be the stationary one $(\bar \vr, \bar u , \bar \mu)$ to \eqref{erg_MFG}. 
\end{remark}

\begin{remark}\label{rmk:station_soln}
    At the same time, we note that any stationary solution $(\bar \vr, \bar u, \bar \mu)$ solving \eqref{erg_MFG} solves \eqref{erg_MFG_t} with $\check \mu(0) = \bar \mu$. Therefore \Cref{prop:erg_MFG_red_full} implies that, even under the reduced assumptions, there is a solution to \eqref{erg_MFG_t} for at least one initial condition. 
\end{remark}

\begin{remark}\label{rmk:check_u_bounded}
    All the potential functions, $\check u$, $\bar u$, and $U_0$, are only unique up to addition of a constant vector, $k\vec{1}$, $k\in\R$. Also, note that for \eqref{erg_MFG_t} $\check u$ can be unbounded as $t$ changes. To avoid this we will say that $(\check \vr, \check u, \check \mu)$ is a solution to \eqref{erg_MFG_t} only when $\check u$ is bounded. For both examples given in \Cref{rmk:U_0_soln} and \Cref{rmk:station_soln}, $\check u$ is bounded because $U_0$ is bounded.
\end{remark}

\section{Asymptotic Results}\label{sec:25}

In this section we provide the results for the three kinds of asymptotic Nash equilibria derived from: the time-dependent MFG \eqref{erg_MFG_t}, the stationary MFG system \eqref{erg_MFG}, and the master equation \eqref{ME}. The following subsection introduces the large deviation results for the empirical distributions of each of these asymptotic equilibria, formulated generally.

\subsection{Asymptotic Markovian Nash-equilibria}

Before presenting the asymptotic Markovian Nash equilibria, we introduce a lemma that eliminates the dependence on the initialization for long-time averaging under expectation, particularly for $J_0^{n,i}$. For the rest of the paper, we write $\E_{\pi}[g(\bX_t)]$ to denote the conditional expectation $\E[g(\bX_t) \mid \bX_0 \sim \pi]$, for $g$ a measurable function. For example, $g(\bX_t)$ might be the empirical distribution of $\bX_t$.

\begin{lemma}\label{lem:freedom}
    Fix $n\in\N_{>1}$. Let $(\beta^i)_{i\in[n]} = \bbeta \in \calA^{(n)}$, and let $g:[d]^n \to \R$ be measurable and bounded. If $(\bX^{\bbeta}_t)_{t\ge0}$ is the jump process of players using the profile $\bbeta$, then for all $\pi_0, \pi_1 \in \calP([d]^n)$:
    \begin{align*}
        \limsup_{T\to\iy} \frac{1}{T} \E_{\pi_0} \int_0^T g(\bX^{\bbeta}_t) dt = \limsup_{T\to\iy} \frac{1}{T}\E_{\pi_1} \int_0^T g(\bX^{\bbeta}_t) dt.
    \end{align*}
\end{lemma}
The proof of this lemma is standard and follows typical ergodicity techniques. For the sake of completeness, its proof is provided in the Appendix.

\begin{remark}
    By \Cref{lem:freedom}, we obtain that the $n$-player ergodic cost $J_0^{n,i}$ is independent of the initial configuration of players. So for a Markovian strategy profile $\bbeta \in \calA^{(n)}$ and configuration $\bx\in [d]^n$, we may write $J_0^{n,i}(\bx, \bbeta)$ as $J_0^{n,i} (\bbeta)$ without confusion. When $\bbeta$ is symmetric, the costs for all the players are the same and we may further simplify the notation and write the latter as $J_0^{n} (\bbeta)$.
\end{remark}

For any Markovian strategy profile $\bbeta:=(\beta^i)_{i\in[n]} \in \calA^{(n)}$, denote by $(\bX^{\bbeta}_t)_{t\ge0}$ the jump process of all players using $\bbeta$ and $(X^{\bbeta, i}_t)_{t\ge0}$ 
the position of the $i$-th player. Denote $(\mu^{\bbeta}_t)_{t\ge0}$ as the empirical measure of $(\bX^{\bbeta}_t)_{t\ge0}$ and, after removing the $i$-th player as $(\mu^{\bbeta, i}_t)_{t\ge0}$. The empirical distribution $\mu^{\bbeta}$ has the dynamics:
\begin{align}\label{eq:mu_beta}
\mu_t^{\bbeta} = \mu_0^{\bbeta} + \int_0^t \sum_{x,y\in [d]} \frac{e_y - e_x}{n} dK_{xy}^{\bbeta} (s),
\end{align} 
where:
\begin{align}\label{eq:K}
K_{xy}^{\bbeta} (t) = \int_0^t \int_{\A^{d\times d}} \1_{\Big\{\Big[0, \; \sum_{i\in [n]}\1_{\{X^{\bbeta, i}_{s^-} = x\}} \beta^i_{xy}(s^-, \bX^{\bbeta}_{s^-})\Big)\Big\}} \calK_{xy}^{\bbeta}(ds, dr),
\end{align} 
and $\calK^{\bbeta}_{xy}$ is a Poisson random measure with jump intensity given by:
\[
\tilde \nu (A) := \sum_{x,y \in [d]} \text{Leb} (A\cap \A_{xy}), \quad \A_{xy} := \{ u\in \A^{d\times d} \mid u_{wz} = 0\quad \forall (w,z) \neq (x,y)\}.
\] We note that the compensator for this random measure is:
\[
\int_0^t\sum_{i\in [n]}\1_{\{X^{\bbeta, i}_s = x\}} \beta^i_{xy}(s, \bX^{\bbeta}_{s})ds.
\]

Now with the requisite background, we introduce the main theoretical results of the paper. The first, \Cref{thm:epsilon_equilibrium}, asserts that any solution to the time-dependent ergodic MFG system gives an approximate Nash equilibrium when applied in the $n$-player game; moreover, for the empirical measure associated with this profile, we have the propagation of chaos. \blue{The theorem addresses the system $(\check \vr, \check u, \check \mu)$, but it does not necessitate the full set of assumptions; instead, it relies on a reduced set that does not guarantee the existence of a solution to the system. Therefore, we make this assumption explicit.}


\begin{theorem}\label{thm:epsilon_equilibrium}
     Assume the reduced assumptions hold. Let $\eta \in\calP([d])$ and assume \eqref{erg_MFG_t} has a solution $(\check \vr, \check u, \check \mu)$ with initial condition $\check\mu(0) = \eta$, as with the examples given in Remarks \ref{rmk:U_0_soln} and \ref{rmk:station_soln}. Then there exists $C>0$, independent of $n$ and $\eta$, such that the Markovian strategy profile:
    \begin{align*}
        \check{\Gamma} := \big(\check \gamma \big)_{i\in [n]},  \quad \check \gamma(t) := [\gamma^*_y(x, \Delta_x \check u(t))]_{x,y\in [d]},
    \end{align*} is a Markovian $(C/\sqrt{n})$-Nash equilibrium. Moreover, the cost for each player in the game, $J_0^{n} (\check \Gamma)$, satisfies
    \begin{align}\label{eqn:cost_convergence}
   | J_0^{n} (\check \Gamma) - \check \vr |\le \frac{C}{\sqrt{n}}.
    \end{align} Also, there exists $c>0$ such that for any $\pi\in\calP([d]^n)$ and any $n\in\N_{>1}$,
    \begin{align}\label{eqn:pi_thm1}
        \E_{\pi} |\mu^{\check \Gamma}_t - \check \mu(t)| \leq Ce^{-ct} \E_\pi |\mu^{\check \Gamma}_0 - \check \mu(0)| + \frac{C}{\sqrt{n}}.
    \end{align} Suppose $(X^{\check \Gamma, i}_0)_{i\in [n]}$ are exchangeable, though not necessarily independent, with $X^{\check \Gamma, i}_0 \sim \check \mu (0) = \eta$ for all $i\in [n]$, \blue{where $\check\mu(0)$ is not on the boundary of $\calP([d])$}. For $\pi = \check \pi$, the joint distribution induced by this initialization, we have from \eqref{eqn:pi_thm1} that: \begin{align}\label{eq:propagation_check}
    \sup_{t\in\R_+} \E_{\check \pi} |\mu^{\check \Gamma}_t - \check \mu(t)| \leq \frac{C}{\sqrt{n}},
    \end{align}  
    
\end{theorem} 

\begin{remark}
    The assumption that $\check \mu(0)$ is not on the boundary of $\calP([d])$ is so that we can assume there exists some constant $q>0$ for which $\check \mu_x(t) > q$ for all $x\in [d]$ and for $t\in\R_+$. When $\check \mu(0)$ is on the boundary, we can replace the final display with:
    \[
    \sup_{t\geq 1} \E_{\check \pi} |\mu_t^{\check \Gamma} - \check \mu(t)| \leq \frac{C}{\sqrt{n}},
    \] though the choice of $1$ is arbitrary---any positive constant would do. This is because we must allow a moment for the measure to enter the interior of the simplex. Since the transition rates between any states are bounded away from zero, the deterministic process $\check \mu$ will always move away from the boundary. Throughout the remainder of the paper, this technical point will remain implicit in any statement of a supremum over all $t\in\R_+$, except for \Cref{lem:concentration_inequality_1}, where we note a few details.
\end{remark}

\begin{remark}
For the Markovian strategy profile $\check \Gamma$, each player plays according to her own state and a deterministic flow of measures. We emphasize that we allow the deviating player to use a strategy based on the current time as well as the knowledge of the current states of all the players. As we will show, this extra information does not help the deviating player that much, since they can only improve their cost by at most $C/\sqrt{n}$.
\end{remark}

\begin{remark}
    Recent work by H\"ofer and Soner \cite{soner} studies an example of non-unique solutions to \eqref{erg_MFG_t} when $d=2$ under a specified anti-monotone mean field running cost $F$, and quadratic running cost $f$. They show that, for certain choices of constants in the rates $\A$ and by scaling $F$, there are infinitely-many distinct MFE. Fortunately, \Cref{thm:epsilon_equilibrium} applies to these examples as well. Since we only require the reduced assumptions, we can accommodate these examples and prove that the infinitely-many equilibria studied in \cite{soner} constitute $(C/\sqrt{n})$-Nash equilibria as well.
\end{remark}

With Theorem \ref{thm:epsilon_equilibrium} in hand, which provides an asymptotic Markovian Nash equilibrium based on the time-dependent ergodic MFG system \eqref{erg_MFG_t}, we recall that the stationary ergodic MFG system \eqref{erg_MFG} is a special case of \eqref{erg_MFG_t} with the initial state $\eta=\bar\mu$. Consequently, strategies derived from the stationary system also yield an asymptotic Markovian Nash equilibrium. Namely, let $ \mu^{\bar\Gamma}_t$ be the empirical distribution of $n$ players using the time-stationary Markovian strategy profile:
\begin{align}\label{eqn:stationary_control}
    \bar\Gamma := (\bar\gamma)_{i\in [n]} \quad \text{ where } \quad \bar\gamma := [\gamma^*_y(x, \Delta_x \bar u)]_{x,y\in [d]}.
\end{align} 
So, when $\check \Gamma = \bar\Gamma$, we can apply \Cref{thm:epsilon_equilibrium} to get the following corollary.

\begin{corollary}\label{cor:stationary_eq}
    Assume the reduced assumptions hold. Then, the stationary ergodic MFG system \eqref{erg_MFG} has a solution $(\bar\vr, \bar u, \bar \mu)$. Additionally, there exists $C>0$, independent of $n$, such that the Markovian strategy profile $\bar\Gamma$ is a Markovian $(C/\sqrt{n})$-Nash equilibrium.  Moreover, the cost for each player in the game, $J_0^{n} (\bar \Gamma)$, satisfies
    \begin{align}\label{eqn:cost_convergence_cor}
    |J_0^{n} (\bar \Gamma) - \bar \vr| \leq \frac{C}{\sqrt{n}}.
    \end{align} Also, for any $\pi \in\calP([d]^n)$,
    \begin{align}\label{eqn:pi_cor}
        \E_{\pi} |\mu^{\bar \Gamma}_t - \bar \mu| \leq Ce^{-ct}\E_\pi |\mu_0^{\bar \Gamma} - \bar \mu | +\frac{C}{\sqrt{n}}.
    \end{align} Suppose $(X^{\bar\Gamma, i}_0)_{i\in [n]}$ are exchangeable, though not necessarily independent, with $X^{\bar\Gamma, i}_0 \sim \bar\mu$ for all $i\in [n]$. For $\pi = \bar \pi$, the joint distribution induced by this initialization, we have from \eqref{eqn:pi_cor} that:
    \begin{align*}
        \sup_{t\in\R_+} \E_{\bar \pi} |\mu^{\bar \Gamma}_t - \bar \mu| \leq \frac{C}{\sqrt{n}}.
    \end{align*}
\end{corollary} 

The next result shows that the time-dependent Markovian strategy using the master equation's potential function $U_0$ is also a ($C/\sqrt{n}$)-Nash equilibrium. For brevity's sake, we will write the Markovian strategy profile as:
\begin{align}\label{eqn:me_control}
    \Gamma_0 := (\gamma_0 )_{i\in [n]}, \quad \gamma_0(t) := [\gamma^*_y(x, \Delta_x U_0 (\cdot, \mu^{\Gamma_0,-x}_t))]_{x,y\in [d]} .
\end{align} where $(\mu^{\Gamma_0, -x}_t)_{t\geq 0}$ is the empirical measure $(\mu^{\Gamma_0}_t)_{t\geq 0}$, but with one fewer player in state $x$.


Note that in \eqref{eqn:me_control}, we substitute the empirical distribution into the solution of the master equation, whereas in \eqref{eqn:stationary_control}, the strategy is based solely on the stationary MFG solution. 
This difference implies that the following theorem cannot be treated as a corollary and requires a distinct proof.

\begin{theorem}\label{thm:epsilon_equilibrium_2}
    Assume the full assumptions are in force. There exists $C>0$, independent of $n$, such that the strategy profile $\Gamma_0$ is a Markovian $(C/\sqrt{n})$-Nash equilibrium. Moreover, the cost for each player in the game, $J_0^{n} ( \Gamma_0)$, satisfies 
    \[
    |J_0^{n} ( \Gamma_0) - \vr | \leq \frac{C}{\sqrt{n}}.
    \] With $\pi_0$, the stationary distribution of $\bX^{\Gamma_0}$, and for all $\pi\in \calP([d]^n)$:
    \begin{align}\label{eqn:pi_thm2}
    \limsup_{T\to\iy}\E_{\pi} |\mu^{\Gamma_0}_T - \bar \mu| = \E_{\pi_0} |\mu^{\Gamma_0}_0 - \bar \mu| \leq \frac{C}{\sqrt{n}}.
    \end{align} 
\end{theorem}

As with $\check \Gamma$ and $\bar\Gamma$, we allow a player deviating from $\Gamma_0$ to use full information about their system, but that this information will not benefit them by more than $C/\sqrt{n}$.

\begin{remark}
    Note that the convergence of the empirical measures here is established under the stationary distribution $\pi_0$, or, in the limit, for an arbitrary initial distribution. Recently, Cohen and Huffman \cite{coh-huf2025} demonstrated a uniform-in-time weak convergence of the empirical measure for particle systems to a limiting flow of measures, which is governed by the Kolmogorov equation of a McKean--Vlasov (non-linear) Markov chain. These general results are applied there to the model considered here.
Specifically, assuming that all the players are initialized independently according to $\check\mu(0)$, then for any test function $\Phi: \mathcal{P}([d]) \to \mathbb{R}$ with Lipschitz gradient, we show that
\begin{equation}\notag
	\sup_{t\in\R_+}\left\lvert \EE \left[ \Phi(\mu^{\Gamma_0}_t)\right] - \Phi(\check\mu(t))   \right\rvert  \leq \frac{C}{n}\qquad\text{and}\qquad \sup_{t\in\R_+}\left\lvert \EE \left[ \Phi(\mu^{\check\Gamma}_t)\right] - \Phi(\check\mu(t))   \right\rvert  \leq \frac{C}{n}.
\end{equation}
The prototypical example of $\Phi$ is the squared Euclidean distance, $\eta\mapsto\lvert \eta  - \bar\mu \rvert^2$. 
    
\end{remark}
\color{black}\begin{remark}\label{rem10}
    This paper does not address the convergence of exact Nash equilibria to mean field equilibria. Our method relies on the Lipschitz continuity of $U_0$, the solution to the master equation. 
    While the structure of the Nash system resembles that of the master equation, connecting the two falls outside the scope of this work. The first author is actively working on this problem, which remains open (see \cite{MR4083905}).

Furthermore, while we believe that the convergence rate of the values $\varrho^n \to \varrho$ could be improved to $O(1/n)$, as seen in some finite horizon Markovian cases \cite{CardaliaguetDelarueLasryLions, bay-coh2019, cec-pel2019}
or some ergodic linear-quadratic \cite[Remark 5.3]{bar-pri2014}
, our probabilistic approach yields an $O(1/\sqrt{n})$ rate due to the propagation of chaos estimates, which is tight and consistent for example with the propagation of chaos in the previously mentioned works and with the finite-state mean field model \cite{ying2018}. Improving the rate for $|\varrho^n - \varrho|$ would require alternative methods, such as considering long finite horizons or discounted master equations, and leveraging exponential ergodicity. This approach, which the first author is actively exploring, requires different tools than those used in the present analysis.
\end{remark}
\color{black}

\subsection{Large deviations results} 
We start with a remark about $\bar\mu$ and how it is an equilibrium for Kolmogorov's equation. Specifically, the large deviation results we provide require a \textit{globally asymptotically stable equilibrium} and we flesh out precisely what we mean by this. Then, we introduce the laws featured in the large deviations result and the good rate functions for each law. \blue{The (large deviation) rate functions represent the cost of transferring the system from one distribution to another, with their specific forms depending on the time horizon and the spaces involved: whether we are dealing with distributions in \(\mathcal{P}([d])\) or flows of measures over \([0,T]\).
} \textit{For all large deviation results, we assume the full assumptions are in force.}

\begin{proposition}\label{prop:erg_stability}
Let $(\bar\vr, \bar u, \bar\mu)$ be the unique stationary solution to \eqref{erg_MFG}. Then, $\bar\mu$ is the unique globally asymptotically stable equilibrium for \eqref{erg_FP_t}. 
\end{proposition}

\begin{remark}
    As a reminder, $\bar\mu$ is an equilibrium to \eqref{erg_FP_t} means that $\bar\mu$ satisfies:
    \[
    0 = \sum_{y\in [d]} \bar\mu_y \gamma^*_x(y,\Delta_y U_0 (\cdot, \bar\mu)),
    \] that is, the derivative term in \eqref{erg_FP_t} is $0$. Recalling that $U_0(\cdot,\bar\mu) = \bar u$ from \Cref{prop:me_properties}, we see that this amounts to $\bar\mu$ solving the measure equation from the stationary ergodic MFG system \eqref{erg_MFG}. Finally, to be globally asymptotically stable means that for all initial conditions $\eta \in \calP([d])$, the solution $\mu_0$ to \eqref{erg_FP_t} tends to $\mu_0(t) \to \bar\mu$ as $t\to\iy$. For more about this definition, see \cite[Section~3]{ying2018}. 
\end{remark}

The main theorem deals with large deviations of the empirical measures. We will state the theorem in terms of $\kappa^n \in \{ \mu^{\Gamma_0},  \mu^{\bar\Gamma}\}$ and the corresponding rates $a^\kappa \in \{\gamma_0, \bar\gamma\}$  from \eqref{eqn:me_control} and \eqref{eqn:stationary_control}. 
We provide three large deviation results, for:
\[
\scrL^n_T := \PP \circ (\kappa^n_T)^{-1}, \quad \text{ for }\quad \scrL^n_{[0,T]} := \PP \circ ((\kappa^n_t)_{t\in[0,T]})^{-1},
\] and for $\scrL^n_\iy$, the unique invariant measure for the process $\kappa^n$ 
For each law, there are good rate functions that we take an aside to define. So, define:
\[
\tau(a) := e^a - a -1,
\] and its Legendre transform:
\[
\tau^* (a) := \begin{cases}
    (a+1) \log(a+1) - a, & a > -1,\\
    1, & a = -1,\\
    +\iy, & a <-1.
\end{cases} 
\]  
Denote the set of all absolutely continuous functions $\R_+ \to \calP([d])$ as $\calC_{\text{abs}}$ and for a few forthcoming matrices, we will denote the transpose of a matrix $B$ by $B^*$. Define the subset: 
\[
    \calC_{\text{abs}}^\eta := \Big\{\mu \in \calC_{\text{abs}} \mid \mu(0) = \eta, \quad \dot\mu(t) = L(t)^*\mu(t), \quad \text{ for some } L:\R_+ \to\calQ \text{ is measurable}\Big\}.
\] 
\blue{This is the set of all (absolutely continuous) flows of measures with (possibly time-dependent) rate matrices that, at each time instant $t$, belong to 
$\calQ$, starting from the distribution $\eta$.} Define also a function $S_{[0,T]}(\cdot \mid \eta) : \calC_{\text{abs}}^\eta \to \R$ as:
\[
    S_{[0,T]} (\mu \mid \eta) := \int_0^T \xi(t,y) a^\kappa_{x,y}(\mu(t)) \tau^*\Big(\frac{ L_{x,y}(t)}{a^{\kappa}_{x,y}(\xi(t))} - 1\Big)dt.
    \] 
\blue{As we shall see in a moment, this {\it action functional} represents the ``cost" or ``action" associated with transitioning the system state along the trajectory $(\mu_t)_{t\in[0,T]}$, starting from the initial state $\mu(0) = \eta$.
}
For $\xi \in\calP([d])$, set $A_\xi := [a^\kappa_{x,y}(\xi)]_{x,y \in [d]}$; so $a^\kappa_{x,y}(\xi) = \gamma_y^*(x,\Delta_x U_0(\cdot, \xi))$ for $\kappa^n = \mu^{\Gamma_0}$, and $a^\kappa_{x,y}(\xi) = \bar\gamma_{xy} = \gamma^*_y(x,\Delta_x \bar u)$ for $\kappa^n = \mu^{\bar\Gamma}$. So for $\kappa^n = \mu^{\bar\Gamma}$, $a^\kappa(\xi)$ is independent of $\xi$. As proved in \cite[Theorem~3.2]{MR3354770}, $S_{[0,T]}(\mu \mid \eta)$ can be written in an equivalent form as:
    \begin{align}\label{eqn:path_rt_function}
    S_{[0,T]} (\mu\mid\eta) = \int_0^T ||| \dot \mu(t) - A^*_{\mu(t)} \mu(t) |||_{\mu(t)} dt,
    \end{align}
    where we adopt the notation from \cite{MR3354770} for the function $|||R|||_\xi$ meaning:
    \begin{align*}
    |||R|||_\xi := \sup_{\phi : [d] \to \R} \Big[\sum_{x\in [d]} R_x \phi_x - \sum_{x,y\in [d]} \tau (\phi_y - \phi_x) \xi_x a^\kappa_{x,y} (\xi)\Big].
    \end{align*} 
   We note that $|||\cdot|||_{\xi}$ is not a norm. 
 \blue{This form of $S_{[0,T]}$  evaluates how the trajectory  $\mu$  deviates from the desired path  that solves the McKean--Vlasov ODE $\dot\theta(t)=A^*_{\theta(t)}\theta(t)$ over $[0, T]$, with $\mu(0)=\eta$. }  
   
   Through $S_{[0,T]}(\cdot \mid \eta)$ we can define $S_T(\cdot \mid \eta) : \calP([d]) \to \R$ as:
    \begin{align}\label{eqn:T_rate_fn}
    S_T(\xi\mid \eta) := \inf  \left\{S_{[0,T]} (\mu \mid \eta)\mid \;\mu \in \calC_{\text{abs}}^\eta, \quad\mu(T) = \xi\right\}.
    \end{align} 
\blue{This functional measures the minimal cost to transfer the distribution from the initial distribution $\mu(0)=\eta$ to the terminal distribution at time $T$, given by $\mu(T)=\xi$.}    
    Finally, we can define:
    \begin{align}\label{eqn:inf_rate_fn}
    s(\eta) := \inf_{\hat\mu} \int_0^\iy \sum_{x,y\in [d]} \hat\mu(t,y) a^{\kappa}_{x,y}(\hat\mu(t)) \tau^*\Big(\frac{\hat L_{x,y}(t)}{a^{\kappa}_{x,y}(\hat\mu(t))} - 1\Big)dt,
    \end{align} and the infimum is taken over all $\hat\mu : \R_+ \to \calP([d])$ such that $\hat\mu(0) = \eta$, $\lim_{t\to\iy} \hat\mu(t) = \bar\hat\mu$, and $\hat\mu$ satisfies the reversed Kolmogorov equation: \[
    \dot{\hat{\mu}}
(t) = - \hat L(t)\hat\mu(t),\qquad t\ge0, 
    \] 
    for some measurable $\hat L : \R_+ \to \calQ$. For a reference on large deviations for particle systems with jumps, see \cite{MR1324810}. \blue{The dynamics \( \dot{\hat{\mu}}
(t) = -\hat{L}(t)^* \hat{\mu}(t) \) represents a time reversal when contrasted with the direction of the McKean--Vlasov dynamics. The function $s$ represents the least expensive path (over all time) that transports the system state from the globally asymptotically stable equilibrium \( \bar\mu \) to \( \eta \) in the forward-time dynamics, aligning with the direction of the McKean--Vlasov dynamics. In the time-reversed dynamics, this cost pertains to a path that starts at \( \eta \) and ends at \( \bar\mu \).
}

\begin{theorem}[Large Deviations]\label{thm:large_deviations}
    For all $T>0$, the laws $\scrL^n_{[0,T]}$, $\scrL^n_{T}$, and $\scrL^n_\iy$ satisfy large deviation principles with speed $n$ and good rate functions given by $S_{[0,T]}$ \eqref{eqn:path_rt_function}, $S_{T}$ \eqref{eqn:T_rate_fn}, and $s$ \eqref{eqn:inf_rate_fn}, respectively. Explicitly, let $E$ be any Borel subset of $\calP([d])$, denote its interior and closure by $\text{int}(E)$ and $\text{cls(E)}$, respectively, and we have:
    \[
    -\inf_{\eta \in \text{int}(E)} S_T(\eta \mid \bar\mu) \leq \liminf_{n\to\iy} n^{-1} \log(\scrL^n_T (E)) \leq \limsup_{n\to\iy} n^{-1} \log(\scrL^n_T (E)) \leq -\inf_{\eta\in\text{cls}(E)} S_T(\eta \mid \bar\mu),
    \] and
    \[
    -\inf_{\eta \in \text{int}(E)} s(\eta) \leq \liminf_{n\to\iy} n^{-1} \log(\scrL^n_\iy(E)) \leq \limsup_{n\to\iy} n^{-1} \log(\scrL^n_\iy(E)) \leq -\inf_{\eta\in\text{cls}(E)} s(\eta).
    \] Now let $G$ be a Borel set in the Skorohod topology of the space $\calC^{\bar\mu}_{\text{abs}}$. We also have:
    \begin{align*}
    -\inf_{\nu \in \text{int}(G)} S_{[0,T]}(\nu \mid \bar\mu) &\leq \liminf_{n\to\iy} n^{-1} \log(\scrL^n_{[0,T]} (G)) \\
    &\qquad \leq \limsup_{n\to\iy} n^{-1} \log(\scrL^n_{[0,T]} (G)) \leq -\inf_{\nu\in\text{cls}(G)} S_{[0,T]}(\nu \mid \bar\mu).
    \end{align*}
\end{theorem}

\section{Proofs for \Cref{thm:epsilon_equilibrium}}\label{sec:asymp_proofs}

In this section, we give the proof for \Cref{thm:epsilon_equilibrium} concerning the asymptotic Nash equilibria derived from the time-dependent ergodic MFG \eqref{erg_MFG_t}. The first lemma we require is  \Cref{lem:concentration_inequality_1}, which gives a concentration inequality for players using a symmetric Markovian strategy profile when they are initialized the same way. The second, \Cref{lem:pre_duality}, compares an empirical measure under an arbitrary Markovian strategy profile to $\check\mu$. The bound we derive will depend on the time $t$, the initial conditions of each system, the number of players $n$, and the difference between the strategies. The two prior lemmas combine to obtain the propagation of chaos estimate \eqref{eqn:pi_thm1} that allows us to compare mean field costs $F$ at one point in the proof of the theorem. 

Throughout this section, we require the reduced assumptions only. In the forthcoming proofs, $C$ and $c$ will be positive constants that may change value from one line to the next. Whenever $C$ and $c$ appear in the proofs, they depend only on the problem data and we emphasize that they are always independent of $n, t$, and any initialization of the players. The concentration inequality lemma is:

\begin{lemma}\label{lem:concentration_inequality_1}
Let $\bbeta$ be a symmetric Markovian strategy  profile with $\bbeta := (\beta^i)_{i\in[n]}$, $\beta^i \in\calA^{n,i}$, and let $\bX^{\bbeta}$ be the jump process of all players using $\bbeta$. Suppose further that $(X^{\bbeta, i}_0)_{i\in [n]}$ are exchangeable though not necessarily independent. So, all processes are initialized the same way; that is, for all $i\in [n]$, $\calL(X^{\bbeta, i}_0) = \calL (X^{\bbeta, 1}_0)$. 
Assume $\calL(X^{\bbeta, 1}_0)$ is not on the boundary of $\calP([d])$. Then there exists $q>0$ depending only on $\A$ such that for all $t\in\R_+$, and all $i\in [n]$,
\begin{equation*}
\calL(X^{\bbeta, i}_t)_x > q >0 \quad \text{ and }\quad 1-\calL(X^{\bbeta, i}_t)_x > q >0 \quad \text{ for all }\quad x\in [d].
\end{equation*} Moreover, denote by $\pi^{\bbeta} \in \calP([d]^n)$ the joint distribution induced by this initialization. there exists $C>0$, independent of $n$, so that: 
\begin{align}\label{exp_upper_bound_1}
\sup_{t\in\R_+} \PP_{\pi^{\bbeta}}\Big(\Big|\frac{1}{n}\sum_{j\in[n]} \delta_{X^{\bbeta, j}_t} (x) - \calL(X^{\bbeta, 1}_t)_x\Big|\geq y \Big) \leq 2e^{-ny^2/4q},
\end{align} 
\begin{equation}\label{eqn:hoeffding_1}
    \sup_{t\in\R_+} \EE_{\pi^{\bbeta}}|\mu^{\bbeta}_t - \calL(X^{\bbeta, i}_t)| \leq \frac{C}{\sqrt{n}}.
\end{equation} When $\calL(X^{\bbeta, 1}_0)$ is on the boundary of $\calP([d])$, there exists such $q$ for $t\geq 1$ and so we must modify the suprema in \eqref{exp_upper_bound_1} and \eqref{eqn:hoeffding_1} to be over all $t \geq 1$.
\end{lemma}

Also let $(\check \vr, \check u, \check \mu)$ be a solution to the time-dependent ergodic MFG system \eqref{erg_MFG_t}. As formulated in the statement of the theorem, recall that $\check \gamma_{xy}(t) := \gamma^*_y(x,\Delta_x \check u(t))$ and set $\check \gamma_x := (\check \gamma_{xy})_{y\in [d]}$, $\check \gamma := (\check \gamma_x)_{x\in [d]}$, and $\check \Gamma := (\check \gamma)_{i=1}^n$. Note that since $\eqref{erg_MFG_t}$ is the time-dependent system, the control $\check \gamma$ depends implicitly on time through the potential function $\check u$. The following lemma is helpful for the proof of \Cref{thm:epsilon_equilibrium} and for the proof of \Cref{thm:epsilon_equilibrium_2}.

\begin{lemma}\label{lem:pre_duality}
    Fix $\eta\in\calP([d])$ and let $(\check \vr, \check u, \check \mu)$ solve \eqref{erg_MFG_t} with $\check \mu(0) = \eta$. There exist $C,c$ positive constants such that for any $n \in\N_{>1}$ and any $\bbeta \in \calA^{(n)}$, any initialization $\pi \in\calP([d]^n)$ with $\bX^{\bbeta}_0 \sim\pi$, and any $t\in\R_+$, 
    \begin{align}\label{eqn:control_compare_1}
    \begin{split}
    \E_\pi |\mu^{\bbeta}_t - \check \mu(t)| &\leq Ce^{-ct}\E_\pi |\mu^{\bbeta}_0 - \check \mu(0)| + \frac{C}{\sqrt{n}} \\
    &\qquad + \E_\pi \int_0^t e^{-c(t-s)} \sum_{x\in [d]} \mu^{\bbeta}_{s,x} \Big|\frac{1}{n} \sum_{i\in [n]} \big[\beta^i_x(s,\bX^{\bbeta}_s) - \check \gamma_{x}(t) \big]\1_{\{X^{\bbeta, i}_s = x\}} \Big|ds.
    \end{split}
    \end{align} Suppose $(X^{\bbeta, i}_0)_{i\in [n]}$ are exchangeable, though not necessarily independent, with $X^{\bbeta, i}_0 \sim \check \mu (0)$ for all $i\in [n]$. For $\pi = \check \pi$, the joint distribution induced by this initialization, we have from \eqref{eqn:control_compare_1} that
    \begin{align}\label{eqn:control_compare_2}
    \E_{\check \pi} |\mu^{\bbeta}_t - \check \mu(t)| \leq \frac{C}{\sqrt{n}} + \E_{\check \pi} \int_0^t e^{-c(t-s)} \sum_{x\in [d]} \mu^{\bbeta}_{s,x} \Big|\frac{1}{n} \sum_{i\in [n]} \big[\beta^i_x(s,\bX^{\bbeta}_s) - \check \gamma_{x}(t) \big]\1_{\{X^{\bbeta, i}_s = x\}} \Big|ds.
    \end{align} 
\end{lemma}

The proofs of all lemmas are provided in the subsequent part of this section, following the conclusion of the proof of \Cref{thm:epsilon_equilibrium}.

\subsection{Proof of \Cref{thm:epsilon_equilibrium}} 

We split the proof into two parts. In the first part, we prove the propagation of chaos results \eqref{eqn:pi_thm1} and \eqref{eq:propagation_check}. We will then use these results in the second part to prove that $\check \Gamma$ is a $(C/\sqrt{n})$-Nash equilibrium and \eqref{eqn:cost_convergence}.

\subsubsection{Propagation of chaos}

To prove \eqref{eqn:pi_thm1} and \eqref{eq:propagation_check}, we start by proving a related result to \eqref{eqn:pi_thm1}. We will show there exist $C,c$ positive constants such that for any $n\in\N_{>1}$, any $\beta \in\calA^{n,i}$, and any initialization $\pi \in\calP([d]^n)$ with $\bX^{[\check \Gamma^{-i} ; \beta]}_0 \sim \pi$, and any $t\in\R_+$,
    \begin{align*}
        \E_\pi |\mu_t^{[\check \Gamma^{-i} ; \beta]} - \check \mu(t)| \leq Ce^{-ct} \E_\pi |\mu_0^{[\check \Gamma^{-i} ; \beta]} - \check \mu(0)| +\frac{C}{\sqrt{n}}.
    \end{align*} First note that $\mu^{[\check \Gamma^{-i} ; \beta], i} = \mu^{\check \Gamma, i}$ since the $i$-th player is removed from both processes and since the strategy $\check \gamma$ is independent of the players' process $\bX^{[\check \Gamma^{-i} ; \beta]}$. Moreover, since \[
\mu^{\check \Gamma} = n^{-1} \big((n-1)\mu^{\check \Gamma, i} + \delta_{X^{\check \Gamma, i}}\big),
\] the measures $\mu^{\check \Gamma}$ and $\mu^{\check \Gamma, i}$ differ by a term of order $\calO(n^{-1})$. So, it suffices to show the main estimate for $\mu^{\check \Gamma}$. Applying \Cref{lem:pre_duality} for $\bbeta= \check \Gamma$, we note that $\beta^i =\check \gamma$ for every $i\in [n]$ and the integral term is zero, yielding:
\[
\E_\pi |\mu^{\check \Gamma}_t - \check \mu(t)| \leq Ce^{-ct} \E_\pi |\mu^{\check \Gamma}_0 - \check \mu(0)| + \frac{C}{\sqrt{n}}.
\] We recognize the prior display as \eqref{eqn:pi_thm1}. To get \eqref{eq:propagation_check}, when $\pi = \check \pi$, we note that the term in expectation on the right-hand side is no more than $C/\sqrt{n}$ by \Cref{lem:concentration_inequality_1}.

\subsubsection{Approximate Markovian Nash Equilibrium}

We can now use the propagation of chaos results to obtain the remaining points of \Cref{thm:epsilon_equilibrium}. We abuse notation slightly by writing $X^{\beta, i}$ for $X^{[\check\Gamma^{-i};\beta], i}$, keeping in mind that the dynamics of player $i$ are implicitly dependent on the other players' dynamics, through player $i$'s strategy $\beta$, and the rest of the players' strategy profile.

By \Cref{lem:freedom} we can consider $J_0^{n,i}$ with any initial distribution for the agents. In the following display, we make use of this fact to study the value when a single player, player $i$, deviates from the strategy $\check \gamma$. Player $i$ will use $\beta\in\calA^{n,i}$ instead. Let $\check \pi \in \calP([d]^n)$ be the joint distribution induced by 
 initializing every player according to $\check \mu(0) \in \calP([d])$.
Then, adding and subtracting a term:

\begin{align}\label{eqn:1}
\begin{split}
    J^{n,i}_0 ([\check \Gamma^{-i}; \beta]) &= \limsup_{T\to\iy} \frac{1}{T}\EE_{\check \pi} \int_0^T [f(X^{\beta, i}_t, \beta(t, \bX^{[\check \Gamma^{-i};\beta]}_t)) + F(X^{\beta, i}_t,  \mu^{[\check \Gamma^{-i};\beta],i}_t)]dt \\
    &= \limsup_{T\to\iy} \Big(\frac{1}{T}\EE_{\check \pi} \int_0^T [f(X^{\beta, i}_t, \beta(t, \bX^{[\check \Gamma^{-i};\beta]}_t)) + F(X^{\beta, i}_t, \check \mu(t))]dt \\
    &\qquad+  \frac{1}{T}\EE_{\check \pi} \int_0^T [F(X^{\beta, i}_t,  \mu^{[\check \Gamma^{-i};\beta],i}_t) - F(X^{\beta, i}_t, \check \mu(t))]dt\Big).\\
\end{split}
\end{align}

We will consider each term on the right-hand side of \eqref{eqn:1} separately. For the first term, we note two things: (1) that we initialized $X^{\beta, i}_0 \sim \check \mu(0)$ and (2) that $\check \mu$ is an MFE by \Cref{prop:me_properties}. Then by \Cref{def:erg_MFE} and the definition of $J_0$, the strategy $\beta$ that minimizes this function must be $\check \gamma$. Hence,
\begin{align}\label{eqn:3}
    \inf_{\beta \in \calA} \limsup_{T\to\iy} \frac{1}{T}\EE_{\check \pi} \int_0^T [f(X^{\beta, i}_t, \beta(t, \bX^{[\check \Gamma^{-i};\beta]}_t)) + F(X^{\beta, i}_t, \check \mu(t))]dt = J_0(\check \gamma, \check \mu).
\end{align} We will keep this fact aside for the moment and return to \eqref{eqn:1} in order to deal with the second term on the right-hand side. We note once more that each of the players has an initial state with law $\check \mu(0)$. 
Bringing the absolute value inside the integral, using that $F$ is Lipschitz, using Fubini's theorem to exchange to the expectation and integral, and by $\mu^{[\check \Gamma^{-i}; \beta], i} = \mu^{\check \Gamma, i}$ with \eqref{eq:propagation_check}, we can bound the second term on the right-hand side of \eqref{eqn:1}:
\begin{align}\label{eqn:4}
    \Big|\frac{1}{T}\EE_{\check \pi} \int_0^T [F(X^{\beta, i}_t,  \mu^{[\check \Gamma^{-i} ; \beta],i}_t) - F(X^{\beta, i}_t, \check \mu(t))]dt \Big| \leq \frac{C}{\sqrt{n}}.
\end{align} 
So using \eqref{eqn:1} and \eqref{eqn:4}, followed by \eqref{eqn:3}, 
\begin{align}\label{eqn:cn_nash}
    \begin{split}
        J_0^{n} (\check \Gamma) - &J_0^{n,i} ([\check \Gamma^{-i} ; \beta]) \\
        &= J_0 ( \check \gamma, \check \mu ) - \limsup_{T\to\iy} \frac{1}{T}\EE_{\check \pi^n} \int_0^T [f(X^{\beta, i}_t, \beta(t, \bX^{[\check \Gamma^{-i};\beta]}_t)) + F(X^{\beta, i}_t, \check \mu(t))]dt + \calO(n^{-1/2}) \\
        &\leq \calO(n^{-1/2}).
    \end{split}
\end{align} Since $\beta$ was arbitrary,
$\check \Gamma$ is a $(C/\sqrt{n})$-Nash equilibrium. It only remains to show \eqref{eqn:cost_convergence} and by \eqref{eqn:1}, \eqref{eqn:3}, \eqref{eqn:4}, it suffices to show that $J_0(\check \gamma, \check \mu) = \check \vr$. This follows immediately from follows from \Cref{prop:me_properties}. $\square$

\subsection{Proof of \Cref{lem:concentration_inequality_1}}

Denote $\kappa(t) := \calL(X^{\bbeta, 1}_t)$. By assumption, all players are initialized independently according to $\kappa(0)$. Since $\bbeta$ is symmetric, we claim the players are exchangeable with the same law $\kappa(t)$ for all $t\in\R_+$. That is, for all $j\in [n]$ and all $t\in\R_+$, we claim $\calL(X^{\bbeta, j}_t) = \kappa(t)$.

Now, since the rates $\A$ are bounded away from zero and since $\calL(X^{\bbeta, i}_t) = \kappa(t)$, we must have that there exists $q>0$ such that for all $x\in [d]$ and all $t>1$, $\calL(X^{\bbeta, i}_t)_x > q>0$, even when $\kappa(0)$ is on the boundary. The choice of time $1$ for this threshold is not crucial to the proof, any positive constant would do. By symmetry of the players, for any $t\in\R_+$, $(X^{\bbeta, i}_t)_{i\in[n]}$ are exchangeable with the same marginal law $\kappa(t)$. Fix $x\in [d]$. Then $(\delta_{X^{\bbeta, i}_t}(x))_{i\in[n]}$ are exchangeable Bernoulli random variables, each with probability $\kappa_x(t)$ to be $1$. By \cite[Corollary~2.2]{hoeffdingIneq} , we get \eqref{exp_upper_bound_1} immediately. Using \eqref{exp_upper_bound_1}, and by definition of the expectation,
\begin{align*}
    \EE_{\pi^{\bbeta}} |\mu^{\bbeta}_{t,x} - \kappa_x(t)| = \int_0^\infty \PP_{\pi^{\bbeta}} \Big(\Big|\frac{1}{n}\sum_{i\in[n]} \delta_{X^{\bbeta, i}_t} (x) - \kappa_x(t) \Big|\geq y \Big) dy 
    \leq \int_0^\infty 2e^{-ny^2/4q} dy 
    = 2\sqrt{\frac{\pi q}{n}},
\end{align*} which establishes \eqref{eqn:hoeffding_1}. $\square$

\subsection{Proof of \Cref{lem:pre_duality}}

Recall the definition of $\check \Gamma$ from \Cref{thm:epsilon_equilibrium}. Define the transition matrix $Q(t)$ and its inverse by: 
\[
Q(t) := \exp \Big( \int_0^t \check \Gamma(s) ds \Big)\quad \text{ and } \quad Q^{-1}(t) := \exp \Big( - \int_0^t \check \Gamma (s) ds \Big).
\] Also define:
\[
\Gamma^{\bbeta} (s)  := \Big(\frac{1}{n} \sum_{i\in [n]} \beta^i_{xy}(s,\bX^{\bbeta}_s)\1_{\{X^{\bbeta, i}_s = x\}} \Big)_{x,y \in [d]}.
\] 
Note that for any time $s$ and any realization of $\bX^{\bbeta}_s$, we have $\Gamma^{\bbeta}_s \in \calQ$. From the dynamics of $\mu^{\bbeta}$ and since $\check \mu$ solves Kolmogorov's equation from \eqref{erg_MFG_t}, we get that:
    \begin{align*}
        \begin{split}
        &(\mu^{\bbeta}_{s} - \check\mu(s))^T Q^{-1}(s) \Big|_0^t \\
        &\quad=  \int_0^t (\mu^{\bbeta}(s)^T \Gamma^{\bbeta} (s) - \check \mu(s)^T \check \Gamma(s) ) Q^{-1} (s) ds +\int_0^t (\mu^{\bbeta}_s - \check \mu(s) )^T [- \check \Gamma(s) Q^{-1}(s)] ds
        \\&\qquad+\frac{1}{n}\sum_{x,y\in[d]}\int_0^t (e_y-e_x)Q^{-1}(s^-)\Big(d K^{\bbeta}_{xy}(s)-\sum_{i\in [n]} \beta^i_{xy}(s,\bX^{\bbeta}_s)\1_{\{X^{\bbeta, i}_s = x\}}ds\Big),
        \end{split}
    \end{align*} 
    where $K^{\bbeta}$ is given in \eqref{eq:K}.
    Multiplying both sides by $Q(t)$ and cancelling the terms with $\check \mu(s)^T\check\Gamma(s)Q^{-1}(s)$ from the first two integrals, we get,
    \begin{align} \notag
        &(\mu^{\bbeta}_t - \check \mu(t)) -(\mu^{\bbeta}_0-\check \mu(0)) Q(t)\\\label{eqn:5} 
        &\quad= \int_0^t(\mu^{\bbeta}_s)^T(\Gamma^{\bbeta}(s)-\check \Gamma(s))Q^{-1}(s) Q(t) ds
        +\frac{1}{n} \sum_{x,y\in[d]}\int_0^t (e_y-e_x)Q^{-1}(s^-)Q(t)dM^{\bbeta}_{xy}(s),
    \end{align}
    where $M^{\bbeta}_{xy}$ is a martingale given by,
    \begin{align*}
        M^{\bbeta}_{xy}(t):=\Big( K^{\bbeta}_{xy}(t)-\int_0^t\sum_{i\in [n]} \beta^i_{xy}(s,\bX^{\bbeta}_s)\1_{\{X^{\bbeta, i}_s = x\}}ds\Big).
    \end{align*} 
    Let $R(s,t) := (e_y-e_x) Q^{-1}(s) Q(t)$. For the following display, we use: the Burkholder--Davis--Gundy inequality; that the quadratic variation of a compensated Poisson process is the original process, so $\langle M_{xy}^{\bbeta}(s) \rangle = K^{\bbeta}_{xy}(s)$; Jensen's inequality for the concave function $\sqrt{\cdot}$; adding and subtracting the compensator for $K^{\bbeta}_{xy}(s)$ and using that the martingale term has expectation zero; the boundedness of the rates $\beta^i$; and finally $R(s,t)R(s,t)^T = |R(s,t)|^2$ with \cite[Lemma~3.3]{CZ2022}:
        \begin{align}\label{eqn:obs1}
        \begin{split}
            \E \Big|\int_0^t R(s^-,t) dM^{\bbeta}_{xy}(s) \Big| &\leq C \E \Big[\Big( \int_0^t R(s^-,t)R(s^-,t)^T \langle dM_{xy}^{\bbeta}(s) \rangle \Big)^{1/2} \Big] \\
            &= C \E \Big[\Big( \int_0^t R(s^-,t)R(s^-,t)^T dK^{\bbeta}_{xy}(s) \Big)^{1/2} \Big] \\
            &\leq C \Big[ \E \Big( \int_0^t R(s^-,t)R(s^-,t)^T dK^{\bbeta}_{xy}(s) \Big) \Big]^{1/2} \\
            &= C \Big[ \E \Big( \int_0^t R(s,t)R(s,t)^T \sum_{i\in [n]} \beta^i_{xy}(s,\bX^{\bbeta}_s) \1_{\{X^{\bbeta, i}_s = x\}}ds \Big) \Big]^{1/2} \\
            &\leq Cn^{1/2} \Big[ \int_0^t R(s,t) R(s,t)^T ds \Big]^{1/2} \\
            &\leq Cn^{1/2}.
        \end{split}
        \end{align}

Also, let $\tilde \mu(t) := (\mu^{\bbeta}_0)^T Q(t)$ and note that while $Q(t)$ is deterministic, $\tilde \mu(t)$ is random since $\mu^{\bbeta}_0$ is random. Note that $\check \mu(t) = (\check \mu(0))^T Q(t)$ too. So by \cite[Lemma~3.3]{CZ2022}, 
\[
|\tilde \mu(t) - \check \mu(t)| \leq Ce^{-ct} |\mu^{\bbeta}_0 - \check \mu(0)|.
\] Taking the expectation,
\begin{align}\label{eqn:obs2}
\EE_\pi |(\mu^{\bbeta}_0 - \check \mu(0)) Q(t)| \leq Ce^{-ct} \EE_\pi |\mu^{\bbeta}_0 - \check \mu(0)|.
\end{align}

We now return to \eqref{eqn:5} and, first taking the expectation, use the martingale bound \eqref{eqn:obs1} along with \eqref{eqn:obs2} to get that:
\begin{align*}
    \E_\pi |\mu^{\bbeta}_t - \check\mu (t)| \leq Ce^{-ct}\E_\pi |\mu^{\bbeta}_0 - \check \mu(0)| + \frac{C}{\sqrt{n}} + \E_\pi \Big| \int_0^t(\mu^{\bbeta}_s)^T(\Gamma^{\bbeta}(s)-\check \Gamma (s))Q^{-1}(s) Q(t) ds \Big|.
\end{align*} So applying triangle inequality, again using \cite[Lemma~3.3]{CZ2022}, and expanding the matrix notation as a sum finishes the proof of \eqref{eqn:control_compare_1}. For \eqref{eqn:control_compare_2}, we note that when $\pi = \check \pi$, the first term on the right-hand side is no more than $C/\sqrt{n}$ by \Cref{lem:concentration_inequality_1}. $\square$

\section{Proofs for \Cref{thm:epsilon_equilibrium_2}}\label{sec:thm2}

In addition to \Cref{lem:pre_duality} from the prior section, we need two lemmas to approach the proof of \Cref{thm:epsilon_equilibrium_2}. The first, \Cref{lem:stat_is_enough}, introduces a mock ergodic system with a solution $(\vr^i, W^i_0)$ meant to mimic the optimal value and associated potential for the deviating $i$-th player, optimizing against the other players, who are each using $\gamma_0$. Through this detour, we prove that the deviating player's optimal strategy is actually stationary. This is helpful to reduce the problem of proving $\Gamma_0$ is a Markovian $(C/\sqrt{n})$-Nash equilibrium from studying all strategies $\calA^{n,i}$ for player $i$ to just stationary strategies $\calA^{n,i}_S$. 

The second lemma, \Cref{lem:miracle}, is modeled after duality estimates originally used for the MFG system. Its estimate is crucial to avoiding a consequence of the deleterious feedback effect: to compare strategies, we must compare processes that we plug into them; and to compare processes, we must compare their strategies. The proof of the lemma relies on the full assumptions as well; namely, we will need that $F$ is Lasry--Lions monotone and that $H$ is concave. 

In the forthcoming proofs, $C$ and $c$ will be positive constants that may change value from one line to the next. Whenever $C$ and $c$ appear in the proofs, they depend only on the problem data and we emphasize that they are always independent of $n, t$ and any initialization of the players. Recall that for \Cref{thm:epsilon_equilibrium_2}, the full assumptions hold, and will be in force throughout this section. Consequently, the entirety of \Cref{prop:erg_MFG_red_full} and \Cref{prop:me_properties} hold. We will let $(\vr, \check u, \check \mu)$ be the unique solution to \eqref{erg_MFG_t} and let $(\vr, \bar u, \bar \mu)$ be the unique solution to \eqref{erg_MFG}.

For any $\bx\in [d]^n$ we define the empirical distribution with and without the $i$-th player as:
\[
\mu_{\bx}^{n} := \frac{1}{n}\sum_{j\in [n]} e_{x^j}, \qquad \mu_{\bx}^{n,i} := \frac{1}{n}\sum_{j\in [n], j\neq i} e_{x^j}.
\]

\begin{lemma}\label{lem:stat_is_enough}
    Based on $\gamma_0$ from \eqref{eqn:me_control} define $\gamma^j_0(\bx) := \gamma^*(x^j, \Delta_{x^j} U_0 (\cdot, \mu_{\bx}^{n,j})).$ There exist $\vr^i \in\R$, $W_0^i : [d]\times \calP^{n-1}([d]) \to\R$ such that $(\vr^{i}, W_0^i)$ satisfies:
    \begin{align}\label{eqn:aux_erg}
        \vr^i &= H(x^i,\Delta_{x^i} W_0^i (\cdot, \mu_{\bx}^{n,i})) + F(x^i,\mu_{\bx}^{n,i}) + \sum_{j\neq i} \gamma^j_0(\bx) \cdot \Delta_{x^j} U_0 (\cdot, \mu_{\bx}^{n,j}).
    \end{align} Moreover, for $\gamma^W_y (\bx) := \gamma^*_y(x^i,\Delta_{x^i} W^{i}_0(\cdot, \mu_{\bx}^{n,i}))$, $\gamma^W := (\gamma^W_y)_{y\in [d]}$, we have that:
    \begin{align}\label{eqn:actual_optimal}
        \vr^{i} = J_0^{n,i} ([\Gamma_0^{-i} ; \gamma^W]) = \inf_{\beta\in \calA^{n,i}} J_0^{n,i} ([\Gamma_0^{-i} ; \beta]).
    \end{align} Interpreting \eqref{eqn:actual_optimal} in words, a stationary Markovian strategy is optimal for a player deviating from $\Gamma_0$. 
\end{lemma}

\begin{lemma}\label{lem:miracle}
    There exists $C>0$, independent of $n$, such that for all $\beta\in\calA^{n,i}$ and all $\pi\in\calP([d]^n)$:
    \[
    \limsup_{T\to\iy} \frac{1}{T} \E_\pi \int_0^T \sum_{x\in [d]} (\mu^{[\Gamma_0^{-i} ; \beta]}_{t,x} + \check \mu_x(t)) |\Delta_x U_0(\cdot, \mu_t^{[\Gamma_0^{-i} ; \beta]}) - \Delta_x \check u(t)|^2 dt \leq \frac{C}{n}.
    \]
\end{lemma}

\subsection{Proof of \Cref{thm:epsilon_equilibrium_2}}

As with \Cref{thm:epsilon_equilibrium}, we split the proof into two halves. The first part proves the propagation of chaos given in \eqref{eqn:pi_thm2}. For the propagation of chaos, we start by using \Cref{lem:pre_duality} to pass off the problem of comparing processes to comparing their strategies. The strategies are Lipschitz, so it suffices to compare their different potential functions. Then we use the duality-style lemma, \Cref{lem:miracle}, to show that the integral term left over from the application of \Cref{lem:pre_duality}, and that is comprised of these potential functions, is small. Along the way, we will need to use that the stationary distribution $\pi^\beta$ forces the processes inside the expectation to be independent of time; this helps us periodically move the processes outside of pesky, time-dependent integrals. 

In order to use the propagation of chaos to prove the profile is an asymptotic Nash equilibrium, we must start by restricting the proof to stationary strategies. Otherwise, we would not necessarily have a stationary distribution and the prior part would not apply. After we apply propagation of chaos and obtain the result for stationary strategies, we can use \Cref{lem:stat_is_enough} to upgrade the result so that it holds for any Markovian strategy.

\subsubsection{Propagation of chaos}
In this part we will prove: \begin{align}\label{eqn:b_prop_chaos_0}
    \E_{\pi^\beta} |\mu^{[\Gamma_0^{-i} ; \beta^S]}_0 - \bar\mu| \leq \frac{C}{\sqrt{n}},
\end{align} for $\beta^S \in\calA^{n,i}_S$, an arbitrary stationary Markovian strategy and $\pi^{\beta} \in\calP([d]^n)$, the unique stationary distribution for the jump process $\bX^{[\Gamma_0^{-i} ; \beta^S]}$. Note that \eqref{eqn:b_prop_chaos_0} is more general than \eqref{eqn:pi_thm2}. The full strength of \eqref{eqn:b_prop_chaos_0} is necessary for the latter part of the proof, where we prove $\Gamma_0$ is an asymptotic Nash equilibrium, and so deal with a deviating player.


Recall that $\Gamma_0$ is comprised of $\gamma_0$ defined in \eqref{eqn:me_control} and recall $\bar\gamma$ from \eqref{eqn:stationary_control}. We now apply \Cref{lem:pre_duality} with $\bbeta = [\Gamma_0^{-i} ; \beta^S]$ and $\pi = \pi^\beta$, use the fact that the rates are bounded to aggregate another term into $C/\sqrt{n}$, and then use that $(\tfrac{n-1}{n}) (\mu^{[\Gamma_0^{-i} ; \beta^S],i}_{s,x}) \leq 1$ to get:
\begin{align*}
    \E_{\pi^\beta} |\mu^{[\Gamma_0^{-i} ; \beta^S]}_t - \bar\mu| &\leq Ce^{-ct} \E_{\pi^\beta} |\mu^{[\Gamma_0^{-i} ; \beta^S]}_0 - \bar\mu| +\frac{C}{\sqrt{n}} \\
        &\quad + \E_{\pi^\beta} \int_0^t e^{cs-ct} \sum_{x\in [d]} \mu^{[\Gamma_0^{-i} ; \beta^S]}_{s,x} \big|\big(\tfrac{n-1}{n} (\mu^{[\Gamma_0^{-i} ; \beta^S],i}_{s,x}) [\gamma_{0,yx}(s) - \bar \gamma_{yx} ] \\
            &\qquad\qquad\qquad + \tfrac{1}{n}\1_{\big\{X^{[\Gamma_0^{-i} ; \beta^S],i}_s = x\big\}} [\beta_{yx}^S (\bX^{[\Gamma_0^{-i} ; \beta^S]}_s) - \bar \gamma_{yx}] \big)_{y\in [d]}\big| ds \\
    &\leq Ce^{-ct} \E_{\pi^\beta} |\mu^{[\Gamma_0^{-i} ; \beta^S]}_0 - \bar \mu| +\frac{C}{\sqrt{n}} \\
        &\quad + \E_{\pi^\beta} \int_0^t e^{cs-ct} \sum_{x\in [d]} \mu^{[\Gamma_0^{-i} ; \beta^S]}_{s,x} \big|\big(\tfrac{n-1}{n}\big) (\mu^{[\Gamma_0^{-i} ; \beta^S],i}_{s,x}) [\gamma_{0,yx}(s) - \bar \gamma_{yx} ]\big| ds \\
    &\leq Ce^{-ct} \E_{\pi^\beta} |\mu^{[\Gamma_0^{-i} ; \beta^S]}_0 - \bar \mu| +\frac{C}{\sqrt{n}} \\
        &\quad + \E_{\pi^\beta} \int_0^t e^{cs-ct} \sum_{x\in [d]} \mu^{[\Gamma_0^{-i} ; \beta^S]}_{s,x} |\gamma_{0,yx}(s) - \bar \gamma_{yx}| ds .
\end{align*}
Now, leveraging the fact that $\pi^\beta$ represents the stationary distribution, along with the definitions of $\gamma_0$ and $\bar\gamma$ in terms of $\gamma^*$, and considering the Lipschitz continuity of $\gamma^*$, we can derive:
\begin{align}\label{eqn:32_1}
    \begin{split}
    \E_{\pi^\beta} |\mu^{[\Gamma_0^{-i} ; \beta^S]}_t - \bar\mu| &\leq Ce^{-ct} \E_{\pi^\beta} |\mu^{[\Gamma_0^{-i} ; \beta^S]}_0 - \bar\mu| +\frac{C}{\sqrt{n}} \\
    &\quad + C \E_{\pi^\beta}\Big[ \sum_{x\in [d]} \mu^{[\Gamma_0^{-i} ; \beta^S]}_{0,x} \big|\Delta_x (U_0(\cdot, \mu^{[\Gamma_0^{-i} ; \beta^S]}_0) - \bar u)\big|\Big].
    \end{split}
\end{align} We can also use that $\pi^\beta$ is the stationary distribution, then \Cref{lem:freedom} and then \Cref{lem:miracle} to get:
\begin{align}\label{eqn:32_2}
    \begin{split}
        \E_{\pi^\beta}\Big[ \sum_{x\in [d]} (\mu^{[\Gamma_0^{-i} ; \beta^S]}_{0,x} &+ \bar\mu_x) \big|\Delta_x (U_0(\cdot, \mu^{[\Gamma_0^{-i} ; \beta^S]}_0) - \bar u)\big|^2 \Big] \\
        &= \limsup_{T\to\iy} \frac{1}{T} \E_{\pi^\beta}\int_0^T \Big[ \sum_{x\in [d]} (\mu^{[\Gamma_0^{-i} ; \beta^S]}_{t,x} + \bar\mu_x) \big|\Delta_x (U_0(\cdot, \mu^{[\Gamma_0^{-i} ; \beta^S]}_t) - \bar u)\big|^2 dt\\
        &\leq \frac{C}{n}.
    \end{split}
\end{align} We can use Jensen's inequality for $\sqrt{\cdot}$ and that for any $\eta\in\calP([d])$, $\eta_x^2\leq \eta_x$ for all $x\in [d]$ in \eqref{eqn:32_1} and then \eqref{eqn:32_2} to get:
\begin{align*}
    \E_{\pi^\beta} |\mu^{[\Gamma_0^{-i} ; \beta^S]}_t - \bar\mu| &\leq Ce^{-ct} \E_{\pi^\beta} |\mu^{[\Gamma_0^{-i} ; \beta^S]}_0 - \bar\mu| +\frac{C}{\sqrt{n}} \\
        &\quad + C \Big(\E_{\pi^\beta}\Big[ \sum_{x\in [d]} \mu^{[\Gamma_0^{-i} ; \beta^S]}_{0,x} \big|\Delta_x (U_0(\cdot, \mu^{[\Gamma_0^{-i} ; \beta^S]}_0) - \bar u)\big|^2\Big]\Big)^{1/2}\\
    &\leq Ce^{-ct} \E_{\pi^\beta} |\mu^{[\Gamma_0^{-i} ; \beta^S]}_0 - \bar\mu| +\frac{C}{\sqrt{n}} .
\end{align*} We applied the stationary distribution on the left-hand side above as well, so it is independent of $t$. Passing the limit as $t\to\iy$ then yields \eqref{eqn:b_prop_chaos_0}. When $\beta = \gamma_0$ then, we get \eqref{eqn:pi_thm2} where the limiting result follows by ergodicity of the Markov process $\mu_t^{\Gamma_0}$.

\subsubsection{Approximate Markovian Nash Equilibrium}

Now we prove the main part of the theorem. By a small abuse of notation, we will write $X^{[\Gamma_0^{-i} ; \beta^S], i}$ as $X^{\beta^S, i}$, keeping in mind that the position of the $i$-th player will depend on $\Gamma_0$ implicitly since $\beta^S$ depends on the positions of the other players. Again we rely on $\Cref{lem:freedom}$ for the freedom to choose the initialization of the players. We will see that it is particularly convenient to use $\pi^\beta$, the stationary distribution. By the definition of $J_0^{n,i}$, using that $\pi^\beta$ is the stationary distribution, adding and subtracting a term, and once more using that $\pi^\beta$ is the stationary distribution:

\begin{align}\label{eqn:7} 
\begin{split}
    J^{n,i}_0 ([\Gamma^{-i}_0; &\beta^S]) \\
    &= \limsup_{T\to\iy} \frac{1}{T} \EE_{\pi^\beta} \int_0^T [f(X^{\beta^S, i}_t, \beta^S( \bX^{[\Gamma_0^{-i};\beta^S]}_t)) + F(X^{\beta^S, i}_t,  \mu^{[\Gamma^{-i}_0 ; \beta^S],i}_t)]dt \\
    &= \EE_{\pi^\beta} \Big[f(X^{\beta^S, i}_0, \beta^S( \bX^{[\Gamma_0^{-i};\beta^S]}_0)) + F(X^{\beta^S, i}_0,  \mu^{[\Gamma^{-i}_0 ; \beta^S],i}_0)\Big] \\
    &= \EE_{\pi^\beta} \Big[f(X^{\beta^S, i}_0, \beta^S( \bX^{[\Gamma_0^{-i};\beta^S]}_0)) + F(X^{\beta^S, i}_0, \bar \mu ) + F(X^{\beta^S, i}_0,  \mu^{[\Gamma_0^{-i};\beta^S],i}_0) - F(X^{\beta^S, i}_0, \bar \mu )\Big] \\
    &= \limsup_{T\to\iy} \frac{1}{T}\EE_{\pi^\beta} \int_0^T [f(X^{\beta^S, i}_t, \beta^S( \bX^{[\Gamma_0^{-i};\beta^S]}_t)) + F(X^{\beta^S, i}_t, \bar \mu )]dt \\
        &\qquad+ \EE_{\pi^\beta}  \Big[F(X^{\beta^S, i}_0,  \mu^{[\Gamma_0^{-i};\beta^S],i}_0) - F(X^{\beta^S, i}_0, \bar \mu )\Big].\\
\end{split}
\end{align} 

For the second term in \eqref{eqn:7}, we can use the Lipschitz continuity of $F$, followed by \eqref{eqn:b_prop_chaos_0} to get:
\begin{align}\notag
    \begin{split}
        \Big| \EE_{\pi^\beta} \Big[F(X^{\beta^S, i}_0,  \mu^{[\Gamma_0^{-i};\beta^S],i}_0) - F(X^{\beta^S, i}_0, \bar \mu ) \Big] \Big| &\leq C_{L,F} \EE_{\pi^\beta} | \mu^{[\Gamma_0^{-i};\beta^S],i}_0 - \bar \mu |  \\
        &\leq \frac{C}{\sqrt{n}}.
    \end{split} 
\end{align} 

By the same arguments as in the proof of \Cref{thm:epsilon_equilibrium} then, specifically for \eqref{eqn:3} and \eqref{eqn:cn_nash}, we have that:
\begin{align*}
    J_0^{n,i}(\Gamma_0) - J_0^{n,i} ([\Gamma_0^{-i} ; \beta^S]) \leq \calO(n^{-1/2}).
\end{align*} So far in the proof, we had to assume $\beta^S$ was stationary in order to invoke the stationary distribution. Using \Cref{lem:stat_is_enough} we can now upgrade this result to hold for all $\beta \in \calA^{n,i}$ non-stationary. Recalling $\gamma^W$ from \Cref{lem:stat_is_enough}, and using the prior display with $\beta^S = \gamma^W$: 
\begin{align*}
         J_0^{n,i}(\Gamma_0) - \inf_{\beta \in \calA^{n,i}} J_0^{n,i} ([\Gamma_0^{-i} ; \beta])  = J_0^{n,i}(\Gamma_0) - J_0^{n,i} ([\Gamma_0^{-i} ; \gamma^W]) \leq \calO(n^{-1/2}).
\end{align*} $\square$

\subsection{Proof of \Cref{lem:stat_is_enough}}

The proof proceeds in two main steps. For the first, we establish a unique solution to an auxiliary discounted system. Then, by a vanishing discount argument, we obtain a solution to the auxiliary ergodic system \eqref{eqn:aux_erg}. Finally, by a verification argument---that is, plugging in a carefully chosen process to the mock potential function from \eqref{eqn:aux_erg} and using It\^o's lemma---we get \eqref{eqn:actual_optimal}.

\subsubsection{Existence for a discounted system}

We start by obtaining a solution to an auxiliary discounted system in order to ultimately construct a solution to the ergodic system \eqref{eqn:aux_erg}. We proceed by defining a map $S$ for which any fixed point will solve the discounted system. We prove that for any discount $r$, $S$ is a contraction and therefore there exists a unique solution $W^i_r$ to the auxiliary discounted system. 

Let $r>0$ be the discount factor. We will first prove that there exists $W^i_r : [d]\times \calP^{n-1}([d])\to [0, C_{f+F}/r]$ solving:
\begin{align}\label{eqn:aux_discount}
        rW_r^i(x^i, \mu_{\bx}^{n,i}) &= H(x^i,\Delta_{x^i} W_r^i (\cdot, \mu_{\bx}^{n,i})) + F(x^i,\mu_{\bx}^{n,i}) + \sum_{j\neq i} \gamma^j_0(\bx) \cdot \Delta_{x^j} U_0 (\cdot, \mu_{\bx}^{n,j}). 
\end{align} Define:
\[
\calV := \{ w : [d]^n \to [0,C_{f+F}/r] \mid \exists \phi^w:[d]\times \calP^{n-1}([d]) \to [0,C_{f+F}/r] \text{ s.t. }w(\bx) = \phi^w(x^i, \mu^{n,i}_{\bx})  \}.
\] Define a map $S : \calV \to\calV$ by:
\[
S (w) (\bx) := \inf_{a\in\A^d_{x^i}} R(w, a) (\bx)
\] where:
\begin{align*}
    R(w, a)(\bx) := \frac{1}{r+\sum_{y\neq x^i} a_y} &\Big[\sum_{y,y\neq x^i} a_y \phi^w(y, \mu^{n,i}_{\bx}) + F(x^i, \mu^{n,i}_{\bx}) + f(x^i, a) \\
    &\qquad+ \sum_{j\neq i} \gamma^j_0(\bx) \cdot \Delta_{x^j} U_0 (\cdot, \mu^{n,j}_{\bx})\Big].
\end{align*} Now, 
\begin{align}\label{eqn:S_contraction}
\begin{split}
|(S(w)-S(\tilde w))(\bx)|&\le 
\sup_{a\in\A^d_{x^i}}|(R(w, a) - R(\tilde w, a))(\bx)|\\
&\le \sup_{a\in \A^d_{x^i}}\Big\{\frac{1}{r+\sum_{y\neq x^i} a_y}\Big[\sum_{y,y\ne x^i} a_y|\phi^w(y,\mu^{n,i}_{\bx})-\phi^{\tilde w} (y,\mu^{n,i}_{\bx})| \Big]\Big\}\\
&\leq \sup_{a\in \A^d_{x^i}}\left\{ \frac{\sum_{y\neq x^i} a_y}{r+\sum_{y\neq x^i} a_y}\left[\|w-\tilde w\|_{\sup} \right]\right\}\\
&\le \tilde r \|w-\tilde w\|_{\sup},
\end{split}
\end{align} 
where 
\[
\|w\|_{\sup}:= \max_{\bx\in [d]^n} |w(\bx)|
\] and \[
\tilde{r} := \sup_{x^i \in [d]}\sup_{a \in \A^d_{x^i}} \left\{\frac{\sum_{y\neq x^i} a_y}{r+\sum_{y\neq x^i} a_y} \right\}\in(0,1).
\] Taking the supremum of the left-hand side of \eqref{eqn:S_contraction} over all $\bx$, we obtain that $S$ is a contraction. By Banach fixed point theorem there exists a unique fixed point of $S$, which we call $w^{i}_r$, with $w^i_r(\bx) = W^{i}_r(x^i, \mu^{n,i}_{\bx})$ for some $W^i_r : [d]\times \calP^{n-1}([d]) \to [0,C_{f+F}/r]$ since $w^i_r \in \calV$. Unfurling the definition of $S$ and using that $w^i_r$ is its fixed point, we find that $W^i_r$ satisfies \eqref{eqn:aux_discount}.

\subsubsection{Verification and vanishing discount} 

Now we can make use of the discounted system \eqref{eqn:aux_discount} from the prior step in order to approach \eqref{eqn:aux_erg}. We first confirm that $w^i_r(\bx)$ equals a discounted cost under the strategy profile $\Ups_r$. This gives an alternate form for $w^i_r$ that we will use to show $w^i_r(\cdot) - w^i_r(\hat\bx)$ is bounded in maximum norm, for some $\hat\bx \in [d]^n$, by a stopping time argument. With a uniform bound, we pass to a subsequence of discounts decreasing to zero and show that the discounted system \eqref{eqn:aux_discount} converges to \eqref{eqn:aux_erg}. We then use another verification argument to get \eqref{eqn:actual_optimal}.

For brevity, define $\Ups_r := [\Gamma_0^{-i}; \gamma^{W_r}]$. To get $\vr^i$ and $W^i_0$, we proceed by a vanishing discount argument. To do this, we first verify that $W^i_r$ is the discounted cost:
\[
J_r^{n,i} (\bx, \Ups_r) := \E_{\bx} \int_0^\iy e^{-rt} \left(f(X^{\Ups_r, i}_t, \gamma^{W_r}(\bX^{\Ups_r}_t)) + F(X^{\Ups_r, i}_t, \mu^{\Ups_r, i}_t)\right) dt.
\]

By applying It\^o's lemma to $e^{-rt}w^{i}_r(\bX_t^{\Ups_r})$ and taking the expectations, we get that 
\begin{align}\label{eqn:Ito_dis_n}
&\E_{\bx}\left[e^{-rt} w^{i}_r(\bX_t^{\Ups_r})-w^{i}_r(\bX_0^{\Ups_r})\right]\\\notag
&\quad=
\E_{\bx}\Big[\int_0^t e^{-rs} \Big(-r w^{i}_r(\bX_s^{\Ups_r}) +\sum_{j\neq i} \gamma^j_0(\bX^{\Ups_r}_s) \cdot \Delta_{X^{\Ups_r, j}_s} U_0(\cdot, \mu^{\Ups_r,j}_s) \\\notag
&\qquad\qquad\qquad +\gamma^{W_r}(\bX_s^{\Ups_r})\cdot \Delta_{X^{\Ups_r, i}_s} W^{i}_r(\cdot, \mu^{\Ups_r, i}_s)\Big)ds\Big].
\end{align} 
We use \eqref{eqn:aux_discount} in \eqref{eqn:Ito_dis_n} and rearranging the terms, it follows that 
\begin{align*}
w^{i}_r(\bx) 
&= \E_{\bx}\left[e^{-rt}w^{i}_r(\bX_t^{\Ups_r})\right] + \E_{\bx}\Big[ \int_0^t e^{-rs} \left(f(X^{\Ups_r, i}_s,\gamma^{W_r}(\bX_s^{\Ups_r}))+F(X^{\Ups_r, i}_s, \mu^{\Ups_r, i}_s)\right)ds\Big].
\end{align*} 
Note that $w^i_r$ is bounded. Hence, passing to the limit implies:
\begin{align}\label{eqn:disc_verif}
w^{i}_r(\bx)&= \lim_{t\to\infty}\Big\{\E_{\bx} \Big[e^{-rt} w^{i}_r(\bX_t^{\Ups_r})\Big] + \E_{\bx} \Big[ \int_0^t e^{-rs} \left(f(X^{\Ups_r, i}_s, \gamma^{W_r}(\bX_s^{\Ups_r}))+F(X^{\Ups_r, i}_s, \mu^{\Ups_r, i}_s)\right)ds\Big]\Big\}\\ \notag
&= \E_{\bx}\Big[ \int_0^\infty e^{-rs} \left(f(X^{\Ups_r, i}_s, \gamma^{W_r}(\bX_s^{\Ups_r}))+F(X^{\Ups_r, i}_s, \mu^{\Ups_r, i}_s)\right)ds\Big]\\ \notag
&= J^{n,i}_r(\bx,\Ups_r).
\end{align}

Recall that players choose their rates from the set $[\mathfrak{a}_l, \mathfrak{a}_u]$ where $0<\mathfrak{a}_l<\mathfrak{a}_u <\infty$. Since there are finitely many configurations, every configuration is positive recurrent so there exists $N\in (0,\infty)$, independent of $r$, such that for any arbitrary $\bx, \hat\bx\in [d]^n$, one has $\EE_{\bx}[\hat\tau] \leq N$ where $\hat\tau:=\inf\{t\geq 0\mid \bX_t = \hat\bx\}$. 

Fix an arbitrary $\hat\bx\in [d]^n$. By definition and some arithmetic:
\begin{align*}
    J^{n,i}_r(\bx,\Ups_r) &= \E_{\bx} \Big[ \int_0^{\hat{\tau}} e^{-rt}\left(f(X^{\Ups_r, i}_t, \gamma^{W_r}(\bX_t^{\Ups_r}))+F(X^{\Ups_r, i}_t, \mu^{\Ups_r, i}_t) \right) dt + J^{n,i}_r(\bX_{\hat{\tau}},\Ups_r) \\
    &\qquad\quad - (1-e^{-r\hat{\tau}})J^{n,i}_r(\bX_{\hat{\tau}},\Ups_r)\Big].
\end{align*} By definition of $\hat{\tau}$, we know that $\E[J^{n,i}_r(\bX_{\hat{\tau}},\Ups_r)] = J^{n,i}_r(\hat{\bx},\Ups_r)$ and so by the previous computation and the introduction of absolute value, we see that:
\begin{align*}
    |J^{n,i}_r(\bx,\Ups_r) - J^{n,i}_r(\hat{\bx},\Ups_r)| &\leq \E_{\bx} \Bigg[ \int_0^{\hat{\tau}} e^{-rt}\left|f(X^{\Ups_r, i}_t, \gamma^{W_r}(\bX_t^{\Ups_r}))+F(X^{\Ups_r, i}_t, \mu^{\Ups_r, i}_t) \right| dt \Bigg]\\
    &\qquad + \E_{\bx}[(1-e^{-r\hat{\tau}})|J^{n,i}_r(\bX_{\hat{\tau}},\Ups_r)|].
\end{align*} To deal with the first term after the inequality we note that since $f$ and $F$ are bounded:
\begin{align*}
&\E_{\bx} \Bigg[ \int_0^{\hat{\tau}} e^{-rt}\left|f(X^{\Ups_r, i}_t, \gamma^{W_r}(\bX_t^{\Ups_r}))+  F(X^{\Ups_r, i}_t, \mu^{\Ups_r, i}_t) \right| dt \Bigg] 
\leq \E_{\bx}[\hat{\tau}]C_{f+F}\le NC_{f+F}.
\end{align*} 
For the second term, we use a max bound followed by the inequality $(1-e^{-rz})/r\le z$ for all $r\geq 0$ to see that:
\begin{align*}
    \E_{\bx}[(1-e^{-r\hat{\tau}})|J^{n,i}_r(\bX_{\hat{\tau}},\Ups_r)|] &\leq \E_{\bx}[r^{-1}(1-e^{-r\hat{\tau}})] \max_{\by \in [d]^n} r |J^{n,i}_r(\by,\Ups_r)|\\
    &\leq \E_{\bx}[\hat{\tau}] \max_{\by \in [d]^n} r |J^{n,i}_r(\by,\Ups_r)| \\
    &\leq N C_{f+F},
\end{align*} 
where recall that the solution $w_r^{i}$ to \eqref{eqn:aux_discount} was bounded above by $C_{f+F}/r$ for all configurations $\by\in [d]^n$ and so $J^{n,i}_r(\by,\Ups_r) = w_r^{i}(\by) \leq C_{f+F}r^{-1}$ implies that $r J^{n,i}_r(\by,\Ups_r) \leq C_{f+F}$. This gives the last inequality above. Combining these results:
\[
|w_r^{i} (\bx) - w_r^{i} (\hat{\bx})| \leq 2NC_{f+F}.
\] Since the bound is uniform with respect to $r$, we see that for any sequence $(r_m)_{m\geq 0}\to 0$:
\begin{align*}
\max_{\bx\in [d]^n} |w_{r_m}^{i} (\bx) - w_{r_m}^{i} (\hat{\bx})| \leq 2NC_{f+F}.
\end{align*} 
We note that $\max_{\bx\in [d]^n} | \cdot |$ as above is a norm on the space of functions $[d]^n \to [0,C_{f+F}/r]$ and that the difference $w_{r_m}^{i} (\bx) - w_{r_m}^{i} (\hat{\bx})$ is bounded uniformly in this norm. Therefore there exists a subsequence of $(r_m)_m$, which by abuse of notation we also refer to as $(r_m)_m$, such that $w_{r_m}^{i} (\cdot) - w_{r_m}^{i} (\hat{\bx})$ converges as $r_m\to 0$. Denote this limit by $w^{i}_0(\cdot)$ and since $w^i_{r_m} \in\calV$ there exists $W^i_0$ such that $w^i_0(\bx) = W^i_0(x^i, \mu^{n,i}_{\bx})$. So, for all $\bx\in [d]^n$, $W^i_{r_m}(x^i, \mu_{\bx}^{n,i}) - W^i_{r_m}(\hat x^i, \mu_{\hat \bx}^{n,i}) \to W^i_0(x^i, \mu^{n,i}_{\bx})$ as $m\to\iy$ too. 

We can now apply this vanishing discount to \eqref{eqn:aux_discount} to obtain a solution for \eqref{eqn:aux_erg}. We note that:
\begin{align*}
    \lim_{m\to\infty}\Delta_{x^i} W_{r_m}^{i}(\cdot, \mu^{n,i}_{\bx})  &= \lim_{m\to\infty} \big(W_{r_m}^{i}(y, \mu^{n,i}_{\bx})  -  W_{r_m}^{i}(x^i, \mu^{n,i}_{\bx})\big)_{y \in [d]}\\
    &= \lim_{m\to\infty} \big(W_{r_m}^{i}(y, \mu^{n,i}_{\bx}) -W_{r_m}^{i}(\hat x^i, \mu^{n,i}_{\hat \bx}) + W_{r_m}^{i}(\hat x^i, \mu^{n,i}_{\hat \bx}) -  W_{r_m}^{i}(x^i, \mu^{n,i}_{\bx})\big)_{y \in [d]}\\
    &= \big(W_{0}^{i}(y, \mu^{n,i}_{\bx})  -  W_{0}^{i}(x^i, \mu^{n,i}_{\bx})\big)_{y \in [d]}\\
    &= \Delta_{x^i} W_{0}^{i}(\cdot, \mu^{n,i}_{\bx}).
\end{align*} Again by passing to a subsequence, we can denote $\varrho^{i}$ to be the limit of the real-valued and bounded sequence $\{r_m W_{r_m}^{i} (\hat x^i, \mu_{\hat{\bx}}^{n,i})\}_{m\in\N}$. Once more, recall that $|W_{r_m}^{i}(x^i, \mu_{\bx}^{n,i})- W_{r_m}^{i}(\hat x^i, \mu_{\hat{\bx}}^{n,i})|$ is bounded for any configuration $\bx\in [d]^n$. Now we can apply the continuity of $H$ to obtain that:
\begin{align*}
    \varrho^{i} &=  \lim_{m\to\infty} [r_m W_{r_m}^{i} (\hat x^i, \mu^{n,i}_{\hat\bx}) + r_m(W_{r_m}^{i}(x^i, \mu_{\bx}^{n,i})- W_{r_m}^{i}(\hat x^i, \mu_{\hat{\bx}}^{n,i}))] \\
    &=  \lim_{m\to\infty} r_m W_{r_m}^{i} ( x^i, \mu^{n,i}_{\bx}) \\
    &= \lim_{m\to\infty} \Big[H(x^i,\Delta_{x^i} W^{i}_{r_m}(\cdot, \mu_{\bx}^{n,i})) +  F(x^i,\mu_{\bx}^{n,i}) + \sum_{j\neq i} \gamma^j_0(\bx) \cdot \Delta_{x^j} U_0 (\cdot, \mu_{\bx}^{n,j})\Big]\\
    &= H(x^i,\Delta_{x^i} W^{i}_{0}(\cdot, \mu_{\bx}^{n,i})) +  F(x^i,\mu_{\bx}^{n,i}) + \sum_{j\neq i} \gamma^j_0(\bx) \cdot \Delta_{x^j} U_0 (\cdot, \mu_{\bx}^{n,j}).
\end{align*} 
Hence, we established \eqref{eqn:aux_erg}.

As the final step of this proof, we now turn to prove \eqref{eqn:actual_optimal}. Let $\Ups := [\Gamma_0^{-i} ; \beta]$, $\beta\in\calA^{n,i}$. By another verification-style argument like for \eqref{eqn:disc_verif}, we can show that:
\begin{align*}
\E_{\bx}(w^{i}_0(\bX_0^{\Ups})) &\leq \E_{\bx} \left(w^{i}_0(\bX_T^{\Ups})\right) + \E_{\bx} \left(\int_0^T \left[ f(X_s^{\Ups, i}, \beta(s, \bX_s^{\Ups, i}))+F(X_s^{\Ups, i}, \mu^{\Ups, i}_s)\right]ds\right) -T\varrho^{i},
\end{align*} with equality above when $\beta = \gamma^W$ since $\gamma^W$ is the minimizer in \eqref{eqn:aux_erg}. Dividing by $T$ and passing the $\limsup$ as $T\to \iy$, we get:
\[
\vr^i \leq \E_{\bx} \left(\int_0^T \left[ f(X_s^{\Ups, i}, \beta(s, \bX_s^{\Ups, i}))+F(X_s^{\Ups, i}, \mu^{\Ups, i}_s)\right]ds\right) = J_0^{n,i}(\Ups).
\] Since we have equality when $\beta = \gamma^W$ and inequality otherwise, we conclude \eqref{eqn:actual_optimal}. $\square$

\subsection{Proof of \Cref{lem:miracle}}

The proof is modeled after duality estimates for the MFG, see \cite[Lemma~3.6]{CZ2022}, for example. In that case, the functions in play are solutions to ODEs and there is no randomness. Here, however, we must consider random processes and consequently, the proof is much more involved. Specifically, we will have to deal with three main issues: (1) martingale terms arising from the jump dynamics of the empirical measure, and (2) a deviating player with a potentially different strategy, causing error terms to aggregate, and (3) that because the process we consider is a jump process on $\calP^n([d])$, it may be modified slightly each time it appears, thus accumulating more error. The ultimate goal is to drag the accumulating errors through the proof and into the estimate at the end.

Fix arbitrary $\beta\in\calA^{n,i}$ and $\pi\in\calP([d]^n)$. For brevity, denote the strategy profile $\Ups := [\Gamma_0^{-i} ; \beta]$. Define $dL^\Ups_{yz}(t)$ as the martingale increment:
\[
dL^\Ups_{yz} (t) = dK^\Ups_{yz}(t) - \big[(n-1) \mu_{t,y}^{\Ups, i} \gamma_{z}^*(y,\Delta_y U_0(\cdot, \mu^{\Ups, -y
        }_t))  + \1_{\{X^{\Ups, i}_t = y\}} \beta_{yz}(t,\bX^{\Ups}_t)\big]dt,
\]
where recall that $K^{\bbeta}$ is defined generally in \eqref{eq:mu_beta}.  \textit{We now go through several displays to compute $dU_0(x,\mu^\Ups_t)$.} Recall that $\mu^{\Ups, -x}$ is the empirical distribution of the players, minus one player in position $x$. By It\^o's chain rule, 
\begin{align*}
d U_0 (x, \mu^\Ups_{t}) &= \sum_{y,z\in [d]} \big[ U_0 (x,\mu^\Ups_{t^-} + \tfrac{e_z -e_y}{n}) - U_0 (x, \mu^\Ups_{t^-})\big] dK^\Ups_{yz} (t)
\\
&\small{\text{ [adding and subtracting a rate term to get a martingale]}}
\\
&= \sum_{y,z\in [d]} \Big\{\big[ U_0 (x,\mu^\Ups_t + \tfrac{e_z -e_y}{n}) - U_0 (x, \mu^\Ups_t)\big] 
\\
&\qquad\qquad\qquad \times \big[ (n-1) \mu_{t,y}^{\Ups, i} \gamma_{z}^*(y,\Delta_y U_0(\cdot, \mu^{\Ups, -y
        }_t))  + \1_{\{X^{\Ups, i}_{t} = y\}} \beta_{yz}(t,\bX^{\Ups}_t)\big] dt\Big\}\\
        &\quad + \sum_{y,z\in [d]} \big[ U_0 (x,\mu^\Ups_{t^-} + \tfrac{e_z -e_y}{n}) - U_0 (x, \mu^\Ups_{t^-})\big] dL^\Ups_{yz} (t)
\\
&\small{\text{[adding and subtracting a $D^\eta U_0$ term]}}\\
    &= \sum_{y,z\in [d]} D^\eta_{yz} U_0 (x,\mu^\Ups_t ) \big[  \mu_{t,y}^{\Ups, i} \gamma_{z}^*(y,\Delta_y U_0(\cdot, \mu^{\Ups, -y
        }_t))  + \tfrac1n \1_{\{X^{\Ups, i}_s = y\}} \beta_{yz}(t,\bX^{\Ups}_t)  \big]dt \\
        &\quad + \sum_{y,z\in [d]} \big[ (n-1)((U_0 (x,\mu^\Ups_t + \tfrac{e_z -e_y}{n}) - U_0 (x, \mu^\Ups_t)) - D^\eta_{yz} U_0 (x,\mu^\Ups_t) \big] \\
        &\qquad\qquad \times \big[  \mu_{t,y}^{\Ups, i} \gamma_{z}^*(y,\Delta_y U_0(\cdot, \mu^{\Ups, -y
        }_t))  + \tfrac1n \1_{\{X^{\Ups, i}_s = y\}} \beta_{yz}(t,\bX^{\Ups}_t) \big] dt \\
        &\quad + \sum_{y,z\in [d]} \big[ U_0 (x,\mu^\Ups_{t^-} + \tfrac{e_z -e_y}{n}) - U_0 (x, \mu^\Ups_{t^-})\big] dL^\Ups_{yz} (t) 
\\
&\small{\text{[using Lipschitz continuity of $D^\eta U_0$ and the boundedness of the rates,} }\\
&\small{\text{ we can aggregate terms of order $\calO(n^{-1})$]}}\\
    &= \sum_{y,z\in [d]} D^\eta_{yz} U_0 (x,\mu^\Ups_t ) \mu_{t,y}^{\Ups, i} \gamma_{z}^*(y,\Delta_y U_0(\cdot, \mu^{\Ups, -y
        }_t)) dt \\
        &\quad + \sum_{y,z\in [d]} \big[ U_0 (x,\mu^\Ups_{t^-} + \tfrac{e_z -e_y}{n}) - U_0 (x, \mu^\Ups_{t^-})\big] dL^\Ups_{yz} (t) \\ 
        &\quad + \calO(n^{-1}) 
        \\
        &
 \small{\text{[From the Lipschitz continuity of $\gamma^*$ and $U_0$ as well as the facts that }}\\
 &\small{\text{ $|\mu^{\Ups, i}_{t,y} - \mu^{\Ups}_{t,y}| \leq C/n$, $\qquad | \mu^{\Ups, -y}_t - \mu^{\Ups}_{t}| \leq C/n$,}} \\
 &\small{\text{ we can aggregate more terms into $\calO(n^{-1})$]}}\\
    &= \sum_{y,z\in [d]} D^\eta_{yz} U_0 (x,\mu^\Ups_t ) \mu_{t,y}^{\Ups} \gamma_{z}^*(y,\Delta_y U_0(\cdot, \mu^{\Ups
        }_t)) dt \\
        &\quad + \sum_{y,z\in [d]} \big[ U_0 (x,\mu^\Ups_{t^-} + \tfrac{e_z -e_y}{n}) - U_0 (x, \mu^\Ups_{t^-})\big] dL^\Ups_{yz} (t) \\ 
        &\quad + \calO(n^{-1}) .
\end{align*} 

And finally by \eqref{ME}: 
\begin{align}\label{eqn:6} 
    \begin{split}
        dU_0(x,\mu^\Ups_{t}) &= (\vr - H(x,\Delta_x U_0(\cdot, \mu^{\Ups}_t)) - F(x, \mu^\Ups_t) ) dt \\
        &\quad + \sum_{y,z\in [d]} \big[ U_0 (x,\mu^\Ups_{t^-} + \tfrac{e_z -e_y}{n}) - U_0 (x, \mu^\Ups_{t^-})\big] dL^\Ups_{yz} (t) \\ 
        &\quad + \calO(n^{-1}),
    \end{split}
\end{align} Set the processes:
\begin{align*}
v_x(t) := U_0(x,\mu^{\Ups}_t) - \check u_x(t) \qquad\text{and}\qquad m_x(t):=\mu^{\Ups}_{t,x} - \check\mu_x(t),\qquad (t,x)\in\R_+\times[d].
\end{align*}  
Now, an application of It\^o's product rule implies that:
\begin{align*}
    m(T)\cdot v(T)-m(0)\cdot v(0)
    &= \int_{0}^{T}  \sum_{x\in [d]} d[m_x(t) v_x(t)]\\\notag
    &= \int_{0}^{T} \sum_{x\in [d]} \left(m_x(t)dv_x(t) + v_x(t)dm_x(t)+dm_x(t)dv_x(t)\right).
\end{align*} 
We will treat the three terms in the integrand of the previous display separately. Taking expectations on both sides, using \eqref{erg_MFG_t}, \eqref{eq:mu_beta}, \eqref{eq:K}, and that $L^\Upsilon_{yz}$ is a martingale, we get (with more detailed explanations to follow),
\begin{align}\label{eq:m_v_expansion}
    \E_\pi [m(T)\cdot v(T)-m(0)\cdot v(0)]
    &=\sum_{i=1}^3 \Psi_i(T) + \calO(n^{-1}), 
\end{align}
where the $\calO(n^{-1})$ term comes from \eqref{eqn:6}, where:
\begin{align*}
    &\Psi_1 (T) := \E_\pi \int_0^T \sum_{x\in [d]} m_x(t) \big(\vr - H(x,\Delta_x U_0 (\cdot, \mu^{\Ups}_t)) - F(x,\mu^\Ups_t) -\check \vr + H(x,\Delta_x \check u(t)) + F(x,\check \mu(t))\big)dt ,\\
    &\Psi_2(T) := \E_\pi \int_0^T \sum_{x\in [d]} v_x(t) \Big[\sum_{y,z\in [d]} \tfrac{(e_z - e_y)_x}{n} \big[(n-1) \mu^{\Ups, i}_{t,y} \gamma^*_z(y,\Delta_y U_0(\cdot, \mu^{\Ups, -y}_t)) + \1_{\{X^{\Ups, i}_t = y\}}\beta_{yz}(t,\bX^\Ups_t) \big]\\
    &\qquad\qquad\qquad\qquad\qquad\qquad -\sum_{y\in [d]} \check \mu_y(t) \check \gamma_{yx}(t)\Big]dt,\\
    &\Psi_3 (T) := \E_\pi \int_0^T \sum_{x\in [d]} \sum_{y,z \in [d]} ( U_0(x,\mu^\Ups_{t^-} + \tfrac{e_z - e_y}{n}) - U_0(x,\mu^\Ups_{t^-}) )\tfrac{(e_z - e_y)_x}{n} dK^\Ups_{yz} (t) \\
    &\qquad\quad = \E_\pi \int_0^T \sum_{x\in [d]} \sum_{y,z \in [d]} \Big\{( U_0(x,\mu^\Ups_t + \tfrac{e_z - e_y}{n}) - U_0(x,\mu^\Ups_t) )\tfrac{(e_z - e_y)_x}{n}  \\
    &\qquad\qquad\qquad\qquad\qquad\qquad \times \big[ (n-1) \mu^{\Ups,i}_{t,y} \gamma^*_z(y,\Delta_y U_0(\cdot, \mu^{\Ups, -y}_t)) + \1_{\{X^{\Ups, i}_t = y\}} \beta_{yz}(t,\bX^\Ups_t)   \big] \Big\}dt.
\end{align*} We note that we also used \eqref{erg_MFG_t} to substitute in for $\tfrac{d}{dt}\check u(t)$ and $\tfrac{d}{dt}\check \mu(t)$ in $\Psi_1$ and $\Psi_2$, respectively. Moreover, to derive the first expression for $\Psi_3$ we used the facts that the quadratic variation of a compensated random Poisson measure is the non-compensated measure and that $dK_{xy}^\Ups(t) dK_{wz}^\Ups (t) = 1$ if and only if $(x,y)=(w,z)$ and otherwise $dK_{xy}^\Ups(t) dK_{wz}^\Ups (t) = 0$. We can use the Lipschitz continuity of $U_0$ and the boundedness of $\gamma^*$ and $\beta$ to get that:
\begin{equation}\label{eqn:psi3}
    |\Psi_3(T)| \leq CTn^{-1}. 
\end{equation} We can use $\sum_xm_x(t)=0$ in the integrand of $\Psi_1(T)$ and then the monotonicity of $F$ to get:
\begin{align}\label{eqn:psi1}
\begin{split}
    \Psi_1(T) &= \E_\pi \int_0^T \sum_{x\in [d]} m_x(t) \big(H(x,\Delta_x \check u(t)) - H(x,\Delta_x U_0 (\cdot, \mu^{\Ups}_t)) + F(x,\check\mu(t)) - F(x,\mu^\Ups_t) \big)dt  \\ 
    &\leq \E_\pi \int_0^T \sum_{x\in [d]} m_x(t) \big(H(x,\Delta_x \check u(t)) - H(x,\Delta_x U_0 (\cdot, \mu^{\Ups}_t)) \big)dt.  
\end{split}
\end{align} For $\Psi_2(T)$, we will deal with the term on the first line for the moment, and return to the full expression after. We can use the definition of $(e_z-e_y)_x$ first to separate into two sums, the fact that $\sum_{y\in[d]}\gamma^*_y(x,\cdot)=0$, recombine the sums into one, the Lipschitz continuity of $\gamma^*$ and $U_0$ as well as the fact that $((n-1)\mu^{\Ups, i}_{t,y} + \delta_{X^{\Ups, i}_t}(y))/n = \mu^{\Ups}_{t,y}$ to aggregate more terms of order $\calO(n^{-1})$, and get that it equals:
\begin{align*}
    \begin{split}
        &\E_\pi \int_0^T \sum_{x\in [d]} v_x(t) \Big[ \sum_{y, y\neq x} \frac{1}{n} \big[(n-1) \mu^{\Ups,i}_{t,y} \gamma^*_x(y,\Delta_y U_0(\cdot, \mu_t^{\Ups, -y}))  + \1_{\{X^{\Ups, i}_t = y\}}\beta_{yx}(t,\bX^\Ups_t) \big]  \Big] dt \\
            &\quad + \E_\pi \int_0^T \sum_{x\in [d]} v_x(t) \Big[ \sum_{z, z\neq x} \frac{-1}{n} \big[(n-1) \mu^{\Ups,i}_{t,x} \gamma^*_z(x,\Delta_x U_0(\cdot, \mu_t^{\Ups, -x})) + \1_{\{X^{\Ups, i}_t = x\}}\beta_{xz}(t,\bX^\Ups_t) \big]  \Big] dt \\
        &= \E_\pi \int_0^T \sum_{x\in [d]} v_x(t) \Big[ \sum_{y, y\neq x} \frac{1}{n} \big[(n-1) \mu^{\Ups,i}_{t,y} \gamma^*_x(y,\Delta_y U_0(\cdot, \mu_t^{\Ups, -y}))  + \1_{\{X^{\Ups, i}_t = y\}}\beta_{yx}(t,\bX^\Ups_t) \big]  \Big] dt \\
            &\quad + \E_\pi \int_0^T \sum_{x\in [d]} v_x(t) \Big[ \tfrac{(n-1)\mu^{\Ups,i}_{t,x}}{n} \big[ \gamma^*_x(x,\Delta_x U_0(\cdot, \mu_t^{\Ups, -x})) + \tfrac{1}{n-1}\1_{\{X^{\Ups, i}_t = x\}} \beta_{xx}(t,\bX^\Ups_t) \big]  \Big] dt \\
        &= \E_\pi \int_0^T \sum_{x\in [d]} v_x(t) \Big[ \sum_{y \in [d]} \tfrac{n-1}{n} \big[  \mu^{\Ups,i}_{t,y} \gamma^*_x(y,\Delta_y U_0(\cdot, \mu_t^{\Ups, -y}))  + \calO(n^{-1}) \big]  \Big] dt \\
        &= \E_\pi \int_0^T \sum_{x\in [d]} v_x(t) \Big[ \sum_{y \in [d]}  \big[  \mu^{\Ups}_{t,y} \gamma^*_x(y,\Delta_y U_0(\cdot, \mu_t^{\Ups}))  + \calO(n^{-1}) \big]  \Big] dt \\
    \end{split}
\end{align*} Therefore, returning to the definition of $\Psi_2$, we have:
\begin{align}\label{eqn:psi2}
    \Psi_2(T) &= \E_\pi \int_0^T \sum_{x\in [d]} v_x(t) \Big[ \sum_{y \in [d]}  \big[  \mu^{\Ups}_{t,y} \gamma^*_x(y,\Delta_y U_0(\cdot, \mu_t^{\Ups})) - \check\mu_y(t) \check\gamma_{yx}(t)  + \calO(n^{-1}) \big]  \Big] dt .
\end{align} Using \eqref{eqn:psi1}, \eqref{eqn:psi2}, \eqref{eqn:psi3} in \eqref{eq:m_v_expansion}:

\begin{align}\label{eqn:new_form}
\begin{split}
    \EE_\pi [m(T)\cdot v(T) - m(0)\cdot v(0)] 
    &\leq \hat \EE_\pi \int_0^T \Big\{\sum_{x\in [d]} [H(x,\Delta_x \check u(t)) - H(x,\Delta_x U_0 (\cdot,\mu^\Ups_t))]m_x(t) \\
    &\qquad + v_x(t) \Big[\sum_{y\in [d]} \mu^{\Ups}_{t,y} \gamma^*_x(y,\Delta_y U_0(\cdot, \mu_t^{\Ups})) - \check \mu_y(t) \check\gamma_{yx} (t) \Big]\Big\}dt \\
    &\qquad + T\calO(n^{-1}) .
\end{split}
\end{align} Next, utilizing the identity:
$$\sum_{x\in [d]} \check \gamma_{yx}(t)=\sum_{x\in [d]} \gamma^*_x(y,\Delta_y  U_0(\cdot,\mu^\Ups_t))=0,$$ 
we get that:
\begin{align}\label{eqn:expansion_H}
\begin{split}
    &\sum_{x\in[d]} \Big[\left\{H(x,\Delta_x \check u(t))- H(x,\Delta_x U_0(\cdot,\mu^\Ups_t)) \right\} m_x(t)\\
    &\qquad\qquad\qquad+ v_x(t) \sum_{y\in [d]}\{ \mu^{\Ups}_{t,y} \gamma^*_x(y,\Delta_y U_0(\cdot, \mu_t^{\Ups})) - \check\mu_y(t) \check\gamma_{yx}(t) \} \Big]\\
    &\quad=
    \sum_{x\in[d]}\Big[
    \left\{ H(x,\Delta_x \check u(t))- H(x,\Delta_x U_0(\cdot,\mu^\Ups_t))+\Delta_x v(t) \cdot \gamma^*(x,\Delta_x U_0(\cdot,\mu^\Ups_t)))\right\} \mu^{\Ups}_{t,x} \Big]\\
    &\quad\qquad\;\;\,
    + 
    \sum_{x\in[d]}
    \Big[\left\{ H(x,\Delta_x U_0(\cdot,\mu^\Ups_t))- H(x,\Delta_x \check u(t)) - \Delta_x v(t) \cdot\gamma^*(x,\Delta_x \check u(t)))\right\} \check \mu_x(t) \Big].
    \end{split}
\end{align} 
Using \eqref{gamma_H} and assumption ($B_2$), we obtain the inequalities:
\begin{align}\notag
H(x,\Delta_x \check u(t))- H(x,\Delta_x U_0(\cdot,\mu^\Ups_t)) + \Delta_x v(t) \cdot \gamma^*(x,\Delta_x U_0(\cdot,\mu^\Ups_t)))
 &\leq -C_{2,H} |\Delta_x v(t)|^2,
\\\notag
H(x,\Delta_x  U_0(\cdot, \mu^\Ups_t))- H(x,\Delta_x \check u(t)) - \Delta_x v(t) \cdot\gamma^*(x,\Delta_x \check u(t)))
 &\leq -C_{2,H} |\Delta_x v(t)|^2.
\end{align} 
The last two bounds applied to the expansion in \eqref{eqn:expansion_H} imply:
\begin{align}\label{eqn:bound_H}
\begin{split}
    &\sum_{x\in[d]} \Big[\left\{H(x,\Delta_x \check u(t))- H(x,\Delta_x U_0(\cdot,\mu^\Ups_t)) \right\} m_x(t)\\
    &\qquad\qquad\qquad + v_x(t) \sum_{y\in [d]}\{ \mu^{\Ups}_{t,y} \gamma^*_x(y,\Delta_y U_0(\cdot, \mu_t^{\Ups})) - \check\mu_y(t) \check\gamma_{yx}(t) \} \Big]\\
    &\qquad\qquad\qquad\qquad\qquad\qquad\qquad\qquad \leq -C_{2,H}
    \sum_{x\in[d]} (\mu_{t,x}^\Ups + \check\mu_x(t)) |\Delta_x v (t)|^2.
    \end{split}
\end{align} Plugging in \eqref{eqn:bound_H} into \eqref{eqn:new_form}, moving the negated terms to the left-hand side,
\begin{align*}
    \begin{split}
    C_{2,H} \EE_\pi \int_0^T \sum_{x\in [d]} (\mu_{t,x}^\Ups + \check\mu_x(t)) |\Delta_x v (t)|^2 dt &\leq T \calO(n^{-1}) + \E_\pi [m(0) \cdot v(0) - m(T) \cdot v(T)].
    \end{split}
\end{align*} Dividing the prior display by $T$, taking $\limsup_{T\to\infty}$, and using the boundedness of $m$ and $v$ finishes the proof. $\square$

\section{Proofs of \Cref{prop:erg_stability} and \Cref{thm:large_deviations}}\label{sec:ld_proofs}

We finish the main theoretical portion of the paper by proving the results pertaining to the large deviations of the empirical measures.

\subsection{Proof of \Cref{prop:erg_stability}}
    Consider a globally asymptotically stable equilibrium for \eqref{erg_FP_t}, call it $\overset{\circ}{\mu}$. Define $\overset{\circ}{u} := (U_0 (x,\overset{\circ}{\mu}))_{x\in [d]}$ 
    and note that since $\overset{\circ}{\mu}$ is globally asymptotically stable, we have by taking a limit of \eqref{erg_MFG_t} as $t\to\infty$ that:
    \[
    \bar\vr = H(x,\Delta_x \overset{\circ}{u}) + F(x,\overset{\circ}{\mu}).
    \] Hence, $(\bar\vr, \overset{\circ}{u}, \overset{\circ}{\mu})$ solves \eqref{erg_MFG}. Yet by \Cref{prop:erg_MFG_red_full}, \eqref{erg_MFG} has a unique solution, where recall that $\bar u$ is unique up to an additive constant vector $k\vec{1}$, and therefore for all $x\in [d]$, $\Delta_x \overset{\circ}{u} = \Delta_x \bar u$ with $\overset{\circ}{\mu} = \bar\mu$ in the first place. 

    By definition, $\Delta_x u_0(t,\cdot) = \Delta_x U_0(\cdot, \mu_0(t))$. And since $U_0$ is Lipschitz by \Cref{prop:me_properties}, there exists $C>0$, independent of $t$, such that:
    \[
    |\Delta_x U_0 (\cdot, \mu_0(t))| \leq C (|\Delta_x \bar u| + |\mu_0(t) - \bar\mu|),
    \] which means $|\Delta_x u_0(t,\cdot)|$ is bounded. Therefore, we may re-use the proof from \cite[Theorem~3]{GomesMohrSouza_discrete} to obtain that $\mu_0(t) \to \bar\mu$ as $t\to\iy$.\footnote{It is worth mentioning that this theorem relies on the additional contraction assumptions of that work. However, their proof only uses the contractive assumptions to bound their potential functions $u_0(t)$ in a sub-additive (non-norm) function $\|\cdot\|_\sharp$. Here, $\|u_0(t)\|_\sharp$ measures the maximum difference between potentials of any two states from $[d]$ at time $t\in\R_+$ and so is equivalent to $|\Delta u_0(t)|$ in our case. Namely, for any measurable function $b:\R_+ \to \R^d$, the quantity $\|b(t)\|_\sharp$ is bounded for all $t\in\R_+$ if and only if $|\Delta b(t)|$ is bounded for all $t\in\R_+$. } Recalling that the initial condition $\eta$ was arbitrary, we have that $\bar\mu$ is a globally asymptotically stable equilibrium. $\square$

\subsection{Proof of \Cref{thm:large_deviations}}
    The proof is an application of theorems from \cite{MR3354770}, so we verify the assumptions of the theorems. Namely, the possible transitions on $[d]$ form a complete, finite graph and a fortiori such a graph is irreducible. In either case for $\kappa^n$, the rate mapping is Lipschitz because $\gamma^*$ is Lipschitz by assumption, and $U_0$ is Lipschitz, by \Cref{prop:me_properties}. Moreover, the rates are assumed to be bounded below by $\mathfrak{a}_l>0$. \Cref{prop:erg_stability} verifies the final requirement for each of the following theorems. So, \cite[Theorem~2.1]{MR3354770} proves the large deviation result for $\scrL^n_\iy$; \cite[Theorem~3.2]{MR3354770} for $\scrL^n_{[0,T]}$; and \cite[Theorem~3.3]{MR3354770} for $\scrL^n_T$. As a small comment, the set of paths over which the infimum in the good rate function $s(\cdot)$ is taken is nonempty due to \cite[Lemma~4.2]{MR3354770}. $\square$

\section{Numerics Comparing ($C/\sqrt{n}$)-Nash Equilibria}\label{sec:numerics}

By \Cref{cor:stationary_eq} and \Cref{thm:epsilon_equilibrium_2}, we saw that there are, under the full assumptions, at least two distinct types of ($C/\sqrt{n}$)-Nash equilibria: one driven by the master equation's solution, $\Gamma_0$, and one driven by the solution to the stationary ergodic MFG, $\bar\Gamma$. One might ask whether, qualitatively speaking, players adhering to the approximate Nash equilibria profiles $\bar\Gamma$ or $\Gamma_0$ experience different costs, or player behaviors. In \Cref{sec:ergodic_example}, we solve a class of quadratic examples for the stationary ergodic MFG system \eqref{erg_MFG}. This allows us to explicitly state the strategies for the stationary Markovian strategy profile $\bar\Gamma$ in our numerics. To approximate the strategies making up $\Gamma_0$ however, no such explicit formula is available due to the dependence on $U_0$. To tackle this, the following \Cref{sec:erg_me} adapts the deep Galerkin method to approximate the solution to \eqref{ME}. We then begin our comparison of the strategies forming $\bar\Gamma$ and $\Gamma_0$, compare their respective stationary distributions for each empirical measure, and contrast the realized costs under each regime.

\subsection{Explicit solutions to the stationary ergodic MFG}\label{sec:ergodic_example}

The example of quadratic costs is ubiquitous in MFG literature and will serve as our benchmark problem for the numerics that follow this section. Here, we will explicitly solve the example problem in the finite-state ergodic case. This particular example was studied for the finite-horizon, finite-state case in \cite[Example~1]{cec-pel2019}, \cite[Example~3.1]{bay-coh2019}, and \cite[Example~7.1]{CLZ2024}.

Fix a parameter $\delta \geq 0$ that will serve as a bias to avoid state $1\in [d]$ and fix a scaling for the rate costs $b>0$. We will consider an example with the parameters:
\begin{align*}
    &\A = [1,3], \qquad d = 2, \qquad f(x,a) = \delta \1_{\{x=1\}} +  b\sum_{y\neq x} (a_{xy} - 2)^2. 
\end{align*} Pausing to consider the form of $f$, we note that we have introduced two parameters $\delta$ and $b$ whose interplay will determine the equilibrium state of the game. The larger $\delta$ is, the more players will be incentivized to move quickly away from state $1$ to avoid the additional cost of remaining in that state. Meanwhile, increasing $b$ means that players who want to move faster or slower than the ``free" rate of $2$, to avoid state $1$ for instance, will pay a higher cost. \blue{For $p_y \in [-2b,2b], y\neq x$, we can simply express the Hamiltonian $H$ and its minimizer $\gamma^*$ as: 
\begin{align*}
        &H(x,p) = \delta\1_{\{x=1\}} + \sum_{y\neq x} \Big(2p_y - \frac{p_y^2}{4b}  \Big), \qquad     \gamma^*_y (x,p) = 2 - \frac{p_y}{2b}. 
\end{align*} 
Otherwise, we get that the minimizer $\gamma^*_y(x,p)$ obtains its value at one of the endpoints of the interval $[1,3]$. And so whenever $|\bar u_2 - \bar u_1| \leq 2b$,
\begin{align*}
    H(x,\Delta_x \bar u)= \delta\1_{\{x=1\}} + \sum_{y\neq x} \Big(2(\bar u_y - \bar u_x) - \frac{(\bar u_y - \bar u_x)^2}{4b}  \Big), \qquad \gamma^*_y(x,\Delta_x \bar u) = 2 + \frac{\bar u_x - \bar u_y}{2b}.
\end{align*}} 
To accommodate the full assumptions, we will consider a typical Lasry--Lions monotone mean field cost:
\[
F(x,\eta) = \eta_x.
\] \blue{For the moment, consider the case that $|\bar u_2 - \bar u_1| \leq 2b$. Specifying $b$ and $\delta$, we then compute a solution and by \Cref{prop:erg_MFG_red_full} know that this solution is unique.} With this setup, \eqref{erg_MFG} reads:
\begin{align}\label{eqn:stationary_example}
\begin{split}
    &\bar\vr = 2 (\bar u_2 - \bar u_1) - \frac{1}{4b}  (\bar u_2 - \bar u_1)^2 + \bar \mu_1 + \delta, \\
    &\bar\vr = 2 (\bar u_1 - \bar u_2) - \frac{1}{4b}  (\bar u_2 - \bar u_1)^2 + \bar \mu_2, \\
    &\bar\mu_1 (-\phi_1) + \bar\mu_2 \phi_2 = 0,\\
    &\bar\mu_1 \phi_1 + \bar\mu_2 (-\phi_2) = 0,\\
\end{split}
\end{align} where:
\begin{align}\label{eqn:phi_defs}
    \begin{split}
        \phi_1 := 2 + \frac{\bar u_1 - \bar u_2}{2b}, \quad \phi_2 := 2 + \frac{\bar u_2 - \bar u_1}{2b}.\\
    \end{split}
\end{align} We solve this system of equations explicitly. Using the fact that $\bar\mu_1 = 1-\bar\mu_2$ in the measure equations of \eqref{eqn:stationary_example}, we can show that:
\begin{align*}
    -\phi_1 + \bar\mu_2 (\phi_1 + \phi_2) =0 \implies \bar\mu_2 = \phi_1/4,
\end{align*} where the implication follows from \eqref{eqn:phi_defs}. For notational simplicity, define:
\[
A:= \bar u_2 - \bar u_1.
\] We can write the value equations from \eqref{eqn:stationary_example} as:
\begin{align*}
    &\vr = 2A -\frac{1}{4b}A^2 + \bar\mu_1 +\delta \\
    &\vr = -2A -\frac{1}{4b}A^2 + \bar\mu_2.
\end{align*} Subtracting the equations in the prior display, 
\[
0 = 4A + \bar\mu_1 - \bar\mu_2 + \delta.
\] Since $\bar\mu_2 = \phi_1/4$ and $\bar\mu_1 = 1-\bar\mu_2$, 
\[
4A + \frac{A}{4b} +\delta = 0,
\] and so rearranging gives:
\[
A = \delta \frac{4b}{16b+1}. 
\] And since we can write $\bar\mu_1$ and $\bar\mu_2$ in terms of $A$, we plug in the prior display to get:
\begin{align*}
    \bar\mu_1 = \frac{1}{2} \Big(1 - \frac{\delta}{16b+1} \Big), \qquad \bar\mu_2 = \frac{1}{2} \Big(1 + \frac{\delta}{16b +1} \Big).
\end{align*} With this, we can return to the second value equation of \eqref{eqn:stationary_example} to get an expression for $\vr$ in terms of $\delta$ and $b$,
\[
\vr = \frac{1}{2} + \delta \frac{128 b^2  +24b - 4b\delta + 1}{(16b+1)^2} .
\] Suppose $b=4$, and $\delta =0$ for instance. Then the solution reads:
\begin{align*}
    &\bar\vr = 1/2, \quad \bar u_2 - \bar u_1 = 0, \quad \bar\mu = (1/2, 1/2).
\end{align*} When we take $\delta>0$ we will no longer have symmetry in the stationary measure and in the potential function. For example, take $\delta = 1$ and $b=4$ to get:
\begin{align*}
    \bar\vr = 8483/8450 \approx 1.003,\quad \bar u_2 - \bar u_1 = 16/65,\quad \bar\mu = (32/65, 33/65).
\end{align*} For the remainder of the paper, we will study the case when $b=4$ and $\delta=0$. \blue{In either case, $|\bar u_2 - \bar u_1| \leq 2b$.}

\subsection{Approximate solutions to the ergodic master equation}\label{sec:erg_me}

In this section, we use the DGM to solve the ergodic master equation and in order to numerically compare the strategies. The main difference we find is that the master equation-derived strategy is sensitive to the players' empirical distribution in the simplex and that this effect is most pronounced at the boundary. In a sense, this means that the master equation-derived strategy has a better short-term response than the stationary strategy. For the master equation's corresponding empirical $n$-player measure then, we would expect a faster transition away from the boundary of the simplex compared to the measure of players using the stationary strategy.

We solve the master equation corresponding to this data for $U_0$ using the DGM. While the solution of the master equation is a pair $(\vr, U_0)$, we note that since $\vr = \bar\vr$ by \Cref{prop:me_properties}, we have already found that $\vr=1/2$ in \Cref{sec:ergodic_example}. In \Cref{alg:ergodic_DGM}, we present an algorithm for approximating the potential function $U_0$ of the ergodic master equation, given the game's value $\vr$.

\Cref{alg:ergodic_DGM} is, in a sense, an extension of the DGM method for the stationary ergodic MFG, presented in \cite[Section~4]{MR4264647} for the MFG on the torus. Setting aside the different state-spaces, we note that \cite{MR4264647} approximates the stationary ergodic MFG system while \Cref{alg:ergodic_DGM} approximates the ergodic master equation. Moreover, our \Cref{alg:ergodic_DGM} proceeds with $\vr$ as a known value which we solved for a priori. We showed in \Cref{sec:ergodic_example} an explicit expression for $\vr$ in the case of quadratic (and potentially asymmetric) costs. In general, to compute $\vr$, one would first solve \eqref{erg_MFG} by using \cite{MR4264647} and then using the value for $\vr$ determined, solve for $U_0$ using \Cref{alg:ergodic_DGM}.  

As in the DGM, \Cref{alg:ergodic_DGM} uses a family of parameterized functions $\calU(\cdot, \cdot ; \theta) : [d]\times\calP([d])\to\R$ to approximate $U_0$, where $\theta$ is a vector of the parameters in the neural network $\calU$. For any run of the algorithm, we fix the architecture of $\calU$ beforehand; here we use a fully-connected, feed-forward architecture like in \cite{CLZ2024}, while \cite{MR3874585} used an LSTM architecture. Regardless of the architecture choice, we need to consider a class of neural networks that is rich enough to be able to approximate $U_0$. In our case, we used a smoothed ReLU activation at each layer, ELU. We train the parameter $\theta$ using stochastic gradient descent so that $\calU(\cdot, \cdot;\theta)$ fits \eqref{ME} as well as possible.

\begin{algorithm}
\caption{Ergodic DGM}\label{alg:ergodic_DGM}
\begin{algorithmic}[1]
\State {\bf Input:} 
An initial parameter vector $\theta$. 
\State {\bf Output:} A trained vector $\hat \theta$ such that $(\vr, \calU(\cdot, \cdot; \hat\theta))$ approximately solves \eqref{ME}.

\State Compute:
\[
\hat\theta \in \argmin_{\theta \in \R^{\bar\delta}} L(\theta)
\]  where for a finite set of samples $\calS := \{S \mid S = (x,\eta)\}$, we define $L$ to be the mean squared error:
\begin{align*}
\begin{split}
    L(\theta) := \frac{1}{|\calS|} \sum_{(x,\eta) \in \calS} &\Big\{ \Big|H(x,\Delta_x \calU(\cdot, \eta ;\theta)) + \eta_x + \sum_{y,z \in [d]} \eta_y D^\eta_{yz} \calU(x,\eta ; \theta) \gamma^*_z(y,\Delta_y \calU(\cdot, \eta ; \theta)) - \vr \Big|^2 \\
    &\qquad + \Big| \sum_{y\in [d]} \calU(y,\eta ; \theta) \Big|^2 \Big\}. 
\end{split}
\end{align*}
\end{algorithmic}
\end{algorithm}

\begin{remark}
    The first term in the loss function of \Cref{alg:ergodic_DGM} is the master equation evaluated at uniformly sampled points of $[d]\times \calP([d])$. The second term, however, centralizes the potential function around $0$. As noted in \Cref{prop:me_properties}, $U_0$ is only unique up to an additive constant. Therefore, adding the centralization term encourages the algorithm to reach the same solution on different runs.
\end{remark}

\begin{remark}
    Unfortunately, the master equation is a different kind of PDE than those studied in \cite{MR3874585}, and so the theoretical guarantees of the algorithm's correctness do not apply. While we do not offer a proof of correctness, applying the DGM to the finite-horizon master equation was pioneered in \cite[Algorithm~7]{MR4368188} and proofs of algorithmic correctness for that case were offered in \cite{CLZ2024}. 
\end{remark}

In Figure \ref{fig:rate_diff_2}, we present the difference in strategies, where the strategy $\gamma^*_y(x,\Delta_x U_0(\cdot, \eta))$ is approximated by plugging in the result of the ergodic DGM \Cref{alg:ergodic_DGM}; namely, $\gamma^*_y(x,\Delta_x \calU (\cdot, \eta; \hat\theta))$. In the case we study numerically in \Cref{sec:ergodic_example}, for any $D\in\R^d$, $\gamma_y^*(x,\Delta_x D)$ is a linear transformation of the vector $\Delta_x D$. We note that $\Gamma_0$ has a greater effect near the boundary of the simplex; this has the effect of moving a representative player faster toward the center.

\begin{figure}
\begin{tabular}{cc}
\includegraphics[width=75mm]{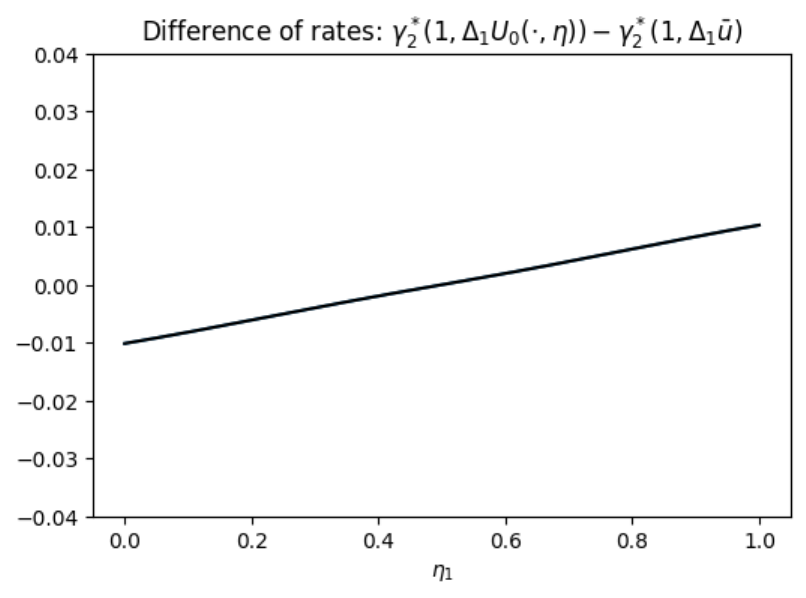}  & \includegraphics[width=75mm]{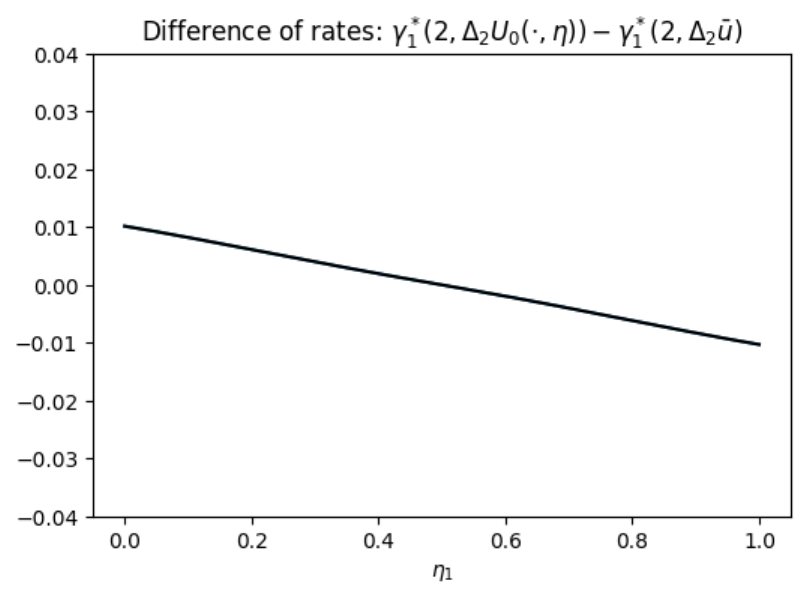} \\
\end{tabular}
\caption{We compare $\Gamma_0$ and $\bar\Gamma$ by the difference in rates $\gamma^*_y(x,\Delta_x U_0(\cdot,\eta)) - \gamma^*_y(x,\Delta_x \bar u)$ in the case $d=2$.}
\label{fig:rate_diff_2}
\end{figure}

We perform the same analysis for the case $d=3$ and observe the differences in the strategies in Figure \ref{fig:rate_diff_3}. The vertical axis represents the proportion of players in state $1$ and the horizontal the proportion in state $2$. The proportion of players in state $3$ is $1$ in the lower, left corner of the plots. Like in Figure \ref{fig:rate_diff_2}, we note that $\Gamma_0$ results in a greater rate of transition when more players are in one's own state.

\begin{figure}
\begin{tabular}{ccc}
\includegraphics[width=50mm]{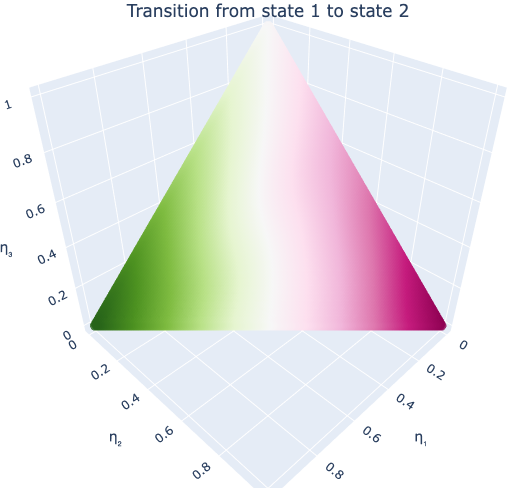}  & \includegraphics[width=50mm]{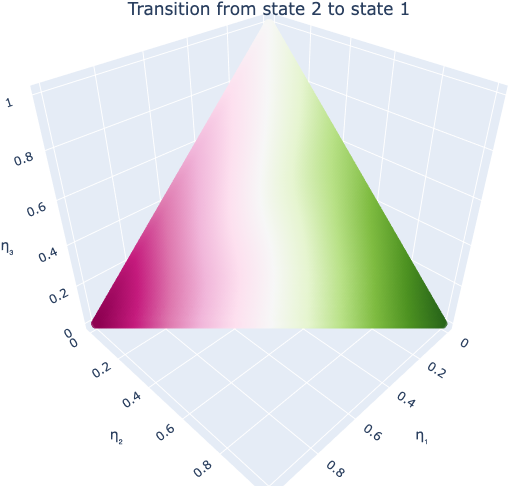} &
\includegraphics[width=50mm]{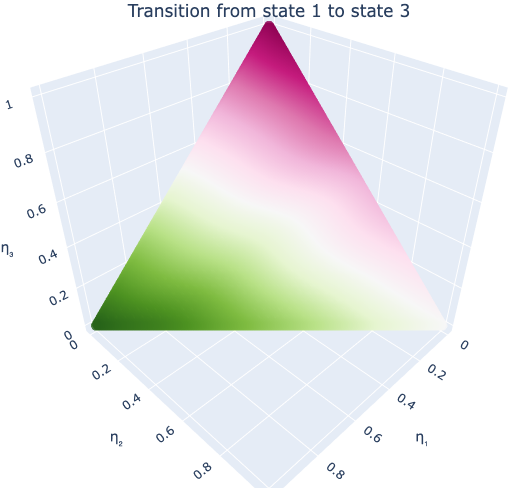} \\
\includegraphics[width=50mm]{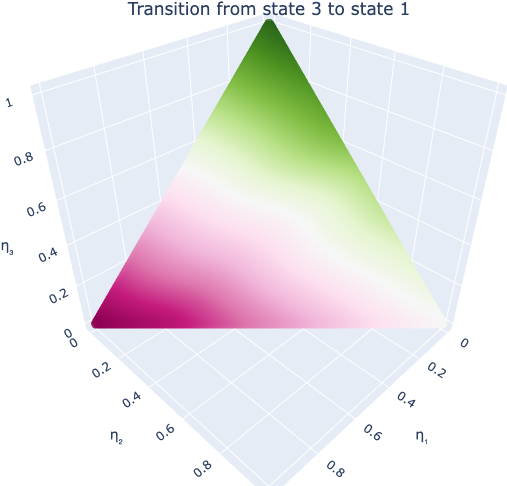} &
\includegraphics[width=50mm]{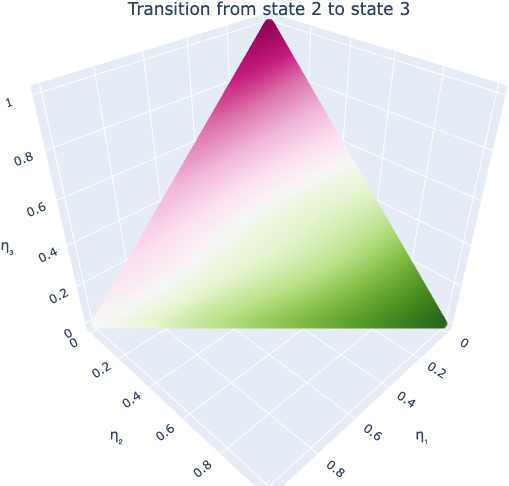} & 
\includegraphics[width=50mm]{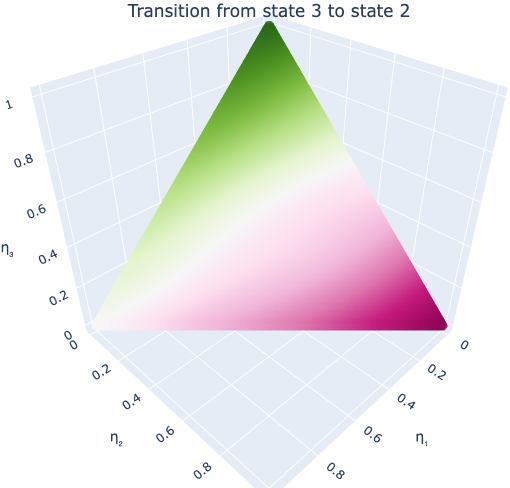} \\
\end{tabular}
\begin{tabular}{c}
\includegraphics[width=150mm]{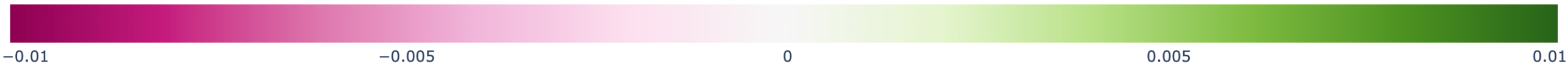}
\end{tabular}
\caption{We compare $\Gamma_0$ and $\bar\Gamma$ by the difference in rates $\gamma^*_y(x,\Delta_x U_0(\cdot,\eta)) - \gamma^*_y(x,\Delta_x \bar u)$ in the case $d=3$.}
\label{fig:rate_diff_3}
\end{figure}

In Figure \ref{fig:concentration_diff}, we note that the difference in rates leads to a greater concentration of the stationary distribution $\scrL_\iy^n$ around the center of the simplex $\calP([2])$ for $\Gamma_0$ than for $\bar\Gamma$. We computed the stationary distribution according to the formula for a birth-death process. Specifically, \cite[Part~III, Chapter~14]{MR2906848} lists the formula for the stationary distribution of a birth-death process with specified birth and death rates for every state; we then consider ``births" as when the empirical distribution increases its proportion in state $1$ and ``deaths" when it increases its proportion in state $2$.

\begin{figure}
\begin{tabular}{cc}
\includegraphics[width=78mm]{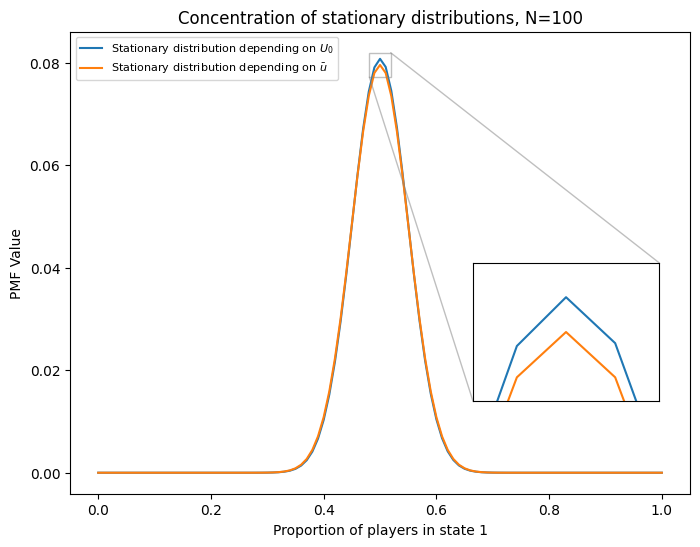}  & \includegraphics[width=78mm]{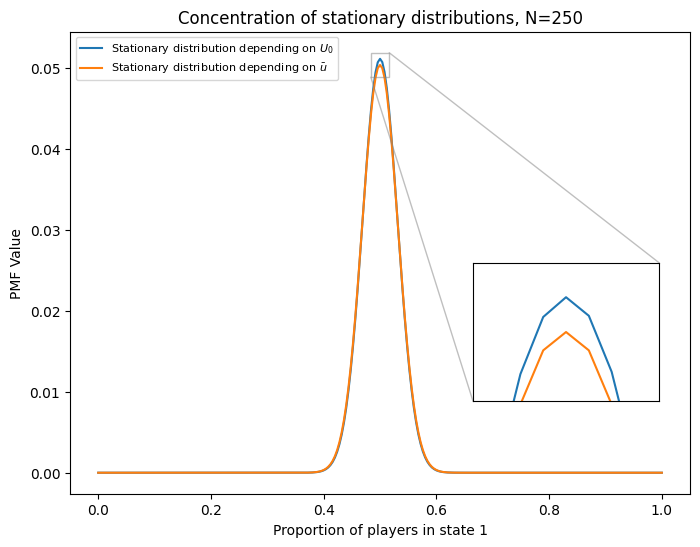} \\
\includegraphics[width=78mm]{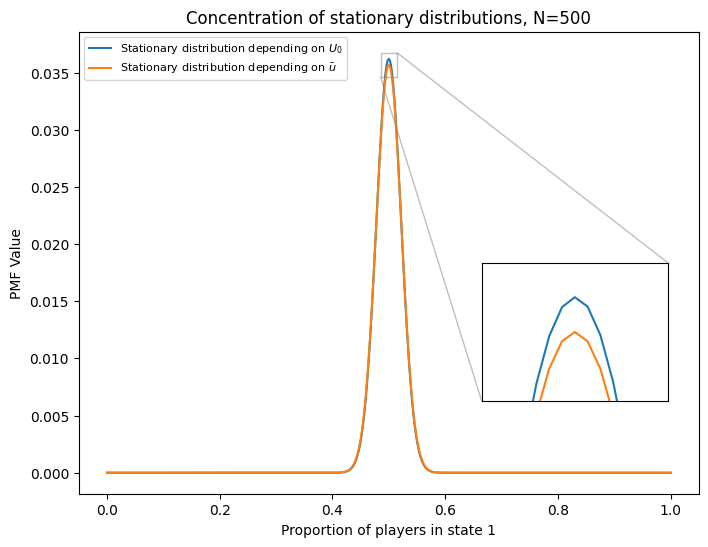} & 
\includegraphics[width=78mm]{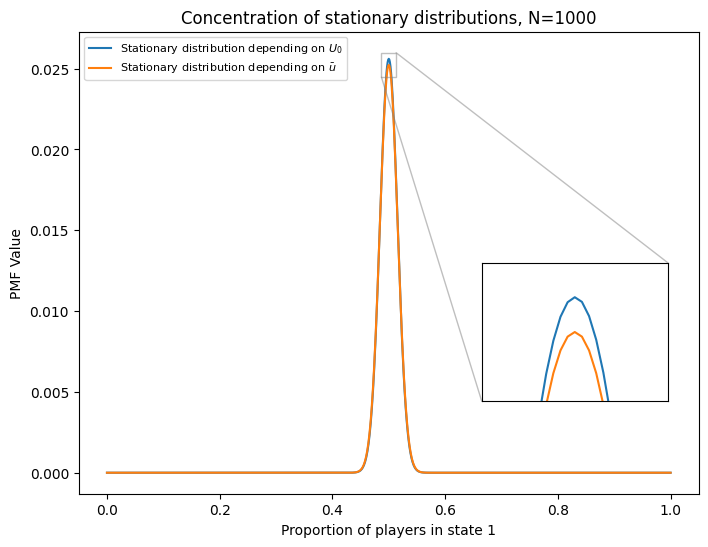} 
\end{tabular}
\caption{Due to the elevated rate for players of $\Gamma_0$ near the boundary of $\calP([2])$, we see a greater concentration of mass toward the center of the simplex, relative to $\bar\Gamma$. }
\label{fig:concentration_diff}
\end{figure}

Given the stationary distributions for $\mu^{\Gamma_0}$ and $\mu^{\bar\Gamma}$, we can compute the exact cost to a single player in each $\eps$-Nash equilibria. We write:
\[
\vr^{n,U_0} := J^n(\Gamma_0), \qquad \vr^{n, \bar u} := J^n(\bar\Gamma).
\] In Figure \ref{fig:real_cost}, we represent $\vr^{n, U_0}$ and $\vr^{n, \bar u}$ as $n$ increases. We see that $\vr^{n, U_0}$ is slightly smaller and so there is a marginal advantage to playing the profile $\Gamma_0$ associated with the master equation.

\begin{figure}
\begin{tabular}{cc}
\includegraphics[width=81mm]{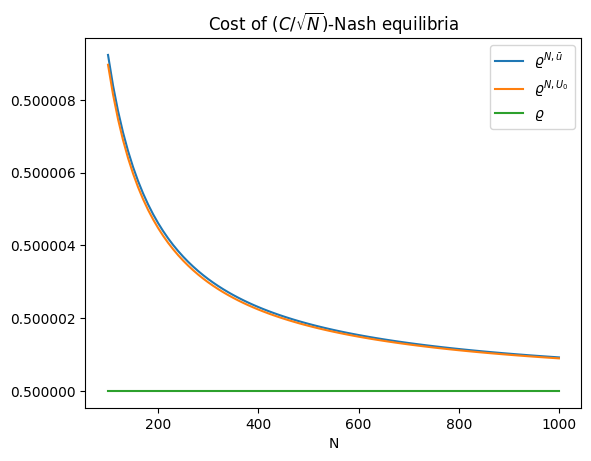}  & \includegraphics[width=81mm]{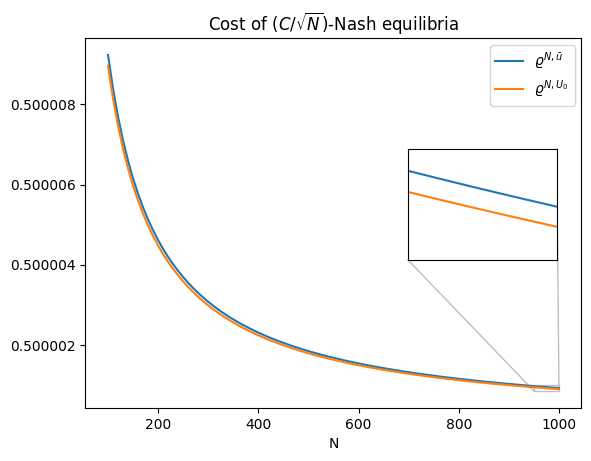} \\ 
\end{tabular}
\caption{The left graph shows the convergence of each realized cost to $\vr$ the optimal cost in the mean field limit. The right graph highlights the fact that $\vr^{n, U_0} < \vr^{n, \bar u}$.}
\label{fig:real_cost}
\end{figure}

\begin{figure}
    \centering
    \includegraphics[width=83mm]{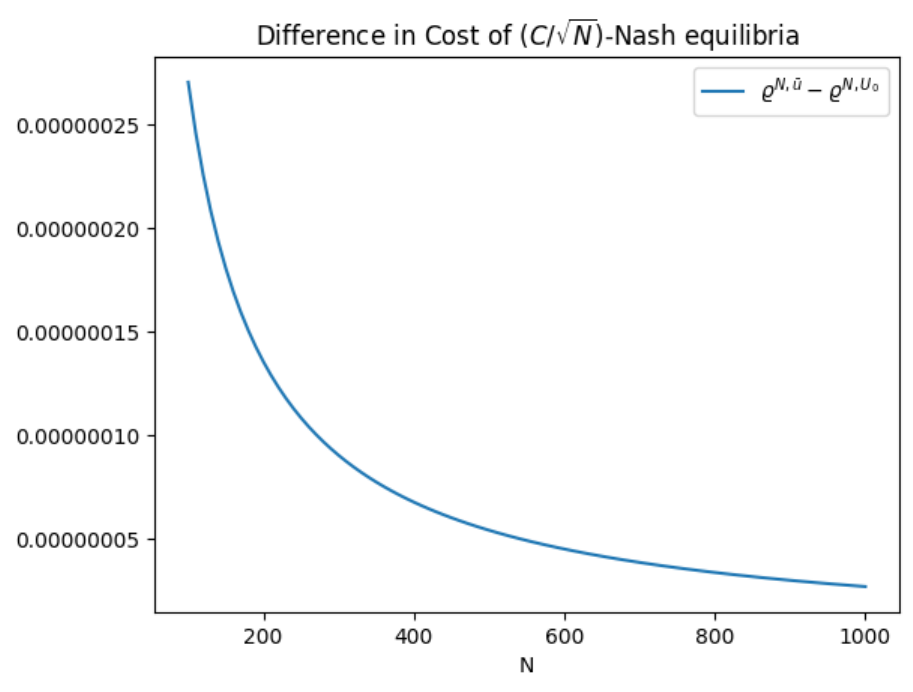}
    \caption{The difference in realized value $\vr^{n, \bar u} - \vr^{n, U_0}$.}
    \label{fig:rho_diff}
\end{figure}

\section{Large Deviations Numerics}\label{sec:ld_numerics}

In what follows, we provide an explicit formula for the good rate function of $\scrL^n_\iy$ when $d=2$. We do this by introducing natural choices for functions $r,\bar r$ that satisfy an asymptotic property sufficient for proving \Cref{thm:special_case}.


As noted briefly in \Cref{sec:numerics}, when $d=2$ we can interpret the empirical measures under $\bar\Gamma$ and $\Gamma_0$ as birth-death processes where movement into state $1$ is interpreted as a birth and movement into state $2$ is interpreted as a death. For $\bar\Gamma$ then, we can write the birth and death rates as the functions:
\begin{align}\label{eqn:bar_b_d}
    \begin{split}
        &\bar b^{(n)}(k/n) = \gamma^*_1(2,\Delta_2 \bar u) n (1-k/n), \\ 
        &\bar d^{(n)}(k/n) = \gamma^*_2(1,\Delta_1 \bar u) n (k/n). \\ 
    \end{split}
\end{align} Similarly, for $\Gamma_0$:
\begin{align}
    \begin{split}
        & b^{(n)}(k/n) = \gamma^*_1(2,\Delta_2 U_0(\cdot, (k/(n-1), 1-k/(n-1)))) n (1-k/n), \\ 
        & d^{(n)}(k/n) = \gamma^*_2(1,\Delta_1 U_0(\cdot, (k/(n-1), 1-k/(n-1)))) n (k/n). \\ 
    \end{split}
\end{align}

We now propose two functions $\bar r, r_0:[0,1]\to\R$ that, for the case $d=2$ only, will allow us to define a simpler form for the rate function than in \Cref{thm:large_deviations}. Set:
\begin{align}
    \begin{split}
        &\bar r(\lam) := \Big(\tfrac{\gamma^*_2(1,\Delta_1 \bar u)}{\gamma^*_1(2,\Delta_2 \bar u)} \Big) \frac{\lam}{1-\lam}, \\
        &r_0(\lam) := \Big(\tfrac{\gamma^*_2(1,\Delta_1  U_0(\cdot, (\lam, 1-\lam)))}{\gamma^*_1(2,\Delta_2 U_0 (\cdot, (\lam , 1-\lam)))} \Big) \frac{\lam}{1-\lam}. \\
    \end{split}
\end{align}

\begin{lemma}\label{lem:r_regular}
    The functions $\log(\bar r), \log(r_0) \in L^1([0,1])$; that is, the log of each function is absolutely integrable. Moreover, 
    \begin{align}\label{eqn:suff_lim}
    \lim_{n\to\iy}\frac{1}{n} \sum_{k=1}^{n-1} \Big(\log\Big(\frac{\bar d^{(n)} ((k+1)/n)}{\bar b^{(n)} (k/n))}\Big) - \log(\bar r(k/n) )\Big) = 0,
\end{align} and $r_0$ satisfies the same limit with $d^{(n)}$ and $b^{(n)}$. 
\end{lemma}
\begin{proof}
    Recall that the rates are bounded in a compact set away from zero. Therefore, using that $\log(ab) = \log(a) + \log(b)$,  for any $r \in \{\bar r, r_0\}$, 
    \[
    \int_0^1 |\log r(\lam)| d\lam \leq C + \int_0^1 \Big|\log\Big(\tfrac{\lam}{1-\lam}\Big) \Big| d\lam = C + \log(4) < \iy.
    \] We check \eqref{eqn:suff_lim} for $r= \bar r$. We now combine the $\log$'s in \eqref{eqn:suff_lim}, use \eqref{eqn:bar_b_d}, and use that the product is telescoping to get:
    \begin{align}
        \lim_{n\to\iy} \frac{1}{n} \sum_{k=1}^{n-1} \log \Big( \frac{k+1}{n-k} \cdot \frac{n-k}{k} \Big) = \lim_{n\to\iy} \frac{1}{n} \log\Big(\prod_{k=1}^n(k+1)/k\Big) = \lim_{n\to\iy} \frac{\log(n+1)}{n} = 0. 
    \end{align} 
    The proof for $r_0$ is similar to that of $\bar r$, though we use at one point that $U_0$ is Lipschitz. 
\end{proof}

\begin{theorem}[Large Deviations, Special Case $d=2$]\label{thm:special_case}
    Recall from \Cref{thm:large_deviations} that $\scrL^n_\iy \in \calP(\calP^n([d]))$ is the stationary distribution of the empirical distribution $\kappa^n \in \{\mu^{\bar\Gamma}, \mu^{\Gamma_0}\}$ and let the $r \in \{\bar r,  r_0\}$ be the corresponding function from \Cref{lem:r_regular}. When $d=2$, $\scrL^n_\iy$ satisfies a large deviation principle with a simplified good rate function $s$  given by:
    \[
    s(\eta) = \tilde s(\eta) - \tilde s_{min},
    \] where:
    \begin{align}\label{eqn:Ys}
        \begin{split}
            &\tilde s(\eta) := \int_0^{\eta_1} \log(m(\lam)) d\lam, \\
            &\tilde s_{min} := \min_{\eta \in \calP([2])} \tilde s(\eta). 
        \end{split}
    \end{align} And in the case of $\kappa^n = \mu^{\bar\Gamma}$, we have explicit formulas for $\tilde s$ and $\tilde s_{min}$:
    \begin{align}\label{eqn:explicit_Y}
    \begin{split}
    & \tilde s(\eta) = \eta_1 \log\Big(\tfrac{\gamma^*_2(1,\Delta_1 \bar u)}{\gamma^*_1(2,\Delta_2 \bar u)}\Big) + \log(\eta_2) + \eta_1 \log(\eta_1/\eta_2),\\
    &\tilde s_{min} = \log\Big(\tfrac{\gamma^*_2(1,\Delta_1 \bar u)}{\gamma^*_1(2,\Delta_2 \bar u)}\Big) - \log\Big(\tfrac{\gamma^*_2(1,\Delta_1 \bar u)}{\gamma^*_1(2,\Delta_2 \bar u)} + 1\Big).
    \end{split}
    \end{align}
\end{theorem}

\begin{proof}
    The purpose of Lemma \ref{lem:r_regular} is to satisfy the assumptions from \cite[(3.3a), (3.3b)]{MR1668135}. As a consequence, \cite[Theorem~3.1]{MR1668135} asserts that $s$ is the good rate function as claimed. It remains to show that, in the case of $\bar r$, we have an explicit form of $s$. 

    That is, we can integrate, using that $\log(ab) = \log(a) + \log(b)$, integration by parts, and the fact that $\eta_1 = 1-\eta_2$,
    \begin{align*}
    \tilde s(\eta) &= \int_0^{\eta_1} \log\Big( \tfrac{\gamma^*_2(1,\Delta_1 \bar u)}{\gamma^*_1(2,\Delta_2 \bar u)} \tfrac{\lam}{1-\lam}\Big) d\lam  \\
    &= \eta_1 \log\Big(\tfrac{\gamma^*_2(1,\Delta_1 \bar u)}{\gamma^*_1(2,\Delta_2 \bar u)} \Big) + \int_0^{\eta_1} \log\Big( \tfrac{\lam}{1-\lam}\Big) d\lam \\
    &= \eta_1 \log\Big(\tfrac{\gamma^*_2(1,\Delta_1 \bar u)}{\gamma^*_1(2,\Delta_2 \bar u)} \Big) + \log(1-\lam) + \lam\log\Big(\tfrac{\lam}{1-\lam}\Big) \Big|_0^{\eta_1} \\
    &= \eta_1 \log\Big(\tfrac{\gamma^*_2(1,\Delta_1 \bar u)}{\gamma^*_1(2,\Delta_2 \bar u)} \Big) + \log(\eta_2) + \eta_1 \log(\eta_1/\eta_2).
    \end{align*} Taking the derivative,
    \[
    \frac{\partial}{\partial \eta_1} \tilde s(\eta) = \log\Big( \tfrac{\gamma^*_2(1,\Delta_1 \bar u)}{\gamma^*_1(2,\Delta_2 \bar u)} \tfrac{\eta_1}{1-\eta_1}\Big).
    \] Therefore, we can find that $\tilde s$ has a critical point at:
    \[
    \eta =(\eta_1,\eta_2), \quad \text{ where } \quad \eta_1 = \Big(\tfrac{\gamma^*_2(1,\Delta_1 \bar u)}{\gamma^*_1(2,\Delta_2 \bar u)} + 1\Big)^{-1}, \quad \eta_2 = 1-\eta_1.
    \] By the first derivative test, this point is the minimum of $\tilde s$. Plugging this point into $\tilde s$ and by some arithmetic, 
    \[
    \tilde s_{min} = \log\Big(\tfrac{\gamma^*_2(1,\Delta_1 \bar u)}{\gamma^*_1(2,\Delta_2 \bar u)}\Big) - \log\Big(\tfrac{\gamma^*_2(1,\Delta_1 \bar u)}{\gamma^*_1(2,\Delta_2 \bar u)} + 1\Big).
    \]
\end{proof}

In the numerical case we visited in the prior section, this means that the good rate function for the $\bar \Gamma$-system is given by:
\[
s(\eta) = \log(\eta_2) + \eta_1 \log(\eta_1/\eta_2) + \log(2).
\] Unfortunately, the good rate function associated with the $\Gamma_0$-system is not explicitly solvable since the integral contains $U_0$. Nonetheless, since we approximately solved for $U_0$ using \Cref{alg:ergodic_DGM}, we can use Monte Carlo integration to approximate $\tilde s$ in that case and thus obtain the good rate function. In Figure \ref{fig:rate_fns}, we compare the good rate functions for $\bar \Gamma$ and $\Gamma_0$.

\begin{figure}
\begin{tabular}{cc}
\includegraphics[width=70mm]{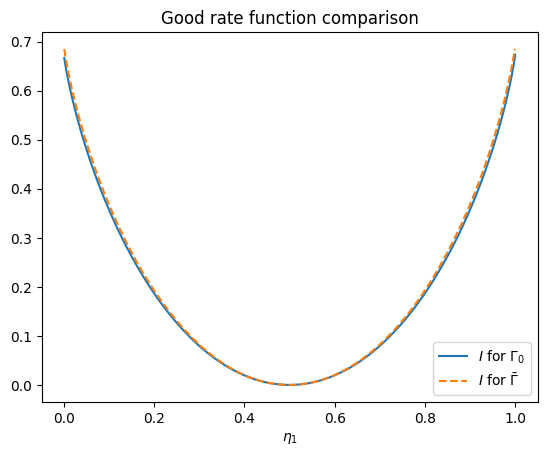}  & \includegraphics[width=73mm]{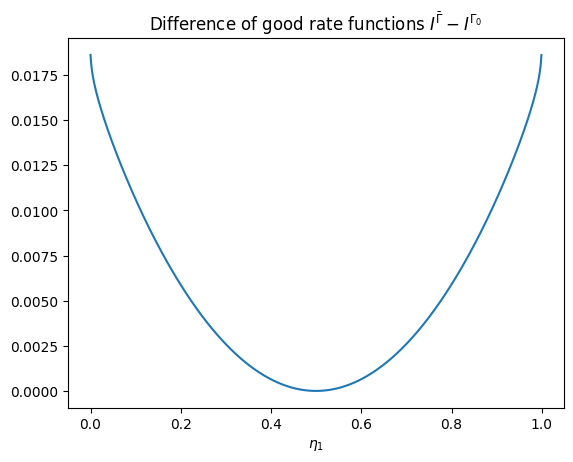} \\ 
\end{tabular}
\caption{On the left-hand side, we compare the good rate functions associated with the invariant measures for $\mu^{\Gamma_0}$ and $\mu^{\bar\Gamma}$. \Cref{thm:special_case} gives the rate function associated with $\mu^{\bar\Gamma}$ while \Cref{alg:ergodic_DGM} and Monte Carlo integration obtain the rate function for $\mu^{\Gamma_0}$. On the right, we highlight the difference between the rate functions. We note that the graph on the right-hand side implies that playing the profile $\bar\Gamma$ is associated with larger deviations as compared with $\Gamma_0$, and that the distinction is most pronounced toward the boundary of $\calP([2])$.}
\label{fig:rate_fns}
\end{figure}

\section{Appendix: Proof of Lemma \ref{lem:freedom}}

    Note that $(\bX^{\bbeta}_t)_{t\ge 0}$ is a Markov process. Denote its distributional law at time $t$ by $(\calL(\bX^{\bbeta}_t)_{\bx}=\PP(\bX^{\bbeta}_t=\bx))_{\bx\in[d]^n}$. Then, it satisfies the Kolmogorov equation:
    \[
    \frac{d}{dt} \calL(\bX^{\bbeta}_t)_{\bx} = \sum_{\by \in [d]^n} \calL(\bX^{\bbeta}_t)_{\by} Q^{\bbeta}_{\by \bx} (t, \bX^{\bbeta}_t), \quad (t,\bx) \in \R_+ \times [d]^n,
    \] where $Q^{\bbeta}_{\by \bx} (t, \bX^{\bbeta}_t)$ is the rate matrix:
    \[
    Q^{\bbeta}_{\by \bx} (t, \bX^{\bbeta}_t) :=
    \begin{cases}
        0, & \by \neq \bx + e_y - e_x, \text{ for any } y\neq x, \\
        \sum_{i\in [n]} \beta^i_{xy}(t, \bX^{\bbeta}_t) \1_{\{X^{\bbeta, i}_t = x\}}, & \by = \bx + e_y - e_x, \text{ for some } y\neq x, \\
        \sum_{i\in [n]} \beta^i_{xx}(t, \bX^{\bbeta}_t) \1_{\{X^{\bbeta, i}_t = x\}}, & \by = \bx.
    \end{cases}
    \] 
    \blue{We note that \(\boldsymbol{\beta} \in \mathcal{A}^n\). According to Assumption (A1), all rates are uniformly bounded away from zero. This condition is sufficient for applying \cite[Lemma~3.3]{CZ2022}, which allows for the rates to be time-dependent. Hence, } 
\begin{align}\label{eqn:exp_ergodic}
    |\calL(\bX^{\bbeta}_t \mid \bX^{\bbeta}_0 \sim \pi_0) - \calL(\bX^{\bbeta}_t \mid \bX^{\bbeta}_0 \sim \pi_1)| \leq K_n e^{-k_n t} |\pi_0 - \pi_1|,
    \end{align} where $K_n, k_n$ are positive constants depending only on $\A$ and $n$. 
    \blue{To illustrate this intuitively, consider two time-inhomogeneous Markov chains, each starting from a different initial state. Because the transition rates are bounded away from zero, these chains will meet exponentially fast. Once they meet, they can be coupled to evolve together. Now, the dependence of the exponential rate above on \( n \) is not an issue, as \( g \) is bounded and we are considering a long-time average. This is what we show in the sequel.
}
\begin{align}\label{eqn:irrelevance_original}
        \begin{split}
            \limsup_{T\to\iy} \frac{1}{T} \E_{\pi_0} &\int_0^T g(\bX^{\bbeta}_t) dt = \limsup_{T\to\iy} \frac{1}{T} \int_0^T \sum_{\bx\in [d]^n} g(\bx) \calL(\bX^{\bbeta}_t \mid \bX^{\bbeta}_0 \sim \pi_0 )_{\bx} dt \\ 
            &= \limsup_{T\to\iy} \frac{1}{T} \int_0^T \sum_{\bx\in [d]^n} g(\bx) \big[\calL(\bX^{\bbeta}_t \mid \bX^{\bbeta}_0 \sim \pi_0 )_{\bx} - \calL(\bX^{\bbeta}_t \mid \bX^{\bbeta}_0 \sim \pi_1 )_{\bx} \big] dt \\
            &\quad + \limsup_{T\to\iy} \frac{1}{T} \int_0^T \sum_{\bx\in [d]^n} g(\bx)  \calL(\bX^{\bbeta}_t \mid \bX^{\bbeta}_0 \sim \pi_1 )_{\bx}  dt \\
            &= \limsup_{T\to\iy} \frac{1}{T} \int_0^T \sum_{\bx\in [d]^n} g(\bx) \big[\calL(\bX^{\bbeta}_t \mid \bX^{\bbeta}_0 \sim \pi_0 )_{\bx} - \calL(\bX^{\bbeta}_t \mid \bX^{\bbeta}_0 \sim \pi_1 )_{\bx} \big] dt \\
            &\quad +\limsup_{T\to\iy} \frac{1}{T} \E_{\pi_1} \int_0^T g(\bX^{\bbeta}_t) dt.
        \end{split}
    \end{align}
But then by triangle inequality, the boundedness of $g$, and finally \eqref{eqn:exp_ergodic},
\begin{align}
    \begin{split}
        \Big| \limsup_{T\to\iy} \frac{1}{T} \int_0^T &\sum_{\bx\in [d]^n} g(\bx) \big[\calL(\bX^{\bbeta}_t \mid \bX^{\bbeta}_0 \sim \pi_0 )_{\bx} - \calL(\bX^{\bbeta}_t \mid \bX^{\bbeta}_0 \sim \pi_1 )_{\bx} \big] dt \Big| \\
        &\leq \limsup_{T\to\iy} \frac{1}{T} \int_0^T \sum_{\bx\in [d]^n} g(\bx) \big|\calL(\bX^{\bbeta}_t \mid \bX^{\bbeta}_0 \sim \pi_0 )_{\bx} - \calL(\bX^{\bbeta}_t \mid \bX^{\bbeta}_0 \sim \pi_1 )_{\bx} \big| dt \\ 
        &\leq C_g \limsup_{T\to\iy} \frac{1}{T} \int_0^T \sum_{\bx\in [d]^n}  \big|\calL(\bX^{\bbeta}_t \mid \bX^{\bbeta}_0 \sim \pi_0 )_{\bx} - \calL(\bX^{\bbeta}_t \mid \bX^{\bbeta}_0 \sim \pi_1 )_{\bx} \big| dt \\
        &\leq C_g \limsup_{T\to\iy} \frac{1}{T} \int_0^T K_n e^{-k_n t}|\pi_0 - \pi_1| dt \\
        &=0.
    \end{split}
\end{align} 
Using this fact in \eqref{eqn:irrelevance_original} concludes the proof. $\square$

%
%
%

 \blue{\section*{Acknowledgments.}
We thank the referees and associate editor for their time, energy, and valuable comments
which greatly improved the quality of our paper.}

\footnotesize
\bibliographystyle{abbrv} 
\bibliography{mor} 

\end{document}